\documentclass
{amsart}

\usepackage[dvipdfmx]{graphicx}
\usepackage{amssymb,amsmath,amsthm,amscd,amsfonts}

\newtheorem{theorem}{Theorem}[section]
\newtheorem{proposition}[theorem]{Proposition}

\newtheorem{lemma}[theorem]{Lemma}

\theoremstyle{definition}
\newtheorem{definition}[theorem]{Definition}

\newtheorem{remark}[theorem]{Remark}
\newtheorem{example}[theorem]{Example}
\newtheorem{notation}[theorem]{Notation}
\newtheorem*{RPL*}{Relation to previous literature}
\newtheorem*{OTP*}{Organization of this paper}
\newtheorem*{Ack*}{Acknowledgements}

\begin{document}


\title[]
{Ricci-mean curvature flows in gradient shrinking Ricci solitons}


\author{Hikaru Yamamoto}
\address{Graduate School of Mathematical Sciences, The University of Tokyo, 3-8-1 Komaba Meguro-ku Tokyo 153-8914, Japan}
\email{yamamoto@ms.u-tokyo.ac.jp}




\begin{abstract} 
Huisken \cite{Huisken} studied asymptotic behavior of a mean curvature flow in a Euclidean space when it develops a singularity of type I, 
and proved that its rescaled flow converges to a self-shrinker in the Euclidean space. 
In this paper, we generalize this result for a Ricci-mean curvature flow moving along a Ricci flow constructed from a gradient shrinking Ricci soliton. 
\end{abstract}


\keywords{mean curvature flow, Ricci flow, self-similar solution, gradient soliton}


\subjclass[2010]{53C42, 53C44}


\thanks{This work was supported by Grant-in-Aid for JSPS Fellows Grant Number 13J06407 and the Program for Leading Graduate Schools, MEXT, Japan. }




\maketitle

\section{Introduction}\label{Intro}
Let $M$ and $N$ be manifolds with dimension $m$ and $n$ respectively, satisfying $m\leq n$. 
Let $g=(\, g_{t}\,;\, t\in[0,T_{1})\,)$ be a smooth 1-parameter family of Riemannian metrics on $N$ and 
$F:M\times [0,T_{2}) \to N$ be a smooth 1-parameter family of immersions with $T_{2}\leq T_{1}$, that is, 
$F_{t}:M\to N$ defined by $F_{t}(\,\cdot\,):=F(\cdot,t)$ is an immersion map. 
We say that the pair of $g$ and $F$ is a solution of the {\it Ricci-mean curvature flow} if it satisfies 
the following coupled equation of the Ricci flow and the mean curvature flow: 
\begin{subequations}
\begin{align}
&\frac{\partial g_{t}}{\partial t}=-2\mathrm{Ric}(g_{t})\label{evoeq21}\\
&\frac{\partial F_{t}}{\partial t}=H(F_{t}), \label{evoeq22}
\end{align}
\end{subequations}
where $H(F_{t})$ denotes the mean curvature vector field of $F_{t}:M\to N$ computed by the ambient Riemannian metric $g_{t}$ at the time $t$. 
Note that this coupling is partial, that is, the Ricci flow equation (\ref{evoeq21}) does not depend on $F$. 
It is clear that a Ricci-mean curvature flow is a mean curvature flow 
when the ambient Riemannian manifold $(N,g_{0})$ is Ricci flat (especially $(\mathbb{R}^n,g_{\mathrm{st}})$). 

Huisken~\cite{Huisken} studied asymptotic behavior of a mean curvature flow in a Euclidean space when it develops a singularity of type I, 
and proved that its rescaled flow converges to a self-shrinker in the Euclidean space. 
In this paper, we generalize this result to a Ricci-mean curvature flow moving along a Ricci flow constructed from a gradient shrinking Ricci soliton. 
Before stating our main results, we review the definition of self-similar solutions in $\mathbb{R}^n$ and the results due to Huisken~\cite{Huisken}. 

On $\mathbb{R}^n$, we naturally identify a point $x=(x^{1},\dots,x^{n})\in\mathbb{R}^n$ with a tangent vector $\overrightarrow{x}\in T_{x}\mathbb{R}^n$ by 
\[\overrightarrow{x}:=x^{1}\frac{\partial}{\partial x^{1}}+\dots+x^{n}\frac{\partial}{\partial x^{n}}. \]
For an immersion map $F:M\to\mathbb{R}^n$, 
we have a section $\overrightarrow{F}\in\Gamma(M,F^{*}(T\mathbb{R}^n))$ defined by $\overrightarrow{F}(p):=\overrightarrow{F(p)}$ for all $p\in M$. 
Then $F:M\to\mathbb{R}^n$ is called a self-similar solution if it satisfies 
\begin{align}\label{selsolinR}
H(F)=\frac{\lambda}{2} {\overrightarrow{F}}^{\bot}
\end{align}
for some constant $\lambda\in\mathbb{R}$, 
where $\bot$ denotes the projection onto the normal bundle of $M$.
A self-similar solution is called a self-expander, steady or self-shrinker when $\lambda>0$, $\lambda=0$ or $\lambda<0$ respectively. 

Let $M$ be an $m$-dimensional compact manifold and 
$F:M\times[0,T)\to \mathbb{R}^n$ be a mean curvature flow with the maximal time $T<\infty$, that is, 
we can not extend the flow over the time $T$. 
Further assume that $F$ satisfies the following two conditions (A1) and (B1).  
\begin{enumerate}
\item[(A1)] The norm of the second fundamental form of $F_{t}$ (denoted by $A(F_{t})$) satisfies \[\limsup_{t\to T}\bigl(\sqrt{T-t}\max_{M}|A(F_{t})|\bigr)<\infty. \]
\item[(B1)] There exists a point $p_{0}$ in $M$ such that $F_{t}(p_{0})\to \mathrm{O}\in \mathbb{R}^n$ as $t\to T$. 
\end{enumerate}
If a mean curvature flow satisfies (A1), then we say that it develops a singularity of type I, 
and for the remaining case we say that it develops a singularity of type II. 
The condition (B1) guarantees that there exists at least one point in $M$ such that its image of the rescaled flow remains in a bounded region in $\mathbb{R}^n$, 
thus the limiting submanifold is nonempty. 
In \cite{Huisken}, it is also assumed that $|A(F_{t})|(p_{0})\to \infty$ as $t\to T$ for $p_{0}$ given in (B1). 
However, this assumption is not necessary to prove Theorem~\ref{Huisken} introduced below. 

For each $t\in(-\infty,T)$, let $\Phi_{t}:\mathbb{R}^{n}\to\mathbb{R}^n$ be a diffeomorphism of $\mathbb{R}^n$ defined by 
\[\Phi_{t}(x):=\frac{1}{\sqrt{T-t}}x. \]
Define the rescaled flow $\tilde{F}:M\times [-\log T, \infty)\to \mathbb{R}^n$ by 
\[\tilde{F}_{s}:=\Phi_{t}\circ F_{t}\quad\mathrm{with}\quad s=-\log(T-t). \]
Then it satisfies the {\it normalized mean curvature flow} equation: 
\[\frac{\partial \tilde{F}_{s}}{\partial s}=H(\tilde{F_{s}})+\frac{1}{2}\overrightarrow{\tilde{F_{s}}}. \]
Huisken proved the following (cf. Proposition 3.4 and Theorem 3.5 in \cite{Huisken}). 

\begin{theorem}\label{Huisken}
Under the assumptions (A1) and (B1), 
for each sequence $s_{j}\to \infty$, there exists a subsequence $s_{j_{k}}$ such that the sequence of 
immersed submanifolds $\tilde{M}_{s_{j_{k}}}:=\tilde{F}_{s_{j_{k}}}(M)$ 
converges smoothly to an immersed nonempty limiting submanifold $\tilde{M}_{\infty}\subset \mathbb{R}^n$, 
and $\tilde{M}_{\infty}$ is a self-shrinker with $\lambda=-1$ in (\ref{selsolinR}). 
\end{theorem}

By this theorem, a self-shrinker can be considered as a local model of a singularity of type I for a mean curvature flow in $\mathbb{R}^n$. 

On the other hand, there is also the notion of type I singularity for a Ricci flow $g=(\, g_{t}\,;\, t\in[0,T)\,)$ on a manifold $N$. 
Assume that $T<\infty$ is the maximal time. 
We say that $g$ forms a singularity of type I if 
\[\limsup_{t\to T}\bigl((T-t)\sup_{N}|\mathrm{Rm}(g_{t})|\bigr)<\infty, \]
where $\mathrm{Rm}(g_{t})$ denotes the Riemannian curvature tensor of $g_{t}$. 
In the Ricci flow case, a gradient shrinking Ricci soliton can be considered as a local model of a singularity of type I (cf. \cite{EndersMullerTopping, Naber, Sesum}). 
Actually, from a gradient shrinking Ricci soliton, we can construct a Ricci flow which develops a singularity of type I by the action of diffeomorphisms and scaling. 
In this paper, we consider a Ricci-mean curvature flow along this Ricci flow, and 
assume that the mean curvature flow and the Ricci flow develop singularities at the same time. 
Then we prove the convergence of the rescaled flow to a self-shrinker in the gradient shrinking Ricci soliton under the type I assumption 
(more precisely, under the assumption (A2) when $N$ is compact, and (A2) and (B2) when $N$ is non-compact). 
The precise settings and main results are the following. 

Let $(N,\tilde{g},\tilde{f})$ be an $n$-dimensional complete gradient shrinking Ricci soliton with
\begin{align}
\mathrm{Ric}(\tilde{g})+\mathop{\mathrm{Hess}}\tilde{f}-\frac{1}{2} \tilde{g}=0. \label{shrisol}
\end{align}
As Hamilton's proof of Theorem 20.1 in \cite{Hamilton4}, one can easily see that $R(\tilde{g})+|\nabla\tilde{f}|^2-\tilde{f}$ is a constant, 
where $R(\tilde{g})$ denotes the scalar curvature of $\tilde{g}$. 
Hence by adding some constant to $\tilde{f}$ if necessary, we may assume that the potential function $\tilde{f}$ also satisfies 
\begin{align}
R(\tilde{g})+|\nabla\tilde{f}|^2-\tilde{f}=0. \label{addshrisol}
\end{align}

For an immersion $F:M\to N$, we get a section $(\nabla\tilde{f})\circ F \in \Gamma(M,F^{*}(TN))$, 
and we usually omit the symbol $\circ F$, for short. 
\begin{definition}\label{selsolingra}
If an immersion map $F:M\to N$ satisfies 
\begin{align}\label{selsolinshri}
H(F)=\lambda{\nabla\tilde{f}}^{\bot}
\end{align}
for some constant $\lambda\in\mathbb{R}$, 
we call it a self-similar solution. 
A self-similar solution is called a self-expander, steady or self-shrinker when $\lambda>0$, $\lambda=0$ or $\lambda<0$ respectively. 
\end{definition}

\begin{definition}\label{defofnormalizedmcf}
If a 1-parameter family of immersions $\tilde{F}:M\times[0,S)\to N$ satisfies 
\begin{align}\label{normalizedmcf}
\frac{\partial \tilde{F}_{s}}{\partial s}=H(\tilde{F}_{s})+\nabla\tilde{f}, 
\end{align} 
we call it a normalized mean curvature flow. 
\end{definition}

Fix a positive time $0<T<\infty$. 
Let $\{\Phi_{t}:N\to N\}_{t\in(-\infty,T)}$ be the 1-parameter family of diffeomorphisms with $\Phi_{0}=\mathrm{id}_{N}$ generated by 
the time-dependent vector field $V_{t}:=\frac{1}{T-t}\nabla \tilde{f}$. 
For $t\in(-\infty,T)$, define 
\[g_{t}:=(T-t)\Phi_{t}^{*}\tilde{g}\quad\mathrm{and}\quad f_{t}:=\Phi_{t}^{*}\tilde{f}. \]
Then $g_{t}$ satisfies the Ricci flow equation (\ref{evoeq21}). 
Assume that $F:M\times [0,T)\to N$ is a solution of Ricci-mean curvature flow (\ref{evoeq22}) along this Ricci flow $g=(\, g_{t}\,;\, t\in[0,T)\,)$. 
We consider the following two conditions (A2) and (B2). 
\begin{enumerate}
\item[(A2)] The norm of the second fundamental form of $F_{t}$ (denoted by $A(F_{t})$) satisfies \[\limsup_{t\to T}\bigl(\sqrt{T-t}\max_{M}|A(F_{t})|\bigr)<\infty. \]
\item[(B2)]  There exists a point $p_{0}\in M$ such that when $t\to T$ 
\[\ell_{F_{t}(p_{0}),t}\to f \quad\mathrm{pointwise~on~}N\times[0,T), \]
where $f:N\times[0,T)\to\mathbb{R}$ is a function on $N\times[0,T)$ defined above 
and $\ell_{\ast,\bullet}:N\times [0,\bullet)\to\mathbb{R}$ is the reduced distance based at $(\ast,\bullet)$. 
\end{enumerate}

\begin{remark}
The condition (A2) corresponds to (A1). 
In (A2), note that $A(F_{t})$ and its norm $|A(F_{t})|$ are computed by the ambient metric $g_{t}$ at each time $t$. 
In this paper, we do not assume that 
\begin{align}\label{diverofA}
\limsup_{t\to T}\sup_{M}|A(F_{t})|=\infty. 
\end{align}
If $F:M\times [0,T)\to N$ satisfies (\ref{diverofA}) and (A2), we say that $F$ develops a singularity of type I. 
Hence (A2) is slightly weaker than the condition that $F$ develops a singularity of type I. 
Especially, non-singular case (that is, $\limsup_{t\to T}\sup_{M}|A(F_{t})|<\infty$) is contained in (A2). 
\end{remark}

\begin{remark}
The condition (B2) corresponds to (B1). 
In (B2), $\ell_{F_{t}(p_{0}),t}$ is the reduced distance for the Ricci flow $g$ based at $(F_{t}(p_{0}),t)$ introduced by Perelman. 
Here we explain this briefly. 
Let $(N,g_{t})$ be a Ricci flow on $[0,T)$. 
For any curve $\gamma:[t_{1},t_{2}]\to N$ with $0\leq t_{1}<t_{2}<T$, we define the $\mathcal{L}$-length of $\gamma$ by 
\[\mathcal{L}(\gamma):=\int_{t_{1}}^{t_{2}}\sqrt{t_{2}-t}\bigl( R(g_{t})+|\dot{\gamma}|^2 \bigr)dt, \]
where $|\dot{\gamma}|$ is the norm of $\dot{\gamma}(t)$ measured by $g_{t}$. 
For a fixed point $(p_{2},t_{2})$ in the space-time $N\times (0,T)$, 
we get the reduced distance 
\[\ell_{p_{2},t_{2}}:N\times [0,t_{2})\to\mathbb{R}\]
based at $(p_{2},t_{2})$ defined by
\[\ell_{p_{2},t_{2}}(p_{1},t_{1}):=\frac{1}{2\sqrt{t_{2}-t_{1}}}\inf_{\gamma}\mathcal{L}(\gamma), \]
where the infimum is taken over all curves $\gamma:[t_{1},t_{2}]\to N$ with $\gamma(t_{1})=p_{1}$ and $\gamma(t_{2})=p_{2}$. 
In Remark~\ref{forHuisken}, we see that (B1) and (B2) are equivalent when $(N,\tilde{g},\tilde{f})$ is the Gaussian soliton $(\mathbb{R}^n,g_{\mathrm{st}},\frac{1}{4}|x|^2)$. 
\end{remark}

If $(N,\tilde{g},\tilde{f})$ is compact (resp.~non-compact), we assume that $F$ satisfies (A2) (resp. (A2) and (B2)). 
As in the Euclidean case, we consider the rescaled flow $\tilde{F}:M\times [-\log T, \infty)\to N$ defined by 
\begin{align}\label{scaledmcf2}
\tilde{F}_{s}:=\Phi_{t}\circ F_{t}\quad\mathrm{with}\quad s=-\log(T-t), 
\end{align}
and we can see that $\tilde{F}$ becomes a normalized mean curvature flow in $(N,\tilde{g},\tilde{f})$ (cf. Proposition~\ref{corresofselsol}). 
Then the main results in this paper are the following. 

\begin{theorem}\label{maincomp}
Assume that $(N,\tilde{g},\tilde{f})$ is compact. 
Let $F:M\times[0,T)\to N$ be a Ricci-mean curvature flow along the Ricci flow $(N,g_{t})$ defined by $g_{t}:=(T-t)\Phi_{t}^{*}\tilde{g}$. 
Assume that $M$ is compact and $F$ satisfies (A2). 
Let $\tilde{F}:M\times[-\log T,\infty)\to N$ be defined by (\ref{scaledmcf2}). 
Then, for any sequence $s_{1}<s_{2}<\cdots<s_{j}<\cdots \to \infty$ and points $\{x_{j}\}_{j=1}^{\infty}$ in $M$, 
there exist sub-sequences $s_{j_{k}}$ and $x_{j_{k}}$ 
such that the family of immersion maps $\tilde{F}_{s_{j_{k}}}:M\to N$ from pointed manifolds $(M,x_{j_{k}})$ converges to an immersion map 
$\tilde{F}_{\infty}:M_{\infty}\to N$ from some pointed manifold $(M_{\infty}, x_{\infty})$.  
Furthermore, $M_{\infty}$ is a complete Riemannian manifold with metric $\tilde{F}_{\infty}^{*}\tilde{g}$ 
and $\tilde{F}_{\infty}$ is a self-shrinker in $(N,\tilde{g},\tilde{f})$ with $\lambda=-1$, that is, 
$\tilde{F}_{\infty}$ satisfies 
\[H(\tilde{F}_{\infty})=-\nabla\tilde{f}^{\bot}. \]
\end{theorem}

\begin{theorem}\label{mainnoncomp}
Assume that $(N,\tilde{g},\tilde{f})$ is non-compact and satisfies the assumption in Remark~\ref{immRiem2}. 
Under the same setting in Theorem~\ref{maincomp}, assume that $M$ is compact and $F$ satisfies (A2) and (B2). 
Then, for any sequence of times $s_{j}$, the same statement as Theorem~\ref{maincomp} holds, where we fix $x_{j}:=p_{0}$ for all $j$. 
\end{theorem}

\begin{remark}\label{immRiem2}
For a complete non-compact Riemannian manifold $(N,\tilde{g})$, we assume that 
there is an isometrically embedding $\Theta:N\to \mathbb{R}^{L}$ into some higher dimensional Euclidean space with 
\[|\nabla^{p}A(\Theta)|\leq \tilde{D}_{p}<\infty \]
for some constants $\tilde{D}_{p}>0$ for all $p\geq 0$. 
Under this assumption, one can see that $(N,\tilde{g})$ must have the bounded geometry by Theorem~\ref{injofsub2} and Gauss equation (\ref{Gaussequation}) (and its iterated derivatives). 
\end{remark}

\begin{remark}
The notion of the convergence of immersions from pointed manifolds is defined in Appendix~\ref{convofsubsec} (cf. Definition~\ref{convofimmmaps}). 
Roughly speaking, it is the immersion version of the Cheeger--Gromov convergence of pointed Riemannian manifolds. 
\end{remark}

\begin{remark}\label{forHuisken}
We see that Theorem~\ref{mainnoncomp} implies Theorem~\ref{Huisken} in $\mathbb{R}^n$. 
Consider $\mathbb{R}^n$ as the Gaussian soliton with potential function $\tilde{f}(x):=\frac{1}{4}|x|^2$. 
Since $\overrightarrow{x}=2\nabla\tilde{f}(x)$, Definition~\ref{selsolingra} coincides with (\ref{selsolinR}) in $\mathbb{R}^n$. 
It is trivial that $(\mathbb{R}^n,g_{\mathrm{st}})$ satisfies the assumption in Remark~\ref{immRiem2}. 
We take $T=1$ for simplicity. Then we have 
\[\Phi_{t}(x)=\frac{1}{\sqrt{T-t}}x, \quad g_{t}\equiv g_{\mathrm{st}}, \quad f(x,t)=\frac{|x|^2}{4(T-t)}. \]
Since $g_{t}$ is the trivial Ricci flow, the condition (A1) and (A2) coincides. 
Furthermore, one can easily see that in this trivial Ricci flow Perelman's reduced distance bases at $(\ast,\bullet)$ is given by 
\[\ell_{\ast,\bullet}(x,t):=\frac{|x-\ast|^2}{4(\bullet-t)}. \]
Hence it is clear 
\[\ell_{F_{t}(p_{0}),t}\to f \quad\mathrm{pointwise~on~}\mathbb{R}^n\times[0,T)\]
when $F_{t}(p_{0})\to \mathrm{O}$ as $t\to T$, that is, the condition (B1) implies (B2). 
Conversely, under the assumption (B2) we can see that $F_{t}(p_{0})\to \mathrm{O}$ as $t\to T$ since 
\[\frac{1}{4t}|F_{t}(p_{0})|^2=\ell_{F_{t}(p_{0}),t}(\mathrm{O},0)\to f(\mathrm{O},0)=0 \]
as $t\to T(<\infty)$. 
Hence (B1) and (B2) are equivalent in $\mathbb{R}^n$, and Theorem \ref{mainnoncomp} implies Theorem \ref{Huisken}. 
\end{remark}

\begin{example}
Here we consider compact examples of self-similar solutions embedded in compact gradient shrinking Ricci solitons. 
Let $(N,\tilde{g},\tilde{f})$ be a compact gradient shrinking Ricci soliton. 
Then $N$ itself and a critical point $P$ ($0$-dimensional submanifold) of $\tilde{f}$ are trivially compact self-similar solutions, since $H=0$ and $\nabla\tilde{f}^{\bot}=0$. 
The next examples are given in {\it K\"ahler-Ricci solitons}. 
Let $(N,\tilde{g},\tilde{f})$ be a compact gradient shrinking K\"ahler Ricci soliton. 
Let $M\subset N$ be a compact {\it complex} submanifold such that the gradient $\nabla\tilde{f}$ is tangent to $M$. 
Then $M$ is a compact self-similar solution, since $H=0$ (by a well-known fact that a complex submanifold in a K\"ahler manifold is minimal) and $\nabla\tilde{f}^{\bot}=0$ on $M$. 
Actually, Cao \cite{Cao} and Koiso \cite{Koiso} (for notations and assumptions, see \cite{KoisoSakane}) constructed examples of compact gradient shrinking K\"ahler Ricci solitons. 
By their construction, each soliton is the total space of some complex $\mathbb{P}^1$-fibration and the gradient of the potential function is tangent to every $\mathbb{P}^1$-fiber. 
Hence each $\mathbb{P}^1$-fiber is a compact self-similar solution with real dimension $2$. 
\end{example}

Finally, we give some comments for Lagrangian self-similar solutions. 
For a Lagrangian immersion $F:L\to N$ in a K\"ahler manifold $N$ with a K\"ahler form $\omega$, a 1-form $\omega_{H}$ on  $L$ defined by 
$\omega_{H}(X):=\omega(H(F),F_{*}X)$ is called the mean curvature form. 
In Theorem 2.3.5 in \cite{Smoczyk2}, Smoczyk proved that there exists no compact Lagrangian self-similar solution with exact mean curvature form in $\mathbb{C}^n$. 
In his proof, it is proved that a compact Lagrangian self-similar solution with exact mean curvature form is a minimal submanifold in $\mathbb{C}^n$. 
However there exists no compact minimal submanifold in $\mathbb{C}^n$. 
Hence the assertion holds. 
As an analog of this theorem, we have the following theorem and its proof is given at the end of Section~\ref{MCFGSRS}. 
\begin{theorem}\label{minLagthm}
Let $(N,g,f)$ be a gradient shrinking K\"ahler Ricci soliton and 
$F:L\to N$ be a compact Lagrangian self-similar solution with exact mean curvature form. 
Then $F:L\to N$ is a minimal Lagrangian immersion. 
\end{theorem}

\begin{RPL*}
Recently there has been some studies in Ricci-mean curvature flows. 
One of main streams of the study is to generalize results established for mean curvature flows in K\"ahler-Einstein manifolds to Ricci-mean curvature flows along K\"ahler-Ricci flows. 
For example, some results for Lagrangian mean curvature flows can be generalized (cf. \cite{HanLi, LotayPacini}). 
Another main stream of the study is to generalize Huisken's monotonicity formula in $\mathbb{R}^n$ to Ricci-mean curvature flows along Ricci flows. 
In this direction, Lott considered a mean curvature flow in a gradient Ricci soliton in Section~5 in \cite{Lott}, 
and a certain kind of monotonicity formula is obtained in gradient steady soliton case. 
He also gave a definition of a self-similar solution for hypersurfaces in a gradient Ricci soliton. 
Our definition of a self-similar solution (cf. Definition~\ref{selsolingra}) coincides with Lott's one for hypersurfaces. 
In Remark~5 in \cite{Lott}, he pointed out the existence of an analog of a monotonicity formula in gradient shrinking soliton case. 
Actually, a monotonicity formula for a mean curvature flow moving in a gradient shrinking Ricci soliton was also given 
by Magni, Mantegazza and Tsatis (cf. Proposition 3.1 in \cite{MagniMantegazzaTsatis}) more directly. 
In this paper, we reintroduce their monotonicity formula in Section~\ref{MCFGSRS}. 
There is also a generalization of Huisken's work to a mean curvature flow in a Riemannian cone manifold (cf. \cite{FutakiHattoriYamamoto}). 
\end{RPL*}

\begin{OTP*}
The rest of this paper is organized as follows. 
In Section~\ref{pf}, we prove Theorem \ref{maincomp} and \ref{mainnoncomp}, after reviewing the proof of Theorem~\ref{Huisken}. 
In this proof, we use lemmas and propositions proved in the following sections and appendices. 
In Section~\ref{MF}, we introduce some general formulas for the first variation of a certain kind of weighted volume functional. 
In Section~\ref{MCFGSRS}, we study some properties of Ricci-mean curvature flows along Ricci flows constructed from gradient shrinking Ricci solitons, 
and introduce the monotonicity formula. Furthermore, we prove the estimates for higher derivatives of the second fundamental forms of a rescaled flow and give an analog of Stone's estimate. 
In Appendix~\ref{EvoEq}, we give a general treatment of evolution equations for tensors along Ricci-mean curvature flows. 
In Appendix~\ref{EstiforLem}, we give an estimate which is used in the proof of Lemma~\ref{2ndderiofmono}. 
In Appendix~\ref{convofsubsec}, we give a definition of convergence of immersion maps into a Riemannian manifold and prove some propositions. 
\end{OTP*}

\begin{Ack*}
I would like to thank my supervisor, A. Futaki for many useful discussions and constant encouragement. 
\end{Ack*}
\section{Proofs of main theorems}\label{pf}
In this section, we give proofs of Theorem \ref{maincomp} and \ref{mainnoncomp}. 
First of all, we review the proof of Theorem~\ref{Huisken}. 
The key results to prove Theorem~\ref{Huisken} are the following (i), (ii) and (iii). 
\begin{enumerate}
\item[(i)] The monotonicity formula for the weighted volume functional (cf. Theorem 3.1 and Corollary 3.2 in \cite{Huisken}). 
\end{enumerate}
Here the weighted volume functional is defined by 
\[\int_{\tilde{M}}e^{-\frac{|x|^2}{4}}d\mu_{\tilde{M}}\]
for a submanifold $\tilde{M}$ in $\mathbb{R}^n$. 
This result corresponds to Proposition \ref{monoformingra} and \ref{monoformingra2}. 
For a submanifold $\tilde{M}$ (or immersion $\tilde{F}:M\to N$) in a gradient shrinking Ricci soliton $(N,\tilde{g},\tilde{f})$, 
we consider the weighted volume functional $\int_{M}e^{-\tilde{f}}d\mu(\tilde{F}^{*}\tilde{g})$. 
The monotonicity formula decides the profile of the limiting submanifold $\tilde{M}_{\infty}$ if it exists. 
\begin{enumerate}
\item[(ii)] Uniform estimates for all derivatives of second fundamental forms of $\tilde{M}_{s_{j}}$ (cf. Proposition 2.3 in \cite{Huisken}). 
\end{enumerate}
This result corresponds to Proposition~\ref{bdofallderiA}. 
It is proved by the parabolic maximum principle for the evolution equation of $|\tilde{\nabla}^{k}\tilde{A}_{s}|^2$ and the argument of degree (it is explained in the proof of Proposition~\ref{bdofallderiA}). 
This result implies the sub-convergence of $\tilde{M}_{s_{j}}$ to some limiting submanifold $\tilde{M}_{\infty}$. 
\begin{enumerate}
\item[(iii)] A uniform estimate for the second derivative of the weighted volume functional. It is proved by Stone's estimate (cf. Lemma 2.9 in \cite{Stone}) and the result (ii). 
\end{enumerate}
In this paper we prepare an analog of Stone's estimate in Lemma~\ref{stonelem}, and by combining Lemma~\ref{stonelem} and Proposition~\ref{bdofallderiA} 
we prove Proposition~\ref{2ndderiofmono} which is an analog of the result (iii). 
This result is necessary in the following sense. 
In general, if we know $\frac{d}{ds}\mathcal{F}(s)\leq 0$ for some smooth non-negative function $\mathcal{F}:[0,\infty)\to [0,\infty)$, 
we can say that $\mathcal{F}$ is monotone decreasing and converges to some value as $s\to\infty$. 
However we can not say that $\frac{d}{ds}\mathcal{F}(s_{j})\to 0$ for any sequence $s_{1}<s_{2}<\cdots\to\infty$. 
If we further know that $|\frac{d^2}{d^2s}\mathcal{F}(s)|\leq C$ uniformly, then we can say that. 
In our situation, $\mathcal{F}(s)$ is the weighted volume of $\tilde{M}_{s}$. 
This argument is pointed out right before Lemma 3.2.7 in \cite{Mantegazza}. 

\begin{proof}[Proof of Theorem~\ref{maincomp}]
First, we prove the existence of a smooth manifold $M_{\infty}$ and a smooth map $\tilde{F}_{\infty}:M_{\infty}\to N$. 
Next, we show that this $\tilde{F}_{\infty}$ is a self-shrinker by using the monotonicity formula (\ref{monoformingraform2}) in Proposition~\ref{monoformingra2}. 

By Proposition~\ref{bdofallderiA}, for all $k=0,1,2,\dots$, there exist constants $C_{k}>0$ such that 
\[|\tilde{\nabla}^{k}A(\tilde{F_{s}})|\leq C_{k} \quad\mathrm{on}\quad M\times[-\log T,\infty). \]
Since $N$ is compact, by Theorem~\ref{convofimm}, we get a sub-sequence $j_{k}$, a pointed manifold $(M_{\infty}, x_{\infty})$ and an immersion map $\tilde{F}_{\infty}:M_{\infty}\to N$ 
with a complete Riemannian metric $F_{\infty}^{*}\tilde{g}$ on $M_{\infty}$ 
such that $\tilde{F}_{s_{j_{k}}}:(M,x_{j_{k}})\to N$ converges to $\tilde{F}_{\infty}:(M_{\infty},x_{\infty})\to N$ in the sense of Definition~\ref{convofimmmaps} as $k\to \infty$. 
We denote $\tilde{F}_{s_{j_{k}}}$ by $\tilde{F}_{k}$ for short. 
Then, there exist 
an exhaustion $\{U_{k}\}_{k=1}^{\infty}$ of $M_{\infty}$ with $x_{\infty}\in U_{k}$ and 
a sequence of diffeomorphisms $\Psi_{k}:U_{k}\to V_{k}:=\Psi_{k}(U_{k})\subset M$ with $\Psi_{k}(x_{\infty})=x_{j_{k}}$ 
such that $\Psi_{k}^{*}(\tilde{F}_{k}^{*}\tilde{g})$ converges in $C^{\infty}$ to $\tilde{F}_{\infty}^{*}\tilde{g}$ uniformly on compact sets in $M_{\infty}$, 
and furthermore the sequence of maps $\tilde{F}_{k}\circ \Psi_{k}:U_{k}\to N$ converges in $C^{\infty}$ to $F_{\infty}:M_{\infty}\to N$ uniformly on compact sets in $M_{\infty}$. 

Let $K\subset M_{\infty}$ be any compact set. 
Then we will prove that 
\[ \int_{K} \Bigl|H(\tilde{F}_{\infty})+\nabla \tilde{f}^{\bot_{\tilde{F}_{\infty}}}\Bigr|_{\tilde{g}}^2 e^{-\tilde{f}\circ\tilde{F}_{\infty}} d\mu(\tilde{F}_{\infty}^{*}\tilde{g})=0. \]
It is clear that this implies that $\tilde{F}_{\infty}:M_{\infty}\to N$ satisfies 
\[H(\tilde{F}_{\infty})=-\nabla\tilde{f}^{\bot_{\tilde{F}_{\infty}}}\]
on $M_{\infty}$, where $\bot_{\tilde{F}_{\infty}}$ denotes the normal projection with respect to $\tilde{F}_{\infty}$. 
Its proof is the following. 
For $K$, there exists $k_{0}$ such that $K\subset U_{k}$ for all $k\geq k_{0}$. 
Since $\tilde{F}_{k}\circ \Psi_{k}:U_{k}\to N$ converges to $F_{\infty}:M_{\infty}\to N$ in $C^{\infty}$ uniformly on $K$ for $k\geq k_{0}$, we have 
\begin{align}\label{lastconv1}
\begin{aligned}
&\int_{K} \Bigl|H(\tilde{F}_{k}\circ \Psi_{k})+\nabla \tilde{f}^{\bot_{\tilde{F}_{k}\circ \Psi_{k}}}\Bigr|_{\tilde{g}}^2 
e^{-\tilde{f}\circ(\tilde{F}_{k}\circ \Phi_{k})} d\mu((\tilde{F}_{k}\circ \Phi_{k})^{*}\tilde{g})\\
\rightarrow&\int_{K} \Bigl|H(\tilde{F}_{\infty})+\nabla \tilde{f}^{\bot_{\tilde{F}_{\infty}}}\Bigr|_{\tilde{g}}^2 e^{-\tilde{f}\circ\tilde{F}_{\infty}} d\mu(\tilde{F}_{\infty}^{*}\tilde{g})
\end{aligned}
\end{align}
as $k\to \infty$. 
Since $\Psi_{k}:U_{k}\to V_{k}\subset M$ is a diffeomorphism, it is clear that 
\begin{align}\label{lastconv2}
\begin{aligned}
&\int_{K} \Bigl|H(\tilde{F}_{k}\circ \Psi_{k})+\nabla \tilde{f}^{\bot_{\tilde{F}_{k}\circ \Psi_{k}}}\Bigr|_{\tilde{g}}^2 
e^{-\tilde{f}\circ(\tilde{F}_{k}\circ \Psi_{k})} d\mu((\tilde{F}_{k}\circ \Psi_{k})^{*}\tilde{g})\\
=&\int_{\Psi_{k}(K)} \Bigl|H(\tilde{F}_{k})+\nabla \tilde{f}^{\bot_{\tilde{F}_{k}}}\Bigr|_{\tilde{g}}^2 e^{-\tilde{f}\circ\tilde{F}_{k}} d\mu(\tilde{F}_{k}^{*}\tilde{g})\\
\leq& \int_{M} \Bigl|H(\tilde{F}_{k})+\nabla \tilde{f}^{\bot_{\tilde{F}_{k}}}\Bigr|_{\tilde{g}}^2 e^{-\tilde{f}\circ\tilde{F}_{k}} d\mu(\tilde{F}_{k}^{*}\tilde{g}). 
\end{aligned}
\end{align}
By using the monotonicity formula (\ref{monoformingraform2}) and Lemma~\ref{2ndderiofmono}, one can prove that 
\begin{align}\label{lastconv3}
\int_{M} \Bigl|H(\tilde{F}_{k})+\nabla \tilde{f}^{\bot_{\tilde{F}_{k}}}\Bigr|_{\tilde{g}}^2 e^{-\tilde{f}\circ\tilde{F}_{k}} d\mu(\tilde{F}_{k}^{*}\tilde{g})\to 0
\end{align}
as $k\to \infty$ by the argument of contradiction. 
Actually, assume that there exist a constant $\delta>0$ and a subsequence $\{\ell\}\subset \{k\}$ with $\ell\to\infty$ such that 
\[\int_{M} \Bigl|H(\tilde{F}_{\ell})+\nabla \tilde{f}^{\bot_{\tilde{F}_{\ell}}}\Bigr|_{\tilde{g}}^2 e^{-\tilde{f}\circ\tilde{F}_{\ell}} d\mu(\tilde{F}_{\ell}^{*}\tilde{g})\geq\delta. \]
Then one can easily see that 
\[\int_{M} \Bigl|H(\tilde{F}_{s})+\nabla \tilde{f}^{\bot_{\tilde{F}_{s}}}\Bigr|_{\tilde{g}}^2 e^{-\tilde{f}\circ\tilde{F}_{s}} d\mu(\tilde{F}_{s}^{*}\tilde{g}) \geq \frac{\delta}{2}, \]
for $s\in[s_{\ell},s_{\ell}+\frac{\delta}{2C'}]$, 
where we used Lemma~\ref{2ndderiofmono} and $C'$ is the constant appeared in that lemma. 
Hence we have that 
\[\int_{-\log T}^{\infty} \int_{M} \Bigl|H(\tilde{F}_{s})+\nabla \tilde{f}^{\bot_{\tilde{F}_{s}}}\Bigr|_{\tilde{g}}^2 e^{-\tilde{f}\circ\tilde{F}_{s}} d\mu(\tilde{F}_{s}^{*}\tilde{g}) ds =\infty. \]
On the other hand, by the monotonicity formula (\ref{monoformingraform2}): 
\[\frac{d}{ds} \int_{M}e^{-\tilde{f}\circ\tilde{F}} \mathop{d\mu(\tilde{F}^{*}\tilde{g})}=-\int_{M} \Bigl|H(\tilde{F})+\nabla \tilde{f}^{\bot_{\tilde{F}}}\Bigr|_{\tilde{g}}^2 e^{-\tilde{f}\circ\tilde{F}} d\mu(\tilde{F}^{*}\tilde{g})\leq 0, \]
the weighted volume 
\[\int_{M} e^{-\tilde{f}\circ\tilde{F}_{s}} d\mu(\tilde{F}_{s}^{*}\tilde{g}) \]
is monotone decreasing and non-negative, thus it converges to some value 
\[\alpha:=\lim_{s\to\infty}\int_{M} e^{-\tilde{f}\circ\tilde{F}_{s}} d\mu(\tilde{F}_{s}^{*}\tilde{g})<\infty. \]
Hence we have 
\begin{align*}
\int_{-\log T}^{\infty} \int_{M} \Bigl|H(\tilde{F}_{s})+\nabla \tilde{f}^{\bot_{\tilde{F}_{s}}}\Bigr|_{\tilde{g}}^2 e^{-\tilde{f}\circ\tilde{F}_{s}} d\mu(\tilde{F}_{s}^{*}\tilde{g})
=-\alpha +\int_{M} e^{-\tilde{f}\circ\tilde{F}_{\bullet}} d\mu(\tilde{F}_{\bullet}^{*}\tilde{g})<\infty, 
\end{align*}
where $\bullet=-\log T$. 
This is a contradiction. 
Thus, by combining (\ref{lastconv1})-(\ref{lastconv3}), it follows that 
\[ \int_{K} \Bigl|H(\tilde{F}_{\infty})+\nabla \tilde{f}^{\bot_{\tilde{F}_{\infty}}}\Bigr|_{\tilde{g}}^2 e^{-\tilde{f}\circ\tilde{F}_{\infty}} d\mu(\tilde{F}_{\infty}^{*}\tilde{g})=0. \]
Here we completed the proof. 
\end{proof}

Next, we give the proof of the non-compact version of the above theorem. 
\begin{proof}[Proof of Theorem~\ref{mainnoncomp}]
We will prove that $\tilde{F}_{s_{j}}(p_{0})$ is a bounded sequence in $(N,\tilde{g})$. 
For any $t_{1}$, $t_{2}$ with $0\leq t_{1}<t_{2}<T$, we can take $\{F_{t}(p_{0})\}_{t\in[t_{1}, t_{2}]}$ as a curve joining $F_{t_{1}}(p_{0})$ and $F_{t_{2}}(p_{0})$. 
Hence we have 
\begin{align*}
\ell_{F_{t_{2}}(p_{0}),t_{2}}(F_{t_{1}}(p_{0}),t_{1})\leq& \frac{1}{2\sqrt{t_{2}-t_{1}}}\int_{t_{1}}^{t_{2}}\sqrt{t_{2}-t}\left( R(g_{t}) + \left|\frac{\partial F_{t}}{\partial t}\right|^2 \right)dt\\
=& \frac{1}{2\sqrt{t_{2}-t_{1}}}\int_{t_{1}}^{t_{2}}\sqrt{t_{2}-t}\bigl( R(g_{t}) + |H(F_{t})|^2 \bigr)dt
\end{align*}
By the assumption (A2), $(T-t)|H(F_{t})|^2$ is bounded, and it is clear that $(T-t)R(g_{t})=R(g_{0})$ and it is also bounded by the assumption in Remark~\ref{immRiem2}. 
Hence we have $R(g_{t}) + |H(F_{t})|^2\leq \frac{C}{T-t}$ for some $C>0$ and 
\begin{align*}
\ell_{F_{t_{2}}(p_{0}),t_{2}}(F_{t_{1}}(p_{0}),t_{1})\leq&  \frac{C}{2\sqrt{t_{2}-t_{1}}}\int_{t_{1}}^{t_{2}}\frac{\sqrt{t_{2}-t}}{T-t}dt\\
\leq &  \frac{C}{2\sqrt{t_{2}-t_{1}}}\int_{t_{1}}^{t_{2}}\frac{1}{\sqrt{T-t}}dt\\
\leq &  C\frac{\sqrt{T-t_{1}}}{\sqrt{t_{2}-t_{1}}}. 
\end{align*}
By the assumption (B2), by taking the limit as $t_{2}\to T$, we have 
\[f(F_{t_{1}}(p_{0}),t_{1})\leq C. \]
Since $f(F_{t}(p_{0}),t)=f_{t}(F_{t}(p_{0}))=\tilde{f}(\tilde{F}_{s}(p_{0}))$, the above bound means that 
\[\tilde{f}(\tilde{F}_{s}(p_{0}))\leq C\]
on $s\in[-\log T,\infty)$. 
In \cite{CaoZhou} (cf, Theorem 1.1), Cao and Zhou proved that there exist positive constants $C_{1}$ and $C_{2}$ such that 
\[\frac{1}{4}(r-C_{1})^2\leq \tilde{f}\leq \frac{1}{4}(r+C_{2})^2\]
on $N$, where $r(q)=d_{\tilde{g}}(q_{0}, q)$ is the distance function from some fixed point $q_{0}$ in $N$. 
Hence we have 
\[d_{\tilde{g}}(q_{0}, \tilde{F}_{s}(p_{0}))\leq 2\sqrt{C}+C_{1}, \]
that is, $\tilde{F}_{s}(p_{0})$ moves in a bounded region in $N$. 
Hence we can use Theorem~\ref{convofimm} with a bounded sequence  $\tilde{F}_{s_{j}}(p_{0})$. 
Then the remainder part of the proof is completely same as the proof of the case that $N$ is compact. 
\end{proof}
\section{Monotonicity formulas}\label{MF}
In this section, we introduce some general formulas which are useful in the following sections and appendices. 
Let $M$ and $N$ be manifolds with dimension $m$ and $n$ respectively, and assume that $m\leq n$ and $M$ is compact. 
We denote the space of all immersion maps from $M$ to $N$ by $\mathfrak{Imm}(M,N)$ and 
the space of all Riemannian metrics on $N$ by $\mathfrak{Met}(N)$. 
Consider the following functional: 
\begin{align}
\begin{gathered}
\mathcal{F}:C^{\infty}(M)\times \mathfrak{Imm}(M,N) \times C^{\infty}(N)_{>0} \times \mathfrak{Met}(N) \to \mathbb{R}\\
\mathcal{F}(u,F,\rho,g):=\int_{M}u\mathop{F^{*}\hspace{-1.2mm}\rho} \mathop{d\mu(F^{*}g)}. 
\end{gathered}
\end{align}
Here $u$ is a smooth function on $M$ and $\rho$ is a positive smooth function on $N$. 
First of all, we remark some elementary symmetric properties associated with $\mathcal{F}$. 
Here we denote diffeomorphism groups of $M$ and $N$ by $\mathrm{Diff}(M)$ and $\mathrm{Diff}(N)$ respectively. 
\begin{remark}
For $\varphi\in\mathrm{Diff}(M)$ and $\psi\in\mathrm{Diff}(N)$, we have 
\[\mathcal{F}(\varphi^{*}u,\psi^{-1}\circ F\circ \varphi,\psi^{*}\rho,\psi^{*}g)=\mathcal{F}(u,F,\rho,g), \]
and for a positive constant $\lambda>0$ we have
\[\mathcal{F}(\lambda^{n-m}u,F,\lambda^{-n}\rho,\lambda^2g)=\mathcal{F}(u,F,\rho,g). \]
\end{remark}
Let $p:=(u,F,\rho,g)$ be a point in $C^{\infty}(M)\times \mathfrak{Imm}(M,N) \times C^{\infty}(N)_{>0} \times \mathfrak{Met}(N)$ and 
$v:=(w,V,k,h)$ be a tangent vector of $C^{\infty}(M)\times \mathfrak{Imm}(M,N) \times C^{\infty}(N)_{>0} \times \mathfrak{Met}(N)$ at $p$. 
Namely, $w\in C^{\infty}(M)$, $V\in\Gamma(M,F^{*}(TN))$, $k\in C^{\infty}(N)$ and $h\in \mathrm{Sym}^2(N)$. 
Then we consider the first variation of $\mathcal{F}$ at $p$ in the direction $v$, denoted by $\delta_{v}\mathcal{F}(p)$. 

\begin{proposition}\label{vari}
We have 
\begin{align}\label{vari1}
\begin{aligned}
\delta_{v}\mathcal{F}(p)
=&-\int_{M}\mathop{u}g(V+{\nabla f}^{\bot_{F}},H(F)+{\nabla f}^{\bot_{F}})\mathop{F^{*}\hspace{-1.2mm}\rho} d\mu(F^{*}g)\\
&+\int_{M}\mathop{u}F^{*}\biggl(\Delta_{g}\rho+k+\frac{1}{2}\rho\mathop{\mathrm{tr}}h\biggr)d\mu(F^{*}g)\\
&+\int_{M}\biggl(w-\Delta_{F^{*}g}u-g(V,F_{*}\nabla u)\\
&\hspace{29mm} +u \mathop{\mathrm{tr}^{\bot_{F}}}\bigl(\mathop{\mathrm{Hess}}f-\frac{1}{2}h\bigr)\biggr)\mathop{F^{*}\hspace{-1.2mm}\rho} d\mu(F^{*}g),
\end{aligned}
\end{align}
where we define $f$ by $\rho=(4\pi\tau)^{-\frac{n}{2}}e^{-f}$ for a positive function $\tau=\tau(t)$ {\rm(}which depends only on $t${\rm)} 
and $H(F)$ is the mean curvature vector field of immersion $F$ from $M$ to a Riemannian manifold $(N,g)$. 
\end{proposition}

\begin{remark}
Here, note that there is an ambiguity of a choice of a function $\tau$, 
but the gradient and Hessian of $f$ do not depend on the choice of $\tau$. 
\end{remark}

\begin{notation}
By $\bot_{F}$, we denote the normal projection with respect to the orthogonal decomposition 
\[F^{*}(TN)=F_{*}(TM)\oplus T^{\bot_{F}}M\]
defined by the immersion $F$, 
and by $\mathop{\mathrm{tr}^{\bot_{F}}}$ we denote the normal trace, that is, 
for a 2-tensor $\eta$ on $N$ and a point $p\in M$, $(\mathop{\mathrm{tr}^{\bot_{F}}}\eta)(p)$ is defined by 
\[(\mathop{\mathrm{tr}^{\bot_{F}}}\eta)(p):=\sum_{j=1}^{n-m}\eta(F(p))(\nu_{j},\nu_{j}), \]
where $\{\nu_{j}\}_{j=1}^{n-m}$ is an orthonormal basis of  $T^{\bot_{F}}_{p}M$. 
\end{notation}

\begin{proof}
Let $\{F_{s}:M\to N\}_{s\in(-\epsilon, \epsilon)}$ be a smooth 1-parameter family of immersions with 
\[F_{0}= F\quad \mathrm{and}\quad \frac{\partial F_{s}}{\partial s}\bigg|_{s=0}=V. \]
Let $u_{s}:=u+sw$, $\rho_{s}:=\rho+sk$ and $g_{s}:=g+sh$. 
Then $p_{s}:=(u_{s},F_{s},\rho_{s},g_{s})$ is a curve in $C^{\infty}(M)\times \mathfrak{Imm}(M,N) \times C^{\infty}(N)_{>0} \times \mathfrak{Met}(N)$ with 
$p_{0}=p$ and $\dot{p}_{0}=v$. 
Then the first variation of $\mathcal{F}$ at $p$ in the direction $v$ is calculated as 
\[\delta_{v}\mathcal{F}(p)=\frac{d}{ds}\bigg|_{s=0}\mathcal{F}(p_{s})=\frac{d}{ds}\bigg|_{s=0}\int_{M}u_{s}\mathop{F_{s}^{*}\hspace{-0.5mm}\rho_{s}}d\mu(F_{s}^{*}g_{s}), \]
and we have
\begin{align}\label{eq1inderi}
\begin{aligned}
&\frac{d}{ds}\bigg|_{s=0}\int_{M}u_{s}\mathop{F_{s}^{*}\hspace{-0.5mm}\rho_{s}}d\mu(F_{s}^{*}g_{s})\\
=&\int_{M}w\mathop{F^{*}\hspace{-1.2mm}\rho}d\mu(F^{*}g)+\int_{M}u\mathop{g(V,\nabla\rho)}d\mu(F^{*}g)+\int_{M}u\mathop{F^{*}\hspace{-0.5mm}k}d\mu(F^{*}g)\\
&+\int_{M}u\mathop{F^{*}\hspace{-1.2mm}\rho}\biggl(\frac{d}{ds}\bigg|_{s=0} d\mu(F_{s}^{*}g)\biggr)+\int_{M}u\mathop{F^{*}\hspace{-1.2mm}\rho} \biggl(\frac{d}{ds}\bigg|_{s=0} d\mu(F^{*}g_{s})\biggr). 
\end{aligned}
\end{align}
It is well-known that the first variation of the induced measure $d\mu(F_{s}^{*}g)$ is given by 
\[\frac{d}{ds}\bigg|_{s=0} d\mu(F_{s}^{*}g)=\{\mathop{\mathrm{div}_{F^{*}g}}F_{*}^{-1}(V^{\top_{F}})-g(H(F),V)\}d\mu(F^{*}g). \]
On the right hand side of the above equation, we decompose $V$ as $V=V^{\top_{F}}+V^{\bot_{F}}\in F_{*}(TM)\oplus T^{\bot_{F}}M$, 
and we take the divergence of $F_{*}^{-1}(V^{\top_{F}})$ on a Riemannian manifold $(M,F^{*}g)$. 

On the other hand, $F^{*}g_{s}$ is a time-dependent metric on $M$. 
Since $g_{s}=g+sh$, we have $F^{*}g_{s}=F^{*}g+sF^{*}h$. 
Thus, the derivation of $F^{*}g_{s}$ is $F^{*}h$ at $s=0$. 
In such a situation, it is also well-known that the first variation of the induced measure $d\mu(F^{*}g_{s})$ of a time-dependent metric on $M$ is given by 
\[\frac{d}{ds}\bigg|_{s=0} d\mu(F^{*}g_{s})=\frac{1}{2}\mathop{\mathrm{tr}}(F^{*}h)d\mu(F^{*}g), \]
where the trace is taken with respect to a metric $F^{*}g$ on $M$. 
By the divergence formula on $(M,F^{*}g)$, we have
\begin{align*}
&\int_{M}u\mathop{F^{*}\hspace{-1.2mm}\rho}\mathop{\mathrm{div}_{F^{*}g}}F_{*}^{-1}(V^{\top_{F}}) d\mu(F^{*}g)\\
=&-\int_{M}(F^{*}g)(F_{*}^{-1}(V^{\top_{F}}), \nabla(u\mathop{F^{*}\hspace{-1.2mm}\rho}))d\mu(F^{*}g)\\
=&-\int_{M}g(V,F_{*}\nabla u)\mathop{F^{*}\hspace{-1.2mm}\rho}d\mu(F^{*}g)-\int_{M}u\mathop{g(V,\nabla \rho^{\top_{F}})}d\mu(F^{*}g).
\end{align*}
Since $\nabla \rho=-\rho \nabla f$, we have
\begin{align}\label{eq2inderi}
\begin{aligned}
&\int_{M}u\mathop{g(V,\nabla\rho)}d\mu(F^{*}g)+\int_{M}u\mathop{F^{*}\hspace{-1.2mm}\rho}\biggl(\frac{d}{ds}\bigg|_{s=0} d\mu(F_{s}^{*}g)\biggr)\\
=&-\int_{M}g(V,F_{*}\nabla u)\mathop{F^{*}\hspace{-1.2mm}\rho}d\mu(F^{*}g)-\int_{M}u \mathop{g(V,H(F)+{\nabla f}^{\bot_{F}})}\mathop{F^{*}\hspace{-1.2mm}\rho} d\mu(F^{*}g). 
\end{aligned}
\end{align}
It is clear that 
\begin{align*}
\mathop{\mathrm{tr}}(F^{*}h)=F^{*}(\mathop{\mathrm{tr}}h)-\mathop{\mathrm{tr}^{\bot_{F}}}h. 
\end{align*}
Hence we have
\begin{align*}
&\int_{M}u\mathop{F^{*}\hspace{-0.5mm}k}d\mu(F^{*}g)+\int_{M}u\mathop{F^{*}\hspace{-1.2mm}\rho} \biggl(\frac{d}{ds}\bigg|_{s=0} d\mu(F^{*}g_{s})\biggr)\\
=&\int_{M}uF^{*}\biggl(k+\frac{1}{2}\rho\mathop{\mathrm{tr}}h\biggr) d\mu(F^{*}g)-\int_{M}\frac{1}{2}u\mathop{F^{*}\hspace{-0.8mm}\rho}(\mathop{\mathrm{tr}^{\bot_{F}}}h) d\mu(F^{*}g). 
\end{align*}
Furthermore, one can easily see that 
\begin{align*}
F^{*}(\Delta_{g}\rho)&=\Delta_{F^{*}g}(F^{*}\hspace{-1.2mm}\rho)-g(H(F),\nabla \rho)+\mathop{\mathrm{tr}^{\bot_{F}}}(\mathop{\mathrm{Hess}}\rho)\\
&=\Delta_{F^{*}g}(F^{*}\hspace{-1.2mm}\rho)+\mathop{F^{*}\hspace{-1.2mm}\rho} g({\nabla f}^{\bot_{F}},H(F)+{\nabla f}^{\bot_{F}})-\mathop{F^{*}\hspace{-1.2mm}\rho}\mathop{\mathrm{tr}^{\bot_{F}}}(\mathop{\mathrm{Hess}}f). 
\end{align*}
Hence we have
\begin{align}\label{eq3inderi}
\begin{aligned}
&\int_{M}u\mathop{F^{*}\hspace{-0.5mm}k}d\mu(F^{*}g)+\int_{M}u\mathop{F^{*}\hspace{-1.2mm}\rho} \biggl(\frac{d}{ds}\bigg|_{s=0} d\mu(F^{*}g_{s})\biggr)\\
=&\int_{M}uF^{*}\biggl(\Delta_{g}\rho+k+\frac{1}{2}\rho\mathop{\mathrm{tr}}h\biggr) d\mu(F^{*}g)-\int_{M}\frac{1}{2}u\mathop{F^{*}\hspace{-0.8mm}\rho}(\mathop{\mathrm{tr}^{\bot_{F}}}h) d\mu(F^{*}g)\\
&-\int_{M}uF^{*}(\Delta_{g}\rho) d\mu(F^{*}g)\\
=&\int_{M}\mathop{u}F^{*}\biggl(\Delta_{g}\rho+k+\frac{1}{2}\rho\mathop{\mathrm{tr}}h\biggr)d\mu(F^{*}g)\\
&-\int_{M}\mathop{u}g({\nabla f}^{\bot_{F}},H(F)+{\nabla f}^{\bot_{F}})\mathop{F^{*}\hspace{-1.2mm}\rho} d\mu(F^{*}g)\\
&+\int_{M}\biggl(-\Delta_{F^{*}g}u+u\mathop{\mathrm{tr}^{\bot_{F}}}(\mathop{\mathrm{Hess}}f-\frac{1}{2}h)\biggr)\mathop{F^{*}\hspace{-1.2mm}\rho}d\mu(F^{*}g), 
\end{aligned}
\end{align}
where we used 
\[\int_{M}u\Delta_{F^{*}g}(F^{*}\hspace{-1.2mm}\rho) d\mu(F^{*}g)=\int_{M}(\Delta_{F^{*}g}u) \mathop{F^{*}\hspace{-1.2mm}\rho}d\mu(F^{*}g).\]
Finally, by combining equations (\ref{eq1inderi})-(\ref{eq3inderi}), we get the formula (\ref{vari1}).
\end{proof}

By using this general formula (\ref{vari1}), we get the following monotonicity formula for Ricci-mean curvature flows. 
\begin{proposition}\label{genmonoRM}
Assume that the pair $g=(\, g_{t}\,;\, t\in[0,T_{1})\,)$ and $F:M\times [0,T_{2}) \to N$ is a solution of Ricci-mean curvature flow with $T_{2}\leq T_{1}$, that is, these satisfy (\ref{evoeq21}) and (\ref{evoeq22}). 
Further assume that a smooth positive function $\rho:N\times[0,T_{1})\to\mathbb{R}^{+}$ on $N$ 
and a smooth non-negative function $u:M\times[0,T_{2})\to\mathbb{R}$ on $M$ satisfy the following coupled equations: 
\begin{subequations}
\begin{align}
&\frac{\partial \rho_{t}}{\partial t}=-\Delta_{g_{t}} \rho_{t} +R(g_{t})\rho_{t}\label{evoeq23}\\
&\frac{\partial u_{t}}{\partial t}=\Delta_{F_{t}^{*}g_{t}}u_{t}-u_{t}\mathop{\mathrm{tr}^{\bot_{F_{t}}}}\bigl(\mathrm{Ric}(g_{t})+\mathop{\mathrm{Hess}}f_{t}) \label{evoeq24}, 
\end{align}
\end{subequations}
where we define $f$ by $\rho=(4\pi\tau)^{-\frac{n}{2}}e^{-f}$ for a positive function $\tau=\tau(t)$. 
Then we have, for all $t\in(0,T_{2})$, 
\begin{align}\label{monoformforRM}
\frac{d}{dt}\mathcal{F}(u_{t},F_{t},\rho_{t},g_{t})=-\int_{M} u_{t}\Bigl|H(F_{t})+\nabla f_{t}^{\bot_{F_{t}}}\Bigr|_{g_{t}}^2F_{t}^{*}\rho_{t}d\mu(F_{t}^{*}g_{t})\leq 0. 
\end{align}
\end{proposition}
\begin{proof}
Since $g_{t}$ is a solution of the Ricci flow (\ref{evoeq21}), $h=-2\mathrm{Ric}(g_{t})$ in the equation (\ref{vari1}) of Proposition~\ref{vari}. 
Furthermore, $V=H(F_{t})$ in this case, and $g(V,F_{t*}\nabla u_{t})=0$ since $V(=H(F_{t}))$ is a normal vector and $F_{t*}\nabla u_{t}$ is a tangent vector. 
Then, the equality (\ref{monoformforRM}) is clear by Proposition~\ref{vari}. 
\end{proof}

\begin{remark}
The equation (\ref{evoeq23}) is called the conjugate heat equation for the Ricci flow, 
and the equation (\ref{evoeq24}) is a linear heat equation with time-dependent potential $\mathop{\mathrm{tr}^{\bot_{F_{t}}}}\bigl(\mathrm{Ric}(g_{t})+\mathop{\mathrm{Hess}}f_{t})$ 
on $M$. 
\end{remark}

\begin{proposition}
Assume that $\tau(t)=T-t$. 
Let $u:M\times[0,T)\to\mathbb{R}$ be a solution for (\ref{evoeq24}). 
Define $v:M\times[0,T)\to\mathbb{R}$ by $u=(4\pi\tau)^{\frac{n-m}{2}}v$. 
Then $v$ satisfies
\begin{align*}
\frac{\partial v}{\partial t}=\Delta_{F^{*}g}v-v\mathop{\mathrm{tr}^{\bot_{F}}}(\mathrm{Ric}+\mathop{\mathrm{Hess}}f-\frac{g}{2\tau})\tag{\ref{evoeq24}$'$}\label{evoeq24'},
\end{align*}
and the converse is also true. 
\end{proposition}
\begin{proof}
We have 
\begin{align*}
&\frac{\partial u}{\partial t}-\Delta_{F^{*}g}u+u\mathop{\mathrm{tr}^{\bot_{F}}}(\mathrm{Ric}+\mathop{\mathrm{Hess}}f)\\
=&-\frac{n-m}{2\tau}u+(4\pi\tau)^{\frac{n-m}{2}}\frac{\partial v}{\partial t}-\Delta_{F^{*}g}u+u\mathop{\mathrm{tr}^{\bot_{F}}}(\mathrm{Ric}+\mathop{\mathrm{Hess}}f)\\
=&(4\pi\tau)^{\frac{n-m}{2}}\biggl(\frac{\partial v}{\partial t}-\Delta_{F^{*}g}v+v\mathop{\mathrm{tr}^{\bot_{F}}}(\mathrm{Ric}+\mathop{\mathrm{Hess}}f-\frac{g}{2\tau})\biggr). 
\end{align*}
Thus, the equivalence is clear. 
\end{proof}

\begin{example}
If the ambient space is a Euclidean space, that is, $(N,g)=(\mathbb{R}^{n},g_{\mathrm{st}})$, 
we can reduce Huisken's monotonicity formula from (\ref{monoformforRM}). 
Let $M$ be an $m$ dimensional compact manifold and $F:M\times[0,T)\to \mathbb{R}^{n}$ be a mean curvature flow. 
Fix a point $y_{0}\in\mathbb{R}^{n}$. Let $\rho:\mathbb{R}^{n}\times[0,T)\to \mathbb{R}^{+}$ be the standard backward heat kernel on $\mathbb{R}^{n}$ at $(y_{0},T)$, that is, 
$\rho$ is defined by
\[\rho(y,t):=\frac{1}{{(4\pi(T-t))}^{\frac{n}{2}}}e^{-\frac{|y-y_{0}|^2}{4(T-t)}}. \]
Of course, $\rho$ satisfies the backward heat equation (\ref{evoeq23}) with $R=0$. 
In this case, since $f$ is $|y-y_{0}|^2/(4(T-t))$, we have 
\[\mathop{\mathrm{Hess}}f=\frac{g_{\mathrm{st}}}{2(T-t)}\quad\mathrm{and}\quad {\mathrm{tr}}^{\bot}(\mathop{\mathrm{Hess}}f)=\frac{n-m}{2(T-t)}. \]
Thus one can easily see that $u:M\times[0,T)\to \mathbb{R}$ defined by 
\[u(p,t):=(4\pi(T-t))^{\frac{n-m}{2}}\]
is the non-negative solution of (\ref{evoeq24}) with initial condition $u(\cdot,0)=(4\pi T)^{\frac{n-m}{2}}$. 
Hence by Theorem~\ref{genmonoRM} we have 
\[\frac{d}{dt}\mathcal{F}(u_{t},F_{t},\rho_{t},g_{\mathrm{st}})=-\int_{M} u_{t}\Bigl|H(F_{t})+\nabla f_{t}^{\bot_{F_{t}}}\Bigr|^2F_{t}^{*}\rho_{t}d\mu(F_{t}^{*}g_{\mathrm{st}}). \]
By definitions, we have
\begin{align*}
\mathcal{F}(u_{t},F_{t},\rho_{t},g_{\mathrm{st}})&=\int_{M}u_{t}\mathop{F_{t}^{*}\hspace{-0.5mm}\rho_{t}}d\mu(F_{t}^{*}g_{\mathrm{st}})\\
&=\int_{M}\frac{1}{{(4\pi(T-t))}^{\frac{m}{2}}}e^{-\frac{|F_{t}-y_{0}|^2}{4(T-t)}}d\mu(F_{t}^{*}g_{\mathrm{st}})
\end{align*}
and 
\[\nabla f_{t}(F_{t}(p))=\frac{\overrightarrow{{F}_{t}}(p)-\overrightarrow{y_{0}}}{2(T-t)}\]
at $p\in M$. Then we get Huisken's monotonicity formula: 
\begin{align*}
&\frac{d}{dt}\int_{M}\frac{1}{{(4\pi(T-t))}^{\frac{m}{2}}}e^{-\frac{|F_{t}-y_{0}|^2}{4(T-t)}}d\mu(F_{t}^{*}g_{\mathrm{st}})\\
=&-\int_{M} \frac{1}{{(4\pi(T-t))}^{\frac{m}{2}}}e^{-\frac{|F_{t}-y_{0}|^2}{4(T-t)}}\biggl|H(F_{t})+\frac{(\overrightarrow{{F}_{t}}(p)-\overrightarrow{y_{0}})^{\bot_{F_{t}}}}{2(T-t)}\biggr|^2d\mu(F_{t}^{*}g_{\mathrm{st}}) \leq 0. 
\end{align*}
\end{example}
\section{mean curvature flows in gradient shrinking Ricci solitons}\label{MCFGSRS}
In this section, we recall some definitions and properties of 
gradient shrinking Ricci solitons and self-similar solutions (cf. Definition~\ref{selsolingra}), 
and prove the monotonicity formula for a Ricci-mean curvature flow along a Ricci flow constructed from a gradient shrinking Ricci soliton and also prove an analog of Stone's estimate. 

Recall that if an $n$-dimensional Riemannian manifold $(N,\tilde{g})$ and a function $\tilde{f}$ on $N$ satisfies the equation (\ref{shrisol}): 
\[\mathrm{Ric}(\tilde{g})+\mathop{\mathrm{Hess}}\tilde{f}-\frac{1}{2} \tilde{g}=0, \]
it is called a gradient shrinking Ricci soliton. 
In this paper we assume that $(N,\tilde{g})$ is a complete Riemannian manifold. 
Then by the result due to Zhang \cite{Zhang}, it follows that $\nabla \tilde{f}$ is a complete vector field on $N$. 
As Theorem~20.1 in Hamilton's paper \cite{Hamilton4}, one can easily see that $R(\tilde{g})+|\nabla\tilde{f}|^2-\tilde{f}$ is a constant. 
Hence by adding some constant to $\tilde{f}$ if necessary, we can assume that the potential function $\tilde{f}$ satisfying (\ref{shrisol}) also satisfy the equation (\ref{addshrisol}): 
\[ R(\tilde{g})+|\nabla\tilde{f}|^2-\tilde{f}=0. \]
As a special case of a more general result for complete ancient solutions by Chen \cite{Chen} (cf.  Corollary 2.5), 
we can see that $(N,\tilde{g},\tilde{f})$ must have the nonnegative scalar curvature $R(\tilde{g}) \geq 0$. 
Hence we have 
\[0\leq |\nabla\tilde{f}|^2 \leq \tilde{f}\quad\mathrm{and}\quad 0\leq R(\tilde{g}) \leq \tilde{f}. \]

Fix a positive time $T>0$ arbitrary. 
Let $\{\Phi_{t}:N\to N\}_{t\in(-\infty,T)}$ be the 1-parameter family of diffeomorphisms with $\Phi_{0}=\mathrm{id}_{N}$ generated by 
the time-dependent vector field $V(t):=\frac{1}{T-t}\nabla \tilde{f}$. 
For $t\in(-\infty,T)$, define 
\[g_{t}:=(T-t)\Phi_{t}^{*}\tilde{g}, \quad f_{t}:=\Phi_{t}^{*}\tilde{f}, \quad \rho_{t}:=(4\pi(T-t))^{-\frac{n}{2}}e^{-f_{t}}. \]
Then by the standard calculation, one can prove the following (cf.  \cite{Muller}).

\begin{proposition}\label{basicricciflow}
$g$ is the solution of the Ricci flow, $\frac{\partial g}{\partial t}=-2\mathrm{Ric}$, on the time interval $(-\infty,T)$ with $g_{0}=T\tilde{g}$, 
and $\rho$ and $f$ satisfy the following equations: 
\begin{align}
&\frac{\partial \rho}{\partial t}=-\Delta_{g} \rho +R(g)\rho \label{grasoleq1}\\
&\mathrm{Ric}(g)+\mathop{\mathrm{Hess}}f-\frac{g}{2(T-t)}=0. \label{grasoleq2}\\
&R(g)+|\nabla f|^2-\frac{f}{T-t}=0. \label{grasoleq3}
\end{align}
\end{proposition}

Recall that an immersion map $F:M\to N$ is called a self-similar solution if it satisfies the equation (\ref{selsolinshri}): 
\[H(F)=\lambda\nabla\tilde{f}^{\bot}, \]
and it is called shrinking when $\lambda<0$, 
steady when $\lambda=0$ and expanding when $\lambda>0$. 
A self-similar solution corresponds to a minimal submanifold in a conformal rescaled ambient space. 
The precise statement is the following. 

\begin{proposition}
Let $F:M\to N$ be an immersion map in a gradient shrinking Ricci soliton $(N,\tilde{g},\tilde{f})$. 
Then the following two conditions are equivalent. 
\begin{enumerate}
\item $F$ is a self-similar solution with coefficient $\lambda$. 
\item $F$ is a minimal immersion with respect to a metric $e^{2\lambda\tilde{f}/m}\tilde{g}$ on $N$. 
\end{enumerate}
Here $m$ is the dimension of $M$. 
\end{proposition}
\begin{proof}
One can easily see that in general if we denote the mean curvature vector field of $F$ in $(N,\tilde{g})$ by $H(F)$ 
then the mean curvature vector field in the conformal rescaling $(N,e^{2\varphi}\tilde{g})$ is 
given by 
\[e^{-2\varphi}(H(F)-m\nabla\varphi^{\bot}). \]
Hence, by putting $\varphi:=\lambda\tilde{f}/m$, the equivalence is clear. 
\end{proof}

From a self-shrinker, we can construct a solution of Ricci-mean curvature flow canonically. 
\begin{proposition}
Let $\tilde{F}:M\to N$ be a self-shrinker with $\lambda=-1$. 
For a fixed time $T>0$, let $\Phi_{t}$ and $g_{t}$ be defined as above, and define $\Psi_{t}:=\Phi_{t}^{-1}$. 
Then $F:M\times[0,T)\to N$ defined by $F(p,t):=\Psi_{t}(\tilde{F}(p))$ satisfies 
\[\biggl(\frac{\partial F}{\partial t}\biggr)^{\bot}=H(F_{t}), \]
in the Ricci flow $(N,g_{t})$ defined on $t\in[0,T)$, 
that is, $F$ becomes a solution of the Ricci-mean curvature flow in $(N,g_{t})$ up to a time-dependent re-parametrization of $M$. 
\end{proposition}
\begin{proof}
By differentiating the identity $\Phi_{t}\circ\Psi_{t}=\mathrm{id}_{N}$, we have
\[\frac{1}{T-t}\nabla \tilde{f}+\Phi_{t*}\biggl(\frac{\partial\Psi_{t}}{\partial t}\biggr)=0. \]
Hence we can see that 
\[\frac{\partial\Psi_{t}}{\partial t}=-\Psi_{t*}\biggl(\frac{1}{T-t}\nabla \tilde{f} \biggr). \]
Since $H(\tilde{F})=-\nabla\tilde{f}^{\bot}$, more precisely $H(\tilde{F})=-\nabla\tilde{f}^{\bot_{\tilde{F},\tilde{g}}}$ 
(note that the notion of the normal projection depends on an immersion map and an ambient metric), we have
\begin{align*}
\biggl(\frac{\partial F}{\partial t}\biggr)^{\bot_{F_{t},g_{t}}}=&\biggl(-\Psi_{t*}\biggl(\frac{1}{T-t}\nabla \tilde{f} \biggr)\biggr)^{\bot_{F_{t},g_{t}}}\\
=&-\Psi_{t*}\biggl(\frac{1}{T-t}\nabla \tilde{f}^{\bot_{\tilde{F},\tilde{g}}} \biggr)\\
=&\frac{1}{T-t}\Psi_{t*}(H(\tilde{F}))\\
=&H(F_{t}), 
\end{align*}
where $H(\tilde{F})$ is the mean curvature vector field with respect to the metric $\tilde{g}$ and 
$H(F_{t})$ is the one with respect to the metric $g_{t}$. 
\end{proof}

There exists a one to one correspondence between Ricci-mean curvature flows in $(N,g_{t})$ 
and normalized mean curvature flows (cf. Definition~\ref{defofnormalizedmcf}) in $(N,\tilde{g})$. 
\begin{proposition}\label{corresofselsol}
For a fixed time $T>0$, let $\Phi_{t}$ and $g_{t}$ be defined as above. 
If $F:M\times[0,T)\to N$ is a Ricci-mean curvature flow along the Ricci flow $(N,g_{t})$, 
then the rescaled flow $\tilde{F}:M\times[-\log T,\infty)\to N$ defined by the equation (\ref{scaledmcf2}): 
\[\tilde{F}_{s}:=\Phi_{t}\circ F_{t}\quad\mathrm{with}\quad s=-\log(T-t)\]
for $s\in [-\log T,\infty)$ 
becomes a normalized mean curvature flow in $(N,\tilde{g})$, that is, it satisfies 
\[\frac{\partial \tilde{F}}{\partial s}=H(\tilde{F})+\nabla\tilde{f}. \]
Conversely, if $\tilde{F}:M\times[-\log T,\infty)\to N$ is a normalized mean curvature flow in $(N,\tilde{g})$, 
then the flow $F:M\times[0,T)\to N$ defined by (\ref{scaledmcf2}) becomes a Ricci-mean curvature flow along the Ricci flow $(N,g_{t})$. 
\end{proposition}
\begin{proof}
By differentiating $\tilde{F}$, we have
\[\frac{\partial \tilde{F}}{\partial s}=\nabla\tilde{f}+(T-t)\Phi_{t*}\biggl(\frac{\partial F}{\partial t}\biggr). \]
Furthermore, it is clear that 
\[(T-t)\Phi_{t*}(H(F_{t}))=H(\tilde{F}_{s}). \]
Hence, the correspondence between Ricci-mean curvature flows along $(N,g_{t})$ and normalized mean curvature flows in $(N,\tilde{g})$ is clear. 
\end{proof}

Here the monotonicity formula for a Ricci-mean curvature flow moving along the Ricci flow $(N,g_{t})$ is almost clear by Proposition~\ref{genmonoRM}. 
\begin{proposition}\label{monoformingra}
For a fixed time $T>0$, let $g_{t}$, $f_{t}$, and $\rho_{t}$ be defined as above, and 
define $u_{t}:=(4\pi(T-t))^{\frac{n-m}{2}}$. 
If $F:M\times[0,T)\to N$ is a Ricci-mean curvature flow along the Ricci flow $(N,g_{t})$ and $M$ is compact, 
then we have the monotonicity formula: 
\begin{align}\label{monoformingraform}
\frac{d}{dt} \int_{M}u\mathop{F^{*}\hspace{-1.2mm}\rho} \mathop{d\mu(F^{*}g)}=-\int_{M} u\Bigl|H(F)+\nabla f^{\bot_{F}}\Bigr|_{g}^2 \mathop{F^{*}\hspace{-1.2mm}\rho} d\mu(F^{*}g)\leq 0. 
\end{align}
\end{proposition}
\begin{proof}
By Proposition~\ref{basicricciflow}, we see that $\rho$ satisfies the conjugate heat equation (\ref{evoeq23}).  
To see that $u$ satisfies the equation (\ref{evoeq24}), we use the equivalent equation (\ref{evoeq24'}). 
In this case, by Proposition~\ref{basicricciflow}, the equation (\ref{evoeq24'}) becomes 
\[\frac{\partial v}{\partial t}=\Delta_{F^{*}g}v, \]
the standard heat equation on $M$, where $u$ and $v$ are related by $u=(4\pi(T-t))^{\frac{n-m}{2}}v$. 
Then $v\equiv 1$ is a trivial solution of (\ref{evoeq24'}). 
Hence $u_{t}=(4\pi(T-t))^{\frac{n-m}{2}}$ becomes a solution of  (\ref{evoeq24}). 
Thus, by Proposition~\ref{genmonoRM}, we have the above monotonicity formula (\ref{monoformingraform}).  
\end{proof}

By Proposition~\ref{monoformingra}, we can deduce the following monotonicity formula of the weighted volume functional for a normalized mean curvature flow, immediately. 
\begin{proposition}\label{monoformingra2}
If $\tilde{F}:M\times[-\log T,\infty)\to N$ is a normalized mean curvature flow in $(N,\tilde{g},\tilde{f})$ and $M$ is compact, 
then we have the monotonicity formula: 
\begin{align}\label{monoformingraform2}
\frac{d}{ds} \int_{M}e^{-\tilde{f}\circ\tilde{F}} \mathop{d\mu(\tilde{F}^{*}\tilde{g})}=-\int_{M} \Bigl|H(\tilde{F})+\nabla \tilde{f}^{\bot_{\tilde{F}}}\Bigr|_{\tilde{g}}^2 e^{-\tilde{f}\circ\tilde{F}} d\mu(\tilde{F}^{*}\tilde{g})\leq 0. 
\end{align}
\end{proposition}
\begin{proof}
In this proof, we follow the notations in Proposition~\ref{corresofselsol}. 
It is clear that $f_{t}\circ F_{t}=\tilde{f}\circ \tilde{F}_{s}$ and $F_{t}^{*}g_{t}=(T-t)\tilde{F}^{*}_{s}\tilde{g}$. 
Hence we have 
\begin{align*}
u_{t}\mathop{F_{t}^{*}\hspace{-0.8mm}\rho_{t}}  \mathop{d\mu(F_{t}^{*}g_{t})}=&(4\pi(T-t))^{-\frac{m}{2}}e^{-f_{t}\circ F_{t}}\mathop{d\mu(F_{t}^{*}g_{t})}\\
=&(4\pi)^{-\frac{m}{2}}e^{-\tilde{f}\circ \tilde{F}_{s}}\mathop{d\mu(\tilde{F}_{s}^{*}\tilde{g})}. 
\end{align*}
Since $H(\tilde{F}_{s})=(T-t)\Phi_{t*}H(F_{t})$ and $\nabla\tilde{f}=(T-t)\Phi_{t*}\nabla f_{t}$, we have 
\[(T-t)\Bigl|H(F_{t})+\nabla f_{t}^{\bot_{F_{t}}}\Bigr|_{g_{t}}^2=\Bigl|H(\tilde{F}_{s})+\nabla \tilde{f}^{\bot_{\tilde{F}_{s}}}\Bigr|_{\tilde{g}}^2. \]
Thus, by the equality (\ref{monoformingraform}), one can easily see that the equality (\ref{monoformingraform2}) holds. 
\end{proof}

To prove the main theorems, we need the following key lemma. 
Its proof is an analog of the proof of Stone's estimate (cf. Lemma~2.9 in \cite{Stone}). 
Stone considered the weight $e^{-\sqrt{\tilde{f}}}$ in the Euclidean case, where $\tilde{f}:=|x|^2/4$. 
However we consider the weight $e^{-\frac{\tilde{f}}{2}}$ here, since $-\frac{\tilde{f}}{2}$ is a smooth function and we can apply Proposition~\ref{vari}. 

\begin{lemma}\label{stonelem}
Assume that $(N,\tilde{g})$ has bounded geometry. 
If $\tilde{F}:M\times[-\log T,\infty)\to N$ is a normalized mean curvature flow in $(N,\tilde{g},\tilde{f})$ and $M$ is compact, 
then there exists a constant $C>0$ such that 
\begin{align}\label{stonebound}
\int_{M}e^{-\frac{\tilde{f}}{2}\circ\tilde{F}} \mathop{d\mu(\tilde{F}^{*}\tilde{g})}\le C. 
\end{align}
uniformly on $[-\log T,\infty)$. 
\end{lemma}
\begin{proof}
In this proof, we follow the notations in Proposition~\ref{corresofselsol}. 
As the proof of Proposition~\ref{monoformingra2}, we have 
\[\int_{M}e^{-\frac{\tilde{f}}{2}\circ\tilde{F}_{s}}\mathop{d\mu(\tilde{F}^{*}_{s}\tilde{g})}
=(4\pi)^{\frac{m}{2}}\int_{M}u_{t}\mathop{F_{t}^{*}\hspace{-0.5mm}\bar{\rho}_{t}} \mathop{d\mu(F_{t}^{*}g_{t})}, \]
where 
\[\bar{\rho}_{t}:=\frac{1}{(4\pi(T-t))^{\frac{n}{2}}}e^{-\frac{f_{t}}{2}}\quad\mathrm{and}\quad u_{t}:=(4\pi(T-t))^{\frac{n-m}{2}}. \]
By Proposition~\ref{vari}, we have 
\begin{align*}
&\frac{d}{dt} \int_{M}u_{t}\mathop{F_{t}^{*}\hspace{-0.5mm}\bar{\rho}_{t}} \mathop{d\mu(F_{t}^{*}g_{t})}\\
=&-\int_{M}u_{t}\Bigl|H(F_{t})+\frac{1}{2}\nabla f_{t}^{\bot_{F_{t}}}\Bigr|_{g_{t}}^2\mathop{F_{t}^{*}\hspace{-0.5mm}\bar{\rho}_{t}} \mathop{d\mu(F_{t}^{*}g_{t})}\\
&+\int_{M}u_{t} F_{t}^{*}\biggl(\frac{\partial \bar{\rho}_{t}}{\partial t}+\Delta_{g_{t}}\bar{\rho}_{t}-R(g_{t})\bar{\rho}_{t}\biggr)  \mathop{d\mu(F_{t}^{*}g_{t})}\\
& +\int_{M}\left( \frac{\partial u_{t}}{\partial t} -\Delta_{F_{t}^{*}g_{t}}u_{t}+
u_{t}\mathop{\mathrm{tr}^{\bot_{F_{t}}}}\left(\frac{1}{2}\mathop{\mathrm{Hess}}f_{t}+\mathrm{Ric}(g_{t})\right)\right)\mathop{F_{t}^{*}\hspace{-0.5mm}\bar{\rho}_{t}} \mathop{d\mu(F_{t}^{*}g_{t})}. 
\end{align*}
By using $\frac{\partial f}{\partial t}=|\nabla f|^{2}$ and $|\nabla f|^{2}=\frac{f}{T-t}-R(g)$, we have 
\[\frac{\partial \bar{\rho}}{\partial t}=\bar{\rho}\biggl(\frac{n}{2(T-t)}-\frac{f}{2(T-t)}+\frac{1}{2}R(g) \biggr). \]
By using $\Delta_{g} f=-R(g)+\frac{n}{2(T-t)}$ and $|\nabla f|^{2}=\frac{f}{T-t}-R(g)$, we have 
\[\Delta_{g}\bar{\rho}=\bar{\rho}\biggl(\frac{f}{4(T-t)}+\frac{1}{4}R(g)-\frac{n}{4(T-t)} \biggr). \]
Hence we have 
\[\frac{\partial \bar{\rho}}{\partial t}+\Delta_{g}\bar{\rho}-R(g)\bar{\rho}=\bar{\rho}\biggl(\frac{n}{4(T-t)}-\frac{f}{4(T-t)}-\frac{1}{4}R(g)\biggr) \leq \frac{\bar{\rho}}{4(T-t)}(n-f). \]
Furthermore, since $u$ satisfies 
\[\frac{\partial u_{t}}{\partial t} -\Delta_{F_{t}^{*}g_{t}}u_{t}+u_{t}\mathop{\mathrm{tr}^{\bot_{F_{t}}}}\left(\mathop{\mathrm{Hess}}f_{t}+\mathrm{Ric}(g_{t})\right)=0, \]
we have
\[\frac{\partial u_{t}}{\partial t} -\Delta_{F_{t}^{*}g_{t}}u_{t}+u_{t}\mathop{\mathrm{tr}^{\bot_{F_{t}}}}\left(\frac{1}{2}\mathop{\mathrm{Hess}}f_{t}+\mathrm{Ric}(g_{t})\right)
=-\frac{1}{2}u_{t}\mathop{\mathrm{tr}^{\bot_{F_{t}}}}\mathop{\mathrm{Hess}}f_{t}.\]
By using $\mathop{\mathrm{Hess}}f_{t}=\frac{1}{2(T-t)}g_{t}-\mathrm{Ric}(g_{t})$, we have 
\[-\frac{1}{2}u_{t}\mathop{\mathrm{tr}^{\bot_{F_{t}}}}\mathop{\mathrm{Hess}}f_{t}=u_{t}\left(-\frac{n-m}{4(T-t)}+\frac{1}{2}\mathop{\mathrm{tr}^{\bot_{F_{t}}}}\mathrm{Ric}(g_{t})\right). \]
It is clear that 
\[\mathop{\mathrm{tr}^{\bot_{F_{t}}}}\mathrm{Ric}(g_{t})\leq (n-m)|\mathrm{Ric}(g_{t})|_{g_{t}}=(n-m)\frac{|\mathrm{Ric}(\tilde{g})|_{\tilde{g}}}{T-t}\leq \frac{C''}{T-t}, \]
where $C'':=(n-m)\max_{N}|\mathrm{Ric}(\tilde{g})|_{\tilde{g}}$ is a bounded constant since $(N,\tilde{g})$ has bounded geometry. 
Hence we have
\begin{align*}
\frac{d}{dt} \int_{M}u_{t}\mathop{F_{t}^{*}\hspace{-0.5mm}\bar{\rho}_{t}} \mathop{d\mu(F_{t}^{*}g_{t})}
< \frac{1}{4(T-t)}\int_{M}\Bigl(C_{0}-f_{t}\circ F_{t}\Bigr)u_{t} \mathop{F_{t}^{*}\hspace{-0.5mm}\bar{\rho}_{t}} \mathop{d\mu(F_{t}^{*}g_{t})}, 
\end{align*}
where $C_{0}:=m+4C''+1$. Since $s=-\log(T-t)$, we have 
\begin{align*}
\frac{d}{ds}\int_{M}e^{-\frac{\tilde{f}}{2}\circ\tilde{F}_{s}}\mathop{d\mu(\tilde{F}^{*}_{s}\tilde{g})}
=(4\pi)^{\frac{m}{2}}(T-t)\frac{d}{dt}\int_{M}u_{t}\mathop{F_{t}^{*}\hspace{-0.5mm}\bar{\rho}_{t}} \mathop{d\mu(F_{t}^{*}g_{t})}. 
\end{align*}
Hence we have
\begin{align*}
\frac{d}{ds}\int_{M}e^{-\frac{\tilde{f}}{2}\circ\tilde{F}_{s}}\mathop{d\mu(\tilde{F}^{*}_{s}\tilde{g})}
<\frac{1}{4}\int_{M}\Bigl(C_{0}-\tilde{f}\circ \tilde{F}_{s}\Bigr)e^{-\frac{\tilde{f}}{2}\circ\tilde{F}_{s}}\mathop{d\mu(\tilde{F}^{*}_{s}\tilde{g})}. 
\end{align*}
Here we divide $M$ into time-dependent three pieces as follows: 
\begin{align*}
&M_{1,s}:=\tilde{F}_{s}^{-1}(\{\tilde{f}\leq C_{0}\}), \\
&M_{2,s}:=\tilde{F}_{s}^{-1}(\{C_{0}< \tilde{f} \leq 2C_{0}\}), \\
&M_{3,s}:=\tilde{F}_{s}^{-1}(\{2C_{0}< \tilde{f}\}). 
\end{align*}
On each component, we have 
\begin{align*}
&\int_{M_{1,s}}\Bigl(C_{0}-\tilde{f}\circ \tilde{F}_{s}\Bigr)e^{-\frac{\tilde{f}}{2}\circ\tilde{F}_{s}}\mathop{d\mu(\tilde{F}^{*}_{s}\tilde{g})}
\leq C_{0}\int_{M_{1,s}}e^{-\frac{\tilde{f}}{2}\circ\tilde{F}_{s}}\mathop{d\mu(\tilde{F}^{*}_{s}\tilde{g})}, \\
&\int_{M_{2,s}}\Bigl(C_{0}-\tilde{f}\circ \tilde{F}_{s}\Bigr)e^{-\frac{\tilde{f}}{2}\circ\tilde{F}_{s}}\mathop{d\mu(\tilde{F}^{*}_{s}\tilde{g})}\leq 0\\
&\int_{M_{3,s}}\Bigl(C_{0}-\tilde{f}\circ \tilde{F}_{s}\Bigr)e^{-\frac{\tilde{f}}{2}\circ\tilde{F}_{s}}\mathop{d\mu(\tilde{F}^{*}_{s}\tilde{g})}
\leq -C_{0}\int_{M_{3,s}}e^{-\frac{\tilde{f}}{2}\circ\tilde{F}_{s}}\mathop{d\mu(\tilde{F}^{*}_{s}\tilde{g})}. 
\end{align*}
Thus we have 
\begin{align}\label{altineq}
\begin{aligned}
&\frac{d}{ds}\int_{M}e^{-\frac{\tilde{f}}{2}\circ\tilde{F}_{s}}\mathop{d\mu(\tilde{F}^{*}_{s}\tilde{g})}\\
<&\frac{C_{0}}{4}\biggl( \int_{M_{1,s}}e^{-\frac{\tilde{f}}{2}\circ\tilde{F}_{s}}\mathop{d\mu(\tilde{F}^{*}_{s}\tilde{g})} - \int_{M_{3,s}}e^{-\frac{\tilde{f}}{2}\circ\tilde{F}_{s}}\mathop{d\mu(\tilde{F}^{*}_{s}\tilde{g})} \biggr). 
\end{aligned}
\end{align}
On the other hand, by the monotonicity formula (cf. Proposition~\ref{monoformingra2}), we have 
\[\int_{M}e^{-\tilde{f}\circ\tilde{F}_{s}} \mathop{d\mu(\tilde{F}_{s}^{*}\tilde{g})}\leq  C', \]
where $C'$ is the value of the left hand side at the initial time $s=-\log T$. 
We further define a region in $M$ by 
\[M_{4,s}:=\tilde{F}_{s}^{-1}(\{\tilde{f}\leq 2C_{0}\})=M_{1,s}\cup M_{2,s}. \]
Since $e^{-\frac{\tilde{f}}{2}}=e^{\frac{\tilde{f}}{2}}e^{-\tilde{f}}\leq e^{\frac{C_{0}}{2}}e^{-\tilde{f}}$ on $M_{1,s}$, we have 
\[\int_{M_{1,s}}e^{-\frac{\tilde{f}}{2}\circ\tilde{F}_{s}}\mathop{d\mu(\tilde{F}^{*}_{s}\tilde{g})}\leq e^{\frac{C_{0}}{2}}\int_{M}e^{-\tilde{f}\circ\tilde{F}_{s}}\mathop{d\mu(\tilde{F}^{*}_{s}\tilde{g})}
\leq e^{\frac{C_{0}}{2}}C'=:C_{1}. \]
As on $M_{1,s}$, we have
\begin{align}\label{M4s}
\int_{M_{4,s}}e^{-\frac{\tilde{f}}{2}\circ\tilde{F}_{s}}\mathop{d\mu(\tilde{F}^{*}_{s}\tilde{g})}\leq e^{C_{0}}C'=:C_{2}. 
\end{align}
Hence, by the inequality (\ref{altineq}), we see that for each $s\in[-\log T,\infty)$ we must have either 
\[\frac{d}{ds}\int_{M}e^{-\frac{\tilde{f}}{2}\circ\tilde{F}_{s}}\mathop{d\mu(\tilde{F}^{*}_{s}\tilde{g})}<0\quad\mathrm{or}\quad 
\int_{M_{3,s}}e^{-\frac{\tilde{f}}{2}\circ\tilde{F}_{s}}\mathop{d\mu(\tilde{F}^{*}_{s}\tilde{g})}\leq C_{1}. \]
Since $M=M_{3,s}\cup M_{4,s}$ and we have the bound (\ref{M4s}), 
we see that for each $s\in[-\log T,\infty)$ we must have either 
\[\frac{d}{ds}\int_{M}e^{-\frac{\tilde{f}}{2}\circ\tilde{F}_{s}}\mathop{d\mu(\tilde{F}^{*}_{s}\tilde{g})}<0\quad\mathrm{or}\quad 
\int_{M}e^{-\frac{\tilde{f}}{2}\circ\tilde{F}_{s}}\mathop{d\mu(\tilde{F}^{*}_{s}\tilde{g})}\leq C_{1}+C_{2}. \]
This condition implies that 
\[\int_{M}e^{-\frac{\tilde{f}}{2}\circ\tilde{F}_{s}}\mathop{d\mu(\tilde{F}^{*}_{s}\tilde{g})}\leq \max\{C_{1}+C_{2},C_{3}\}=:C, \]
where $C_{3}$ is the value of the left hand side at the initial time $s=-\log T$. 
\end{proof}

\begin{remark}\label{A2andC0}
In Section~\ref{Intro}, we consider the condition (A2): 
\[\limsup_{t\to T}(\sqrt{T-t}\sup_{M}|A(F_{t})|_{g_{t}})<\infty, \]
for a Ricci-mean curvature flow $F:M\times[0,T)\to N$ along the Ricci flow $g_{t}$. 
Note that if $M$ is compact then this condition is equivalent to that there exists a constant $C_{0}>0$ such that 
\[\max_{M}|A(F_{t})|_{g_{t}}\leq \frac{C_{0}}{\sqrt{T-t}} \quad\mathrm{on}\quad[0,T). \]
\end{remark}

\begin{proposition}\label{bdofallderiA}
Let $(N,\tilde{g},\tilde{f})$ be a gradient shrinking Ricci soliton with bounded geometry. 
For a fixed time $T>0$, let $\Phi_{t}$ and $g_{t}$ be defined as above, 
and let $F:M\times[0,T)\to N$ be a Ricci-mean curvature flow along the Ricci flow $(N,g_{t})$. 
Assume that $M$ is compact and $F$ satisfies the condition (A2). 
Let $\tilde{F}$ be the normalized mean curvature flow defined by (\ref{scaledmcf2}). 
Then, for all $k=0,1,2,\dots$, there exist constants $C_{k}>0$ such that 
\[|\tilde{\nabla}^{k}A(\tilde{F_{s}})|_{\tilde{g}}\leq C_{k} \quad\mathrm{on}\quad M\times[-\log T,\infty), \]
where $\tilde{\nabla}$ is the connection defined by the Levi--Civita connection on $(N,\tilde{g})$ and the one on $(M,\tilde{F}_{s}^{*}\tilde{g})$. 
\end{proposition}
\begin{proof}
First of all, by the definitions of $g_{t}=(T-t)\Phi_{t}^{*}\tilde{g}$ and $\tilde{F}_{s}=\Phi_{t}\circ F_{t}$, one can easily see that 
\begin{align}
&\tilde{\nabla}^{k}A(\tilde{F}_{s})=\Phi_{t*}\nabla ^{k}A(F_{t}), \ |\tilde{\nabla}^{k}A(\tilde{F}_{s})|_{\tilde{g}}=(T-t)^{\frac{1}{2}+\frac{1}{2}k}|\nabla ^{k}A(F_{t})|_{g_{t}}, \label{orderchange1}\\ 
&|\tilde{\nabla}^{k}\mathrm{Rm}(\tilde{g})|_{\tilde{g}}=(T-t)^{1+\frac{1}{2}k}|\nabla^{k}\mathrm{Rm}(g_{t})|_{g_{t}}  \label{orderchange2}
\end{align}
for all $k=0,1,2,\dots$, 
where $\tilde{\nabla}$ is the connection defined by the Levi--Civita connection on $(N,\tilde{g})$ and the one on $(M,\tilde{F}_{s}^{*}\tilde{g})$, 
and $\nabla$ is the connection defined by the Levi--Civita connection on $(N,g_{t})$ and the one on $(M,F_{t}^{*}g_{t})$. 
In this sense, as Huisken done in \cite{Huisken}, 
we can consider the degree of $\nabla ^{k}A(F_{t})$ is $\frac{1}{2}+\frac{1}{2}k$ and the degree of $\nabla^{k}\mathrm{Rm}(g_{t})$ is $1+\frac{1}{2}k$. 
We will write $A(F_{t})$ and $A(\tilde{F}_{s})$ by $A$ and $\tilde{A}$ respectively, 
and also write $\mathrm{Rm}(g_{t})$ and $\mathrm{Rm}(\tilde{g})$ by $\mathrm{Rm}$ and $\widetilde{\mathrm{Rm}}$ respectively, for short. 
To use the argument of degree more rigorously, we define a set $V_{a,b}$ and a vector space $\mathcal{V}_{a,b}$ as follows. 
First, we recall the notion of $\ast$-product here. 
For tensors $T_{1}$ and $T_{2}$, we write $T_{1}\ast T_{2}$ to mean a tensor formed by a sum of terms each one of them obtained by contracting some indices of the pair $T_{1}$ and $T_{2}$ 
by using $g$, $F^{*}g$ and these inverses, 
and there is a property that 
\begin{align}\label{starprod}
|T_{1}\ast T_{2}|\leq C|T_{1}||T_{2}|, 
\end{align}
where $C>0$ is a constant which depends only on the algebraic structure of $T_{1}\ast T_{2}$. 
Then, for $a,b\in\mathbb{N}$, we define a set  $V_{a,b}$ as the set of all (time-dependent) tensors $T$ on $M$ which can be expressed as 
\[T=(\nabla^{k_{1}}\mathrm{Rm}\ast\dots\ast\nabla^{k_{I}}\mathrm{Rm})\ast(\nabla^{\ell_{1}}A\ast\dots\ast\nabla^{\ell_{J}}A)\ast (\mathop{\ast}^{p}DF)\]
with $I,J,p,k_{1},\dots,k_{I},\ell_{1},\dots,\ell_{J}\in\mathbb{N}$ satisfying 
\begin{align*}
\begin{aligned}
\sum_{i=1}^{I}\biggl(1+\frac{1}{2}k_{i}\biggr)+\sum_{j=1}^{J}\biggl(\frac{1}{2}+\frac{1}{2}\ell_{j}\biggr)=a\quad\mathrm{and}\quad\sum_{j=1}^{J}\ell_{j}\leq b, 
\end{aligned}
\end{align*}
and we define a vector space $\mathcal{V}_{a,b}$ as the set of all tensors $T$ on $M$ which can be expressed as 
\[T=a_{1}T_{1}+\dots+a_{r}T_{r}\]
for some $r\in\mathbb{N}$, $a_{1}\dots a_{r}\in\mathbb{R}$ and $T_{1},\dots,T_{r}\in V_{a,b}$. 

For the case $k=0$, as noted in Remark~\ref{A2andC0}, there exists a constant $C_{0}>0$ such that 
\[|A|\leq \frac{C_{0}}{\sqrt{T-t}} \quad\mathrm{on}\quad M\times[0,T), \]
since $F$ satisfies the condition (A2). 
Hence we have 
\[|\tilde{A}|=\sqrt{T-t}|A|\leq C_{0}. \]

For the case $k\geq 0$, we work by induction on $k\in\mathbb{N}$. 
The case $k=0$ has already proved above. 
For a fixed $k\geq 1$, assume that there exist positive constants $C_{0},C_{1}, \dots, C_{k-1}$ such that 
\[|\tilde{\nabla}^{i}\tilde{A}|\leq C_{i} \quad\mathrm{on}\quad M\times[-\log T,\infty)\]
for $i=0,1,\dots,k-1$. 
We consider the evolution equation of $|\tilde{\nabla}^{k}\tilde{A}|^{2}$, and finally we will prove the bound of $|\tilde{\nabla}^{k}\tilde{A}|^{2}$ by the parabolic maximum principle. 
Since $|\tilde{\nabla}^{k}\tilde{A}|^2=(T-t)^{k+1}|\nabla ^{k}A|^2$ and $\frac{\partial}{\partial s}=(T-t)\frac{\partial}{\partial t}$, 
we have that 
\begin{align*}
\frac{\partial}{\partial s}|\tilde{\nabla}^{k}\tilde{A}|^{2}&=-(k+1)|\tilde{\nabla}^{k}\tilde{A}|^{2}+(T-t)^{k+2}\frac{\partial}{\partial t}|\nabla ^{k}A|^2\\
&\leq (T-t)^{k+2}\frac{\partial}{\partial t}|\nabla ^{k}A|^2. 
\end{align*}
By Proposition~\ref{evofhiA}, there exist tensors $\mathcal{E}[k]\in\mathcal{V}_{\frac{3}{2}+\frac{1}{2}k,k}$, $\mathcal{C}[k]\in\mathcal{V}_{\frac{3}{2}+\frac{1}{2}k,k+1}$ and 
$\mathcal{G}[k]\in\mathcal{V}_{\frac{1}{2}+\frac{1}{2}k,k-1}$ such that 
\[\frac{\partial}{\partial t}|\nabla^{k}A|^2=\Delta |\nabla^{k}A|^2-2|\nabla^{k+1}A|^2+\mathcal{E}[k]\ast \nabla^{k}A+\mathcal{C}[k]\ast \mathcal{G}[k], \]
where  $\Delta$ is the Laplacian on $(M,F_{t}^{*}g_{t})$. 
Let $\tilde{\Delta}$ be the Laplacian on $(M,\tilde{F}_{s}^{*}\tilde{g})$, then we have $(T-t)\Delta=\tilde{\Delta}$. 
Hence we have
\[(T-t)^{k+2}(\Delta |\nabla^{k}A|^2-2|\nabla^{k+1}A|^2)=\tilde{\Delta} |\tilde{\nabla}^{k}\tilde{A}|^2-2|\tilde{\nabla}^{k+1}\tilde{A}|^2. \]

Since $\mathcal{G}[k]\in\mathcal{V}_{\frac{1}{2}+\frac{1}{2}k,k-1}$, 
there exist $r\in\mathbb{N}$, $a_{1}\dots a_{r}\in\mathbb{R}$ and 
\[\mathcal{G}[k]_{1},\dots,\mathcal{G}[k]_{r}\in V_{\frac{1}{2}+\frac{1}{2}k,k-1}\] 
such that 
\[\mathcal{G}[k]=a_{1}\mathcal{G}[k]_{1}+\dots+a_{r}\mathcal{G}[k]_{r}. \]
Hence we have 
\[|\mathcal{G}[k]|\leq |a_{1}||\mathcal{G}[k]_{1}|+\dots+|a_{r}||\mathcal{G}[k]_{r}|. \]
By the definition of $V_{\frac{1}{2}+\frac{1}{2}k,k-1}$, each $\mathcal{G}[k]_{\bullet}$ can be expressed as 
\[(\nabla^{k_{1}}\mathrm{Rm}\ast\dots\ast\nabla^{k_{I}}\mathrm{Rm})\ast(\nabla^{\ell_{1}}A\ast\dots\ast\nabla^{\ell_{J}}A)\ast (\mathop{\ast}^{p}DF)\]
with some $I,J,p,k_{1},\dots,k_{I},\ell_{1},\dots,\ell_{J}\in\mathbb{N}$ satisfying 
\begin{align*}
\begin{aligned}
\sum_{i=1}^{I}\biggl(1+\frac{1}{2}k_{i}\biggr)+\sum_{j=1}^{J}\biggl(\frac{1}{2}+\frac{1}{2}\ell_{j}\biggr)=\frac{1}{2}+\frac{1}{2}k\quad\mathrm{and}\quad\sum_{j=1}^{J}\ell_{j}\leq k-1. 
\end{aligned}
\end{align*}
Hence, by using (\ref{orderchange1}), $(\ref{orderchange2})$, and (\ref{starprod}), we have 
\begin{align*}
&(T-t)^{\frac{1}{2}+\frac{1}{2}k}|\mathcal{G}[k]_{\bullet}|\\
\leq&C(T-t)^{\frac{1}{2}+\frac{1}{2}k}|\nabla^{k_{1}}\mathrm{Rm}|\cdots|\nabla^{k_{I}}\mathrm{Rm}||\nabla^{\ell_{1}}A|\cdots|\nabla^{\ell_{J}}A||DF|^p\\
=&C(\sqrt{m})^{p} |\tilde{\nabla}^{k_{1}}\widetilde{\mathrm{Rm}}|\cdots|\tilde{\nabla}^{k_{I}}\widetilde{\mathrm{Rm}}||\tilde{\nabla}^{\ell_{1}}\tilde{A}|\cdots|\tilde{\nabla}^{\ell_{J}}\tilde{A}| 
\end{align*}
for some constant $C>0$. 
Here note that $|DF|=\sqrt{m}$. 
Since $(N,\tilde{g})$ has bounded geometry, each $|\tilde{\nabla}^{k_{i}}\widetilde{\mathrm{Rm}}|$ is bounded. 
Furthermore, since $\ell_{j}\leq k-1$, each $|\tilde{\nabla}^{\ell_{j}}\tilde{A}|$ is bounded by the assumption of induction. 
Hence there exists a constant $C'>0$ such that 
\[(T-t)^{\frac{1}{2}+\frac{1}{2}k}|\mathcal{G}[k]|\leq C'. \]

Since $\mathcal{E}[k]\in\mathcal{V}_{\frac{3}{2}+\frac{1}{2}k,k}$, 
there exist $r'\in\mathbb{N}$, $b_{1}\dots b_{r'}\in\mathbb{R}$ and 
\[\mathcal{E}[k]_{1},\dots,\mathcal{E}[k]_{r'}\in V_{\frac{3}{2}+\frac{1}{2}k,k}\] 
such that 
\[\mathcal{E}[k]=b_{1}\mathcal{E}[k]_{1}+\dots+b_{r'}\mathcal{E}[k]_{r'}. \]
Hence we have 
\[|\mathcal{E}[k]|\leq |b_{1}||\mathcal{E}[k]_{1}|+\dots+|b_{r'}||\mathcal{E}[k]_{r'}|. \]
By the definition of $V_{\frac{3}{2}+\frac{1}{2}k,k}$, each $\mathcal{E}[k]_{\bullet}$ can be expressed as 
\[(\nabla^{k_{1}}\mathrm{Rm}\ast\dots\ast\nabla^{k_{I}}\mathrm{Rm})\ast(\nabla^{\ell_{1}}A\ast\dots\ast\nabla^{\ell_{J}}A)\ast (\mathop{\ast}^{p}DF)\]
with some $I,J,p,k_{1},\dots,k_{I},\ell_{1},\dots,\ell_{J}\in\mathbb{N}$ satisfying 
\begin{align*}
\begin{aligned}
\sum_{i=1}^{I}\biggl(1+\frac{1}{2}k_{i}\biggr)+\sum_{j=1}^{J}\biggl(\frac{1}{2}+\frac{1}{2}\ell_{j}\biggr)=\frac{3}{2}+\frac{1}{2}k\quad\mathrm{and}\quad\sum_{j=1}^{J}\ell_{j}\leq k. 
\end{aligned}
\end{align*}
If $\max\{\ell_{1},\dots,\ell_{J}\}\leq k-1$, we can prove that $(T-t)^{\frac{3}{2}+\frac{1}{2}k}|\mathcal{E}[k]_{\bullet}|$ is bounded by the same argument as the case of $\mathcal{G}[k]_{\bullet}$. 
If $\max\{\ell_{1},\dots,\ell_{J}\}=k$, one can easily see that the possible forms of $\mathcal{E}[k]_{\bullet}$ are 
\[A\ast A\ast \nabla^{k}A\ast (\mathop{\ast}^{p}DF)\quad\mathrm{and}\quad \mathrm{Rm}\ast \nabla^{k}A\ast (\mathop{\ast}^{p}DF). \]
In both cases, we can see by the same argument as the case of $\mathcal{G}[k]_{\bullet}$ that there exists a constant $\tilde{C}>0$ such that 
$(T-t)^{\frac{3}{2}+\frac{1}{2}k}|\mathcal{E}[k]_{\bullet}|\leq \tilde{C} |\tilde{\nabla}^{k}\tilde{A}|$. 
Hence we can see that there exists a constant $C''>0$ such that 
\[(T-t)^{\frac{3}{2}+\frac{1}{2}k}|\mathcal{E}[k]|\leq C''(1+ |\tilde{\nabla}^{k}\tilde{A}|). \]

Since $\mathcal{C}[k]\in\mathcal{V}_{\frac{3}{2}+\frac{1}{2}k,k+1}$, 
there exist $r''\in\mathbb{N}$, $c_{1}\dots c_{r''}\in\mathbb{R}$ and 
\[\mathcal{C}[k]_{1},\dots,\mathcal{C}[k]_{r''}\in V_{\frac{3}{2}+\frac{1}{2}k,k+1}\] 
such that 
\[\mathcal{C}[k]=c_{1}\mathcal{C}[k]_{1}+\dots+c_{r''}\mathcal{C}[k]_{r''}. \]
Hence we have 
\[|\mathcal{C}[k]|\leq |c_{1}||\mathcal{C}[k]_{1}|+\dots+|c_{r''}||\mathcal{C}[k]_{r''}|. \]
By the definition of $V_{\frac{3}{2}+\frac{1}{2}k,k+1}$, each $\mathcal{C}[k]_{\bullet}$ can be expressed as 
\[(\nabla^{k_{1}}\mathrm{Rm}\ast\dots\ast\nabla^{k_{I}}\mathrm{Rm})\ast(\nabla^{\ell_{1}}A\ast\dots\ast\nabla^{\ell_{J}}A)\ast (\mathop{\ast}^{p}DF)\]
with some $I,J,p,k_{1},\dots,k_{I},\ell_{1},\dots,\ell_{J}\in\mathbb{N}$ satisfying 
\begin{align*}
\begin{aligned}
\sum_{i=1}^{I}\biggl(1+\frac{1}{2}k_{i}\biggr)+\sum_{j=1}^{J}\biggl(\frac{1}{2}+\frac{1}{2}\ell_{j}\biggr)=\frac{3}{2}+\frac{1}{2}k\quad\mathrm{and}\quad\sum_{j=1}^{J}\ell_{j}\leq k+1. 
\end{aligned}
\end{align*}
If $\max\{\ell_{1},\dots,\ell_{J}\}\leq k-1$, we can prove that $(T-t)^{\frac{3}{2}+\frac{1}{2}k}|\mathcal{C}[k]_{\bullet}|$ is bounded by the same argument as the case of $\mathcal{G}[k]_{\bullet}$. 
If $\max\{\ell_{1},\dots,\ell_{J}\}=k$, one can easily see that the possible forms of $\mathcal{C}[k]_{\bullet}$ are 
\[A\ast A\ast \nabla^{k}A\ast (\mathop{\ast}^{p}DF)\quad\mathrm{and}\quad \mathrm{Rm}\ast \nabla^{k}A\ast (\mathop{\ast}^{p}DF), \]
and we have $(T-t)^{\frac{3}{2}+\frac{1}{2}k}|\mathcal{C}[k]_{\bullet}|\leq \tilde{C} |\tilde{\nabla}^{k}\tilde{A}|$ as the case of $\mathcal{E}[k]_{\bullet}$. 
If $\max\{\ell_{1},\dots,\ell_{J}\}=k+1$, one can easily see that the possible form of $\mathcal{C}[k]_{\bullet}$ is 
\[\nabla^{k+1}A\ast (\mathop{\ast}^{p}DF), \]
and we have $(T-t)^{\frac{3}{2}+\frac{1}{2}k}|\mathcal{C}[k]_{\bullet}|\leq \tilde{C}' |\tilde{\nabla}^{k+1}\tilde{A}|$ for some constant $\tilde{C}' >0$. 
Hence we can see that there exists a constant $C'''>0$ such that 
\[(T-t)^{\frac{3}{2}+\frac{1}{2}k}|\mathcal{C}[k]|\leq C'''(1+ |\tilde{\nabla}^{k}A|+|\tilde{\nabla}^{k+1}\tilde{A}|). \]

Hence we have 
\begin{align*}
\frac{\partial}{\partial s}|\tilde{\nabla}^{k}\tilde{A}|^{2}\leq& (T-t)^{k+2}\frac{\partial}{\partial t}|\nabla ^{k}A|^2\\
\leq& \tilde{\Delta} |\tilde{\nabla}^{k}\tilde{A}|^2-2|\tilde{\nabla}^{k+1}\tilde{A}|^2+C''(1+ |\tilde{\nabla}^{k}\tilde{A}|)|\tilde{\nabla}^{k}\tilde{A}|\\
&+C'C'''(1+ |\tilde{\nabla}^{k}\tilde{A}|+|\tilde{\nabla}^{k+1}\tilde{A}|). 
\end{align*}
Since $-|\tilde{\nabla}^{k+1}\tilde{A}|^2+C'C'''|\tilde{\nabla}^{k+1}\tilde{A}|\leq \frac{(C'C''')^2}{4}$, 
we have 
\begin{align*}
\frac{\partial}{\partial s}|\tilde{\nabla}^{k}\tilde{A}|^{2}\leq &\tilde{\Delta} |\tilde{\nabla}^{k}\tilde{A}|^2 -|\tilde{\nabla}^{k+1}\tilde{A}|^2\\
&+C''|\tilde{\nabla}^{k}\tilde{A}|^2+(C''+C'C''')|\tilde{\nabla}^{k}\tilde{A}|+C'C'''+\frac{(C'C''')^2}{4}. 
\end{align*}
By putting $\bar{C}_{k}:=C''+(C''+C'C''')+C'C'''+\frac{(C'C''')^2}{4}, $
we have 
\begin{align}\label{indevofA}
\frac{\partial}{\partial s}|\tilde{\nabla}^{k}\tilde{A}|^{2}\leq \tilde{\Delta} |\tilde{\nabla}^{k}\tilde{A}|^2 -|\tilde{\nabla}^{k+1}\tilde{A}|^2+\bar{C}_{k}(1+|\tilde{\nabla}^{k}\tilde{A}|^2). 
\end{align}
Hence immediately we have 
\begin{align}\label{indevofA2}
\frac{\partial}{\partial s}|\tilde{\nabla}^{k}\tilde{A}|^{2}\leq \tilde{\Delta} |\tilde{\nabla}^{k}\tilde{A}|^2 +\bar{C}_{k}(1+|\tilde{\nabla}^{k}\tilde{A}|^2). 
\end{align}
Note that the inequality (\ref{indevofA}) also holds for $k-1$, that is, we have 
\begin{align}\label{indevofA3}
\frac{\partial}{\partial s}|\tilde{\nabla}^{k-1}\tilde{A}|^{2}\leq \tilde{\Delta} |\tilde{\nabla}^{k-1}\tilde{A}|^2 -|\tilde{\nabla}^{k}\tilde{A}|^2+\bar{C}_{k-1}(1+|\tilde{\nabla}^{k-1}\tilde{A}|^2), 
\end{align}
for some constant $\bar{C}_{k-1}>0$. 
Hence by combining the inequality (\ref{indevofA2}) and (\ref{indevofA3}), we have
\begin{align}\label{indevofA4}
\begin{aligned}
\frac{\partial}{\partial s}(|\tilde{\nabla}^{k}\tilde{A}|^{2}+2\bar{C}_{k}|\tilde{\nabla}^{k-1}\tilde{A}|^{2})\leq& \tilde{\Delta} (|\tilde{\nabla}^{k}\tilde{A}|^{2}+2\bar{C}_{k}|\tilde{\nabla}^{k-1}\tilde{A}|^{2}) \\
&+\bar{C}_{k}-\bar{C}_{k}|\tilde{\nabla}^{k}\tilde{A}|^2\\
&+2\bar{C}_{k}\bar{C}_{k-1}(1+|\tilde{\nabla}^{k-1}\tilde{A}|^2). 
\end{aligned}
\end{align}
Since we have 
\begin{align*}
&\bar{C}_{k}-\bar{C}_{k}|\tilde{\nabla}^{k}\tilde{A}|^2+2\bar{C}_{k}\bar{C}_{k-1}(1+|\tilde{\nabla}^{k-1}\tilde{A}|^2)\\
=&-\bar{C}_{k}(|\tilde{\nabla}^{k}\tilde{A}|^2+2\bar{C}_{k}|\tilde{\nabla}^{k-1}\tilde{A}|^{2})\\
&+\bar{C}_{k}(1+2\bar{C}_{k-1}+2(\bar{C}_{k}+\bar{C}_{k-1})|\tilde{\nabla}^{k-1}\tilde{A}|^2 )
\end{align*}
and $|\tilde{\nabla}^{k-1}\tilde{A}|^2$ is bounded by the assumption of induction, 
one can easily see that there exists a constant $\bar{\bar{C}}_{k}>0$ such that 
\begin{align*}
\frac{\partial}{\partial s}(|\tilde{\nabla}^{k}\tilde{A}|^{2}+2\bar{C}_{k}|\tilde{\nabla}^{k-1}\tilde{A}|^{2}-\bar{\bar{C}}_{k})
&\leq \tilde{\Delta} (|\tilde{\nabla}^{k}\tilde{A}|^{2}+2\bar{C}_{k}|\tilde{\nabla}^{k-1}\tilde{A}|^{2}-\bar{\bar{C}}_{k})\\
&-\bar{C}_{k}(|\tilde{\nabla}^{k}\tilde{A}|^2+2\bar{C}_{k}|\tilde{\nabla}^{k-1}\tilde{A}|^{2}-\bar{\bar{C}}_{k}). 
\end{align*}
Thus, by putting $\mu:=e^{\bar{C}_{k}s}(|\tilde{\nabla}^{k}\tilde{A}|^{2}+2\bar{C}_{k}|\tilde{\nabla}^{k-1}\tilde{A}|^{2}-\bar{\bar{C}}_{k})$, we have 
\[\frac{\partial}{\partial s}\mu\leq\tilde{\Delta}\mu. \]
Since $M$ is compact, $\mu$ is bounded at initial time $s=-\log T$. 
Then, by the parabolic maximum principle, it follows that $\mu$ is also bounded on $M\times[-\log T,\infty)$, that is, 
there exists a constant $\tilde{C}_{k}>0$ such that $\mu\leq \tilde{C}_{k}$ on $M\times[-\log T,\infty)$. 
Hence we have 
\[|\tilde{\nabla}^{k}\tilde{A}|^{2}\leq e^{-\bar{C}_{k}s}\tilde{C}_{k}-2\bar{C}_{k}|\tilde{\nabla}^{k-1}\tilde{A}|^{2}+\bar{\bar{C}}_{k}\leq T^{\bar{C}_{k}}\tilde{C}_{k}+\bar{\bar{C}}_{k}. \]
Thus, by putting $C_{k}:=T^{\bar{C}_{k}}\tilde{C}_{k}+\bar{\bar{C}}_{k}$, we have 
\[|\tilde{\nabla}^{k}\tilde{A}|\leq C_{k}. \]
Hence the induction argument can be proceeded, and we completed the proof. 
\end{proof}

Combining Lemma~\ref{stonelem} and Proposition~\ref{bdofallderiA}, 
we can deduce the following uniform bound of the second derivative of the weighted volume. 
\begin{lemma}\label{2ndderiofmono}
Let $(N,\tilde{g},\tilde{f})$ be a gradient shrinking Ricci soliton with bounded geometry. 
For a fixed time $T>0$, let $\Phi_{t}$ and $g_{t}$ be defined as above, 
and let $F:M\times[0,T)\to N$ be a Ricci-mean curvature flow along the Ricci flow $(N,g_{t})$. 
Assume that $M$ is compact and $F$ satisfies the condition (A2). 
Let $\tilde{F}$ be the normalized mean curvature flow defined by (\ref{scaledmcf2}). 
Then there exists a constant $C'>0$ such that 
\[\Biggl| \frac{d^2}{ds^2} \int_{M}e^{-\tilde{f}\circ\tilde{F}_{s}} \mathop{d\mu(\tilde{F}_{s}^{*}\tilde{g})}\Biggr|
=\Biggl| \frac{d}{ds} \int_{M} \Bigl|H(\tilde{F}_{s})+\nabla \tilde{f}^{\bot_{\tilde{F}_{s}}}\Bigr|_{\tilde{g}}^2 e^{-\tilde{f}\circ\tilde{F}_{s}} d\mu(\tilde{F}_{s}^{*}\tilde{g}) \Biggr| \leq C'\]
uniformly on $[-\log T,\infty)$. 
\end{lemma}
\begin{proof}
As the proof of Proposition~\ref{monoformingra2}, we have 
\begin{align*}
&\int_{M} \Bigl|H(\tilde{F}_{s})+\nabla \tilde{f}^{\bot_{\tilde{F}_{s}}}\Bigr|_{\tilde{g}}^2 e^{-\tilde{f}\circ\tilde{F}_{s}} d\mu(\tilde{F}_{s}^{*}\tilde{g})\\
=&(4\pi)^{\frac{m}{2}}(T-t)\int_{M} u_{t}\Bigl|H(F_{t})+\nabla f_{t}^{\bot_{F_{t}}}\Bigr|_{g_{t}}^2\mathop{F_{t}^{*}\hspace{-1mm}\rho_{t}} d\mu(F_{t}^{*}g_{t}), 
\end{align*}
where $u:=(4\pi(T-t))^{\frac{n-m}{2}}$. 
Since $\frac{d}{ds}=(T-t)\frac{d}{dt}$, we have 
\begin{align}\label{derideri1}
\begin{aligned}
&\frac{d}{ds}\int_{M} \Bigl|H(\tilde{F}_{s})+\nabla \tilde{f}^{\bot_{\tilde{F}_{s}}}\Bigr|_{\tilde{g}}^2 e^{-\tilde{f}\circ\tilde{F}_{s}} d\mu(\tilde{F}_{s}^{*}\tilde{g})\\
=&-(4\pi)^{\frac{m}{2}}(T-t)\int_{M} u_{t}\Bigl|H(F_{t})+\nabla f_{t}^{\bot_{F_{t}}}\Bigr|_{g_{t}}^2\mathop{F_{t}^{*}\hspace{-1mm}\rho_{t}} d\mu(F_{t}^{*}g_{t})\\
&+(4\pi)^{\frac{m}{2}}(T-t)^2\frac{d}{dt}\int_{M} u_{t}\Bigl|H(F_{t})+\nabla f_{t}^{\bot_{F_{t}}}\Bigr|_{g_{t}}^2\mathop{F_{t}^{*}\hspace{-1mm}\rho_{t}} d\mu(F_{t}^{*}g_{t})\\
=&-\int_{M} \Bigl|H(\tilde{F}_{s})+\nabla \tilde{f}^{\bot_{\tilde{F}_{s}}}\Bigr|_{\tilde{g}}^2 e^{-\tilde{f}\circ\tilde{F}_{s}} d\mu(\tilde{F}_{s}^{*}\tilde{g})\\
&+(4\pi)^{\frac{m}{2}}(T-t)^2\frac{d}{dt}\int_{M} u_{t}\Bigl|H(F_{t})+\nabla f_{t}^{\bot_{F_{t}}}\Bigr|_{g_{t}}^2\mathop{F_{t}^{*}\hspace{-1mm}\rho_{t}} d\mu(F_{t}^{*}g_{t}). 
\end{aligned}
\end{align}
First, we consider the term 
\[-\int_{M} \Bigl|H(\tilde{F}_{s})+\nabla \tilde{f}^{\bot_{\tilde{F}_{s}}}\Bigr|_{\tilde{g}}^2 e^{-\tilde{f}\circ\tilde{F}_{s}} d\mu(\tilde{F}_{s}^{*}\tilde{g}). \]
Since $|H(\tilde{F}_{s})|\leq \sqrt{m}|A(\tilde{F}_{s})|$ and we know that $|A(\tilde{F}_{s})|\leq C_{0}$ by Proposition~\ref{bdofallderiA}, 
we can see that 
\begin{align*}
 \Bigl|H(\tilde{F}_{s})+\nabla \tilde{f}^{\bot_{\tilde{F}_{s}}}\Bigr|^2 \leq &|H(\tilde{F}_{s})|^2+2 |H(\tilde{F}_{s})| |\nabla \tilde{f}|+ |\nabla \tilde{f}|^2\\
 \leq & C''(1+|\nabla \tilde{f}|^2)\\
 \leq & C''(1+ \tilde{f}\circ \tilde{F}_{s})
 \end{align*}
 for some constant $C''>0$, where we used $0\leq|\nabla \tilde{f}|^2\leq \tilde{f}$. 
 Hence we have 
 \begin{align}\label{derideri2}
 \begin{aligned}
 &\Biggl| -\int_{M} \Bigl|H(\tilde{F}_{s})+\nabla \tilde{f}^{\bot_{\tilde{F}_{s}}}\Bigr|_{\tilde{g}}^2 e^{-\tilde{f}\circ\tilde{F}_{s}} d\mu(\tilde{F}_{s}^{*}\tilde{g}) \Biggr| \\
 \leq & C'' \int_{M}  (1+ \tilde{f}\circ \tilde{F}_{s}) e^{-\tilde{f}\circ\tilde{F}_{s}} d\mu(\tilde{F}_{s}^{*}\tilde{g}). 
 \end{aligned}
 \end{align}
 Next we consider the term 
 \[(4\pi)^{\frac{m}{2}}(T-t)^2\frac{d}{dt}\int_{M} u_{t}\Bigl|H(F_{t})+\nabla f_{t}^{\bot_{F_{t}}}\Bigr|_{g_{t}}^2\mathop{F_{t}^{*}\hspace{-1mm}\rho_{t}} d\mu(F_{t}^{*}g_{t}). \]
 By Proposition~\ref{vari}, we have 
 \begin{align}\label{derideri3}
 \begin{aligned}
& \frac{d}{dt}\int_{M} u_{t}\Bigl|H(F_{t})+\nabla f_{t}^{\bot_{F_{t}}}\Bigr|_{g_{t}}^2\mathop{F_{t}^{*}\hspace{-1mm}\rho_{t}} d\mu(F_{t}^{*}g_{t})\\
 =&-\int_{M} u_{t}\Bigl|H(F_{t})+\nabla f_{t}^{\bot_{F_{t}}}\Bigr|_{g_{t}}^4\mathop{F_{t}^{*}\hspace{-1mm}\rho_{t}} d\mu(F_{t}^{*}g)+\int_{M} L\bar{u}_{t} \mathop{F_{t}^{*}\hspace{-1mm}\rho_{t}} d\mu(F_{t}^{*}g_{t}), 
 \end{aligned}
 \end{align}
 where we put 
 \begin{align*}
 \begin{aligned}
 \bar{u}_{t}:=&u_{t}\Bigl|H(F_{t})+\nabla f_{t}^{\bot_{F_{t}}}\Bigr|_{g_{t}}^2, \\
 L\bar{u}_{t}:=&\frac{\partial}{\partial t}\bar{u}_{t} - \Delta_{F_{t}^{*}g_{t}}\bar{u}_{t}+\bar{u}_{t}{\mathrm{tr}}^{\bot}(\mathrm{Ric}(g_{t})+\mathop{\mathrm{Hess}}f_{t})\\
 =&u_{t}\biggl(\frac{\partial}{\partial t}-\Delta_{F_{t}^{*}g_{t}} \biggr)\Bigl|H(F_{t})+\nabla f_{t}^{\bot_{F_{t}}}\Bigr|_{g_{t}}^2. 
 \end{aligned}
\end{align*}
First, as the above argument, we can see that 
 \begin{align}\label{derideri4}
 \begin{aligned}
 &\Biggl| -(4\pi)^{\frac{m}{2}}(T-t)^2 \int_{M} u_{t}\Bigl|H(F_{t})+\nabla f_{t}^{\bot_{F_{t}}}\Bigr|_{g_{t}}^4\mathop{F_{t}^{*}\hspace{-1mm}\rho_{t}} d\mu(F_{t}^{*}g_{t}) \Biggr| \\
 \leq & C''' \int_{M}  (1+ \tilde{f}^2\circ \tilde{F}_{s}) e^{-\tilde{f}\circ\tilde{F}_{s}} d\mu(\tilde{F}_{s}^{*}\tilde{g}) 
 \end{aligned}
 \end{align}
 for some constant $C'''>0$. 
Next we consider 
\[\biggl(\frac{\partial}{\partial t}-\Delta_{F_{t}^{*}g_{t}} \biggr)\Bigl|H(F_{t})+\nabla f_{t}^{\bot_{F_{t}}}\Bigr|_{g_{t}}^2. \]
In fact, by the long computation (cf. Lemma~\ref{longcompu}), it follows that there exists a constant $C''''>0$ such that 
\begin{align}\label{derideri5}
(T-t)^2 \Biggl| \biggl(\frac{\partial}{\partial t}-\Delta_{F_{t}^{*}g_{t}} \biggr)\Bigl|H(F_{t})+{\nabla f_{t}}^{\bot_{F_{t}}}\Bigr|_{g_{t}}^2 \Biggr|\leq C''''(1+\tilde{f}\circ\tilde{F}_{s}). 
\end{align}
By combining (\ref{derideri1})-(\ref{derideri5}), it follows that there exists a constant $\bar{C}>0$ such that 
\begin{align*}
&\Biggl| \frac{d}{ds} \int_{M} \Bigl|H(\tilde{F}_{s})+\nabla \tilde{f}^{\bot_{\tilde{F}_{s}}}\Bigr|_{\tilde{g}}^2 e^{-\tilde{f}\circ\tilde{F}_{s}} d\mu(\tilde{F}_{s}^{*}\tilde{g}) \Biggr| \\
\leq & \bar{C} \int_{M}  (1+ \tilde{f}^2\circ \tilde{F}_{s}) e^{-\tilde{f}\circ\tilde{F}_{s}} d\mu(\tilde{F}_{s}^{*}\tilde{g}) . 
\end{align*}
Note that $(1+ \tilde{f}^2)e^{-\tilde{f}}=(1+ \tilde{f}^2)e^{-\frac{\tilde{f}}{2}}e^{-\frac{\tilde{f}}{2}}$ and $(1+ \tilde{f}^2)e^{-\frac{\tilde{f}}{2}}$ is a bounded function on $N$, that is, 
$(1+ \tilde{f}^2)e^{-\frac{\tilde{f}}{2}}\leq \bar{C}'$ for some constant $\bar{C}'$. 
Thus we have 
\begin{align*}
&\Biggl| \frac{d}{ds} \int_{M} \Bigl|H(\tilde{F}_{s})+\nabla \tilde{f}^{\bot_{\tilde{F}_{s}}}\Bigr|_{\tilde{g}}^2 e^{-\tilde{f}\circ\tilde{F}_{s}} d\mu(\tilde{F}_{s}^{*}\tilde{g}) \Biggr| \\
\leq & \bar{C} \int_{M}  (1+ \tilde{f}^2\circ \tilde{F}_{s}) e^{-\tilde{f}\circ\tilde{F}_{s}} d\mu(\tilde{F}_{s}^{*}\tilde{g})\\
\leq & \bar{C} \bar{C}'\int_{M} e^{-\frac{\tilde{f}}{2}\circ\tilde{F}_{s}} d\mu(\tilde{F}_{s}^{*}\tilde{g})\leq \bar{C} \bar{C}'C=:C',  
\end{align*}
where $C$ is the constant appeared in (\ref{stonebound}) of Lemma~\ref{stonelem}. 
\end{proof}

Finally, here we give the proof of Theorem \ref{minLagthm}. 
\begin{proof}[Proof of Theorem \ref{minLagthm}]
We denote the K\"ahler form and the complex structure on $(N,g,f)$ by $\omega$ and $J$ respectively. 
Since $F:L\to N$ is a self-similar solution, $F$ satisfies 
\[H(F)=\lambda {\nabla f}^{\bot}\]
for some constant $\lambda\in\mathbb{R}$. 
Then, by the definition of the mean curvature form $\omega_{H}$, for a tangent vector $X$ on $L$, we have 
\[\omega_{H}(X)=\omega(H(F),F_{*}X)=\lambda\omega({\nabla f}^{\bot},F_{*}X)=\lambda\omega(\nabla f,F_{*}X), \]
where we used the Lagrangian condition in the last equality. 
Since the mean curvature form is exact, there exists a smooth function $\theta$ on $L$ such that $\omega_{H}=d\theta$. 
Let $\{e_{i}\}_{i=1}^{n}$ be an orthonormal local frame on $L$ with respect to the metric $F^{*}g$. 
Since $\omega$ and $J$ are parallel, we have 
\begin{align*}
\Delta\theta=&\nabla_{e_{i}}\omega_{H}(e_{i})-\omega_{H}(\nabla_{e_{i}}e_{i})\\
=&\lambda\nabla_{e_{i}}\omega(\nabla f,F_{*}e_{i})-\omega_{H}(\nabla_{e_{i}}e_{i})\\
=&-\lambda\mathop{\mathrm{Hess}}f(F_{*}e_{i},JF_{*}e_{i})+\lambda\omega(\nabla f,\nabla_{F_{*}e_{i}}F_{*}e_{i})-\lambda\omega(\nabla f,F_{*}(\nabla_{e_{i}}e_{i}))\\
=&-\lambda\mathop{\mathrm{Hess}}f(F_{*}e_{i},JF_{*}e_{i})+\lambda\omega(\nabla f, H(F)). 
\end{align*}
Since the ambient is a gradient shrinking K\"ahler Ricci soliton, we have 
\[\mathop{\mathrm{Hess}}f(F_{*}e_{i},JF_{*}e_{i})=-\mathrm{Ric}(F_{*}e_{i},JF_{*}e_{i})+\frac{1}{2}g(F_{*}e_{i},JF_{*}e_{i})=0. \]
Furthermore, we have 
\[\omega(\nabla f, H(F))=\omega({\nabla f}^{\top}, H(F))=\omega(F_{*}\nabla(F^{*}f),H(F))=-(F^{*}g)(\nabla(F^{*}f),\nabla\theta). \]
Hence $\theta$ satisfies the following linear elliptic equation: 
\[\Delta\theta+\lambda(F^{*}g)(\nabla(F^{*}f),\nabla\theta)=0. \]
Since $L$ is compact, by the maximum principle, we obtain that $\theta$ is a constant, and this implies that $H(F)=0$. 
\end{proof}

\appendix

\section{Evolution equations}\label{EvoEq}
In this appendix, we give a general treatment of evolution equations for tensors with Ricci-mean curvature flows along Ricci flows. 
Note that, in this appendix, we do not assume that $g_{t}$ is the Ricci flow constructed by a gradient shrinking Ricci soliton. 

Let $M$ and $N$ be manifolds with dimension $m$ and $n$ respectively, and assume that $m\leq n$. 
Let $g=(\, g_{t}\,;\, t\in[0,T_{1})\,)$ be a solution of Ricci flow (\ref{evoeq21}) and 
$F:M\times [0,T_{2}) \to N$ be a solution of Ricci-mean curvature flow (\ref{evoeq22}) with $T_{2}\leq T_{1}$. 
Here we introduce the notion of the covariant time derivative $\nabla_{t}$ as in \cite{Smoczyk}. 
Assume that, for each $t\in[0,T_{2})$, $T(t)$ is a smooth section of 
\[E_{t}:=(\mathop{\otimes}^{A}F_{t}^{*}(TN))\otimes(\mathop{\otimes}^{B}F_{t}^{*}(T^{*}N))\otimes(\mathop{\otimes}^{C}TM)\otimes(\mathop{\otimes}^{D}T^{*}M)\]
over $M$, and its correspondence $t\mapsto T(t)$ is smooth. 
Then for each $t\in[0,T_{2})$ we define $(\nabla_{t}T)(t)$ as follows, and it is also a smooth section of $E_{t}$. 
Denote $T$ by local coordinates $(y^{\alpha})_{\alpha=1}^{n}$ on $N$ and $(x^{i})_{i=1}^{m}$ on $M$ as
\[T^{\alpha_{1}\dots\alpha_{A}\hspace{9.5mm}i_{1}\dots i_{C}}_{\hspace{10.5mm}\beta_{1}\dots\beta_{B}\hspace{8.5mm}j_{1}\dots j_{D}}. \]
This is the coefficient of 
\[\frac{\partial}{\partial y^{\alpha_{1}}}\otimes\dots\otimes\frac{\partial}{\partial y^{\alpha_{A}}}\otimes dy^{\beta_{1}}\otimes\dots\otimes dy^{\beta_{B}}
\otimes \frac{\partial}{\partial x^{i_{1}}}\otimes\dots\otimes\frac{\partial}{\partial x^{i_{C}}}\otimes dx^{j_{1}}\otimes\dots\otimes dx^{j_{D}} \]
of $T$. Then the coefficients of $(\nabla_{t}T)(t)$ is defined by
\begin{align*}
(\nabla_{t}T)^{\alpha_{1}\dots\alpha_{A}\hspace{9.5mm}i_{1}\dots i_{C}}_{\hspace{10.5mm}\beta_{1}\dots\beta_{B}\hspace{8.5mm}j_{1}\dots j_{D}}
:=&\frac{\partial}{\partial t}T^{\alpha_{1}\dots\alpha_{A}\hspace{9.5mm}i_{1}\dots i_{C}}_{\hspace{10.5mm}\beta_{1}\dots\beta_{B}\hspace{8.5mm}j_{1}\dots j_{D}}\\
&+\sum_{p=1}^{A}\Gamma_{\gamma\delta}^{\alpha_{p}}H^{\gamma}
T^{\alpha_{1}\dots \delta \dots\alpha_{A}\hspace{9.5mm}i_{1}\dots i_{C}}_{\hspace{14mm}\beta_{1}\dots\beta_{B}\hspace{8.5mm}j_{1}\dots j_{D}}\\
&-\sum_{p=1}^{B}\Gamma_{\gamma\beta_{p}}^{\delta}H^{\gamma}
T^{\alpha_{1}\dots \alpha_{A}\hspace{13.5mm}i_{1}\dots i_{C}}_{\hspace{10.5mm}\beta_{1}\dots \delta\dots \beta_{B}\hspace{8.5mm}j_{1}\dots j_{D}}, 
\end{align*}
where $\Gamma^{\alpha}_{\beta\gamma}$ is the Christoffel symbol of the Levi--Civita connection of $g_{t}$ on $N$ for each time $t$. 
Then one can easily check that this definition does not depend on the choice of local coordinates and 
defines a global smooth section of $E_{t}$ over $M$. 

\begin{remark}
One can easily check that $\nabla_{t}$ satisfies Leibniz rule for tensor contractions. 
For example, for tensors $S^{\alpha}_{ij}$, $T^{\beta}_{k\ell}$, $U_{\alpha\beta}$, $V^{ik}$, $W^{j\ell}$, we have
\begin{align*}
\nabla_{t}(S^{\alpha}_{ij}T^{\beta}_{k\ell}U_{\alpha\beta}V^{ik}W^{j\ell})&=\nabla_{t}S^{\alpha}_{ij}T^{\beta}_{k\ell}U_{\alpha\beta}V^{ik}W^{j\ell}+S^{\alpha}_{ij}\nabla_{t}T^{\beta}_{k\ell}U_{\alpha\beta}V^{ik}W^{j\ell}\\
&+S^{\alpha}_{ij}T^{\beta}_{k\ell}\nabla_{t}U_{\alpha\beta}V^{ik}W^{j\ell}\\
&+S^{\alpha}_{ij}T^{\beta}_{k\ell}U_{\alpha\beta}\nabla_{t}V^{ik}W^{j\ell}+S^{\alpha}_{ij}T^{\beta}_{k\ell}U_{\alpha\beta}V^{ik}\nabla_{t}W^{j\ell}. 
\end{align*}
\end{remark}

Note that in this paper we define 
\begin{align*}
&\mathrm{Rm}(X,Y)Z:=(\nabla_{X}\nabla_{Y}-\nabla_{Y}\nabla_{X}-\nabla_{[X,Y]})Z\\
&R_{\alpha\beta\gamma\delta}:=g\left(\frac{\partial}{\partial y^{\alpha}},\mathrm{Rm}\left(\frac{\partial}{\partial y^{\gamma}},\frac{\partial}{\partial y^{\delta}}\right)\frac{\partial}{\partial y^{\beta}}\right), \\
&R_{\alpha\gamma}:=\mathrm{Ric}_{\alpha\gamma}:=g^{\beta\delta}R_{\alpha\beta\gamma\delta}, 
\end{align*}
and we define 
\[F^{\alpha}_{i}=F^{\alpha}_{i}(t):=\frac{\partial F_{t}^{\alpha}}{\partial x^{i}}, \]
that is a coefficient of the tensor $DF_{t}(=F_{t*})\in\Gamma(M,F_{t}^{*}(TN)\otimes T^{*}M)$. 
By the straightforward computation with the definition of $\nabla_{t}$, 
we get the following formulas, Lemma \ref{changetj}, \ref{evofgt}, and \ref{evofDF}. 

\begin{lemma}\label{changetj}
We have
\begin{align*}
&\nabla_{t}\nabla_{j}T^{\alpha}_{i_{1}\dots i_{k}}-\nabla_{j}\nabla_{t}T^{\alpha}_{i_{1}\dots i_{k}}\\
=&R^{\alpha}_{\ \gamma\delta\beta}H^{\delta}F^{\beta}_{j}T^{\gamma}_{i_{1}\dots i_{k}}
+\frac{\partial}{\partial t}\Gamma^{\alpha}_{\beta\gamma}F^{\beta}_{j}T^{\gamma}_{i_{1}\dots i_{k}}-\sum_{p=1}^{k}\frac{\partial}{\partial t}\Gamma^{\ell}_{ji_{p}}T^{\alpha}_{i_{1}\dots \ell \dots i_{k}}, 
\end{align*}
where $\Gamma^{i}_{jk}$ is the Christoffel symbol of the Levi--Civita connection of $F_{t}^{*}g_{t}$ on $M$ for each time $t$. 
\end{lemma}

\begin{lemma}\label{evofgt}
By the restriction, we consider $g_{t}$, more precisely $g_{t}\circ F_{t}$, as a section of $F_{t}^{*}(T^{*}N)\otimes F_{t}^{*}(T^{*}N)$ over $M$. 
Then we have
\[\nabla_{t}g_{\alpha\beta}=-2R_{\alpha\beta}. \]
\end{lemma}

\begin{lemma}\label{evofDF}
We have 
\[\nabla_{t}F^{\alpha}_{i}=\nabla_{i}H^{\alpha}. \]
\end{lemma}

Combining above lemmas, we have the following. 
\begin{lemma}\label{evofindgRM}
Put $g_{ij}=(F_{t}^{*}g_{t})_{ij}=g_{\alpha\beta}F^{\alpha}_{i}F^{\beta}_{j}$. 
Then we have
\[\frac{\partial}{\partial t}g_{ij}=\nabla_{t}g_{ij}=-2((F^{*}\mathrm{Ric})_{ij}+g(H,A_{ij})). \]
\end{lemma}
\begin{proof}
By the definition of $\nabla_{t}$, the first equality $\frac{\partial}{\partial t}g_{ij}=\nabla_{t}g_{ij}$ is clear. 
By the remark that $\nabla_{t}$ satisfies Leibniz rule for tensor contractions and by Lemma \ref{evofgt} and \ref{evofDF}, we have 
\begin{align*}
\nabla_{t}g_{ij}&=\nabla_{t}(g_{\alpha\beta}F^{\alpha}_{i}F^{\beta}_{j})\\
&=\nabla_{t}g_{\alpha\beta}F^{\alpha}_{i}F^{\beta}_{j}+g_{\alpha\beta}\nabla_{t}F^{\alpha}_{i}F^{\beta}_{j}+g_{\alpha\beta}F^{\alpha}_{i}\nabla_{t}F^{\beta}_{j}\\
&=-2R_{\alpha\beta}F^{\alpha}_{i}F^{\beta}_{j}+g_{\alpha\beta}\nabla_{i}H^{\alpha}F^{\beta}_{j}+g_{\alpha\beta}F^{\alpha}_{i}\nabla_{j}H^{\beta}. 
\end{align*} 
Since $H$ is a normal vector field, we have $g_{\alpha\beta}H^{\alpha}F^{\beta}_{i}=0$. 
By differentiating both sides by $\nabla_{j}$, we have 
\[0=g_{\alpha\beta}\nabla_{j}H^{\alpha}F^{\beta}_{i}+g_{\alpha\beta}H^{\alpha}A^{\beta}_{ji}. \]
Here we used $A^{\beta}_{ji}=\nabla_{j}F^{\beta}_{i}$. 
Since $A_{ij}$ is symmetric, we have 
\[g_{\alpha\beta}\nabla_{i}H^{\alpha}F^{\beta}_{j}+g_{\alpha\beta}F^{\alpha}_{i}\nabla_{j}H^{\beta}=-2g_{\alpha\beta}H^{\alpha}A^{\beta}_{ij}. \]
Here we completed the proof. 
\end{proof}

By using $\frac{\partial}{\partial t}g_{\alpha\beta}=-2R_{\alpha\beta}$ and the Koszul formula, one can deduce the following formula immediately. 
\begin{lemma}\label{evofGamma}
We have
\[\frac{\partial}{\partial t}\Gamma^{\gamma}_{\alpha\beta}=-g^{\gamma\delta}(\nabla_{\alpha}R_{\delta\beta}+\nabla_{\beta}R_{\alpha\delta}-\nabla_{\delta}R_{\alpha\beta}). \] 
\end{lemma}

As an analog of Lemma~\ref{evofGamma}, we can prove the following. 
\begin{lemma}\label{evofindGamma}
We have
\[\frac{\partial}{\partial t}\Gamma^{k}_{ij}=-g^{k\ell}(\nabla_{i}T_{\ell j}+\nabla_{j}T_{i \ell}-\nabla_{\ell}T_{ij}), \]
where we put $T_{ij}:=(F^{*}\mathrm{Ric})_{ij}+g(H,A_{ij})$. 
\end{lemma}

Here we introduce the notion of $\ast$-product following Hamilton \cite{Hamilton2}. 
\begin{notation}
For tensors $S$ and $T$, we write $S\ast T$ to mean a tensor formed by a sum of terms each one of them obtained by contracting some indices of the pair $S$ and $T$ 
by using $g$ and $F^{*}g$ and these inverse. 
There is a property of $\ast$-product that 
\[|S\ast T|\leq C|S||T|, \]
where $C>0$ is a constant which depends only on the algebraic structure of $S\ast T$. 
\end{notation}

\begin{definition}
For $a,b\in\mathbb{N}$, we define a set  $V_{a,b}$ as the set of all (time-dependent) tensors $T$ on $M$ which can be expressed as 
\[T=(\nabla^{k_{1}}\mathrm{Rm}\ast\dots\ast\nabla^{k_{I}}\mathrm{Rm})\ast(\nabla^{\ell_{1}}A\ast\dots\ast\nabla^{\ell_{J}}A)\ast (\mathop{\ast}^{p}DF)\]
with $I,J,p,k_{1},\dots,k_{I},\ell_{1},\dots,\ell_{J}\in\mathbb{N}$ satisfying 
\begin{align*}
\begin{aligned}
\sum_{i=1}^{I}\biggl(1+\frac{1}{2}k_{i}\biggr)+\sum_{j=1}^{J}\biggl(\frac{1}{2}+\frac{1}{2}\ell_{j}\biggr)=a\quad\mathrm{and}\quad\sum_{j=1}^{J}\ell_{j}\leq b, 
\end{aligned}
\end{align*}
and we define a vector space $\mathcal{V}_{a,b}$ as the set of all tensors $T$ on $M$ which can be expressed as 
\[T=a_{1}T_{1}+\dots+a_{r}T_{r}\]
for some $r\in\mathbb{N}$, $a_{1}\dots a_{r}\in\mathbb{R}$ and $T_{1},\dots,T_{r}\in V_{a,b}$. 
\end{definition}

Since $\nabla DF=A$, the following is clear. 
\begin{proposition}\label{propofast}
Assume that $T_{1}\in \mathcal{V}_{a_{1},b_{1}}$, $T_{2}\in \mathcal{V}_{a_{2},b_{2}}$ and $T_{3}\in \mathcal{V}_{a_{3},b_{3}}$. 
Then we have 
\[T_{1}\ast T_{2}\in \mathcal{V}_{a_{1}+a_{2},b_{1}+b_{2}}\quad\mathrm{and}\quad \nabla T_{3}\in\mathcal{V}_{a_{3}+\frac{1}{2},b_{3}+1},\] 
whenever $T_{1}\ast T_{2}$ makes sense. 
\end{proposition}

Combining Lemma \ref{changetj}, \ref{evofGamma}, \ref{evofindGamma}, and Proposition~\ref{propofast}, the following is clear. 
\begin{lemma}\label{degord1}
For a time dependent tensor $T=(T^{\alpha}_{i_{1}\dots i_{k}})\in \mathcal{V}_{a,b}$, we have
\begin{align*}
\nabla_{t}\nabla_{j}T^{\alpha}_{i_{1}\dots i_{k}}-\nabla_{j}\nabla_{t}T^{\alpha}_{i_{1}\dots i_{k}}\in \mathcal{V}_{a+\frac{3}{2}, b+1}. 
\end{align*}
\end{lemma}

\begin{lemma}\label{degord2}
For a tensor $T=(T^{\alpha}_{i_{1}\dots i_{k}})\in\mathcal{V}_{a,b}$, we have
\begin{align*}
\nabla_{j}\Delta T^{\alpha}_{i_{1}\dots i_{k}}-\Delta \nabla_{j}T^{\alpha}_{i_{1}\dots i_{k}}\in\mathcal{V}_{a+\frac{3}{2},b+1}.
\end{align*}
\end{lemma}
\begin{proof}
First of all, we have
\begin{align*}
\nabla_{j}\Delta T^{\alpha}_{i_{1}\dots i_{k}}=&\nabla_{j}\nabla_{p}\nabla^{p} T^{\alpha}_{i_{1}\dots i_{k}}\\
=&\nabla_{p}\nabla_{j}(\nabla^{p} T^{\alpha}_{i_{1}\dots i_{k}})+R^{\alpha}_{\ \beta\gamma\delta}F^{\gamma}_{j}F^{\delta}_{p} \nabla^{p}T^{\beta}_{i_{1}\dots i_{k}}\\
&+R^{p}_{\ \ell jp} \nabla^{\ell}T^{\alpha}_{i_{1} \dots i_{k}}-\sum_{s=1}^{k}R^{\ell}_{\ i_{s} jp} \nabla^{p}T^{\alpha}_{i_{1}\dots \ell \dots i_{k}}, 
\end{align*}
where $R^{i}_{\, jk\ell}$ is the Riemannian curvature tensor of $F_{t}^{*}g_{t}$ on $M$. 
Then, by the Gauss equation: 
\begin{align}\label{Gaussequation}
R_{ki\ell j}=R_{\epsilon\beta\gamma\delta}F^{\epsilon}_{k}F^{\beta}_{i}F^{\gamma}_{\ell}F^{\delta}_{j}-A^{\beta}_{kj}A_{\beta i\ell}+A^{\beta}_{k\ell}A_{\beta i j}\in\mathcal{V}_{1,0}, 
\end{align}
we can see that 
\begin{align*}
\nabla_{j}\Delta T^{\alpha}_{i_{1}\dots i_{k}}-\nabla_{p}\nabla_{j}\nabla^{p} T^{\alpha}_{i_{1}\dots i_{k}}\in\mathcal{V}_{a+\frac{3}{2},b+1}. 
\end{align*}
As above computations, we have
\begin{align*}
\nabla_{j}\nabla^{p} T^{\alpha}_{i_{1}\dots i_{k}}-\nabla^{p}\nabla_{j} T^{\alpha}_{i_{1}\dots i_{k}}\in\mathcal{V}_{a+1,b}. 
\end{align*}
Hence, by differentiating by $\nabla_{p}$, we have
\begin{align*}
\nabla_{p}\nabla_{j}\nabla^{p} T^{\alpha}_{i_{1}\dots i_{k}}-\nabla_{p}\nabla^{p}\nabla_{j} T^{\alpha}_{i_{1}\dots i_{k}}\in \mathcal{V}_{a+\frac{3}{2},b+1}. 
\end{align*}
Note that $\nabla_{p}\nabla^{p}\nabla_{j} T^{\alpha}_{i_{1}\dots i_{k}}=\Delta \nabla_{j} T^{\alpha}_{i_{1}\dots i_{k}}$. 
Here we completed the proof. 
\end{proof}

\begin{lemma}\label{degord3}
For a time dependent tensor $T=(T^{\alpha}_{i_{1}\dots i_{k}})\in\mathcal{V}_{a,b}$ there exists a tensor $\mathcal{D}=\mathcal{D}(T)\in\mathcal{V}_{a+1,b}$ such that 
\begin{align*}
\nabla_{t}|T|^2=2\langle \nabla_{t}T,T\rangle+\mathcal{D}\ast T. 
\end{align*}
\end{lemma}
\begin{proof}
We have
\begin{align*}
\nabla_{t}|T|^2=&\nabla_{t}(g_{\alpha\beta}g^{i_{1}j_{1}}\dots g^{i_{k}j_{k}}T^{\alpha}_{i_{1}\dots i_{k}}T^{\beta}_{j_{1}\dots j_{k}})\\
=&2\langle \nabla_{t}T,T\rangle+\nabla_{t}g \ast T\ast T+\nabla_{t}((F^{*}g)^{-1})\ast T\ast T\\
=&2\langle \nabla_{t}T,T\rangle+(\nabla_{t}g \ast T+\nabla_{t}((F^{*}g)^{-1})\ast T)\ast T. 
\end{align*}
By Lemma \ref{evofgt} and \ref{evofindgRM}, we have 
\begin{align*}
\nabla_{t}g, \nabla_{t}((F^{*}g)^{-1})\in\mathcal{V}_{1,0}. 
\end{align*}
Thus the statement is clear. 
\end{proof}

\begin{lemma}\label{deglow}
For $k\geq 1$, by definitions, it is clear that 
\[F^{\alpha}_{p}\nabla_{i_{1}}\dots\nabla_{i_{k}}A_{\alpha}\in\mathcal{V}_{\frac{1}{2}+\frac{1}{2}k,k}. \]
Actually, it is true that 
\[F^{\alpha}_{p}\nabla_{i_{1}}\dots\nabla_{i_{k}}A_{\alpha}\in\mathcal{V}_{\frac{1}{2}+\frac{1}{2}k,k-1}. \]
For $k=0$, it is clear that 
\[F^{\alpha}_{p}A_{\alpha}=0\]
since $A$ is a normal bundle valued 2-tensor. 
\end{lemma}
\begin{proof}
By differentiating the equation $F^{\alpha}_{p}A_{\alpha}=0$, we have 
\[F^{\alpha}_{p}\nabla_{i_{1}}A_{\alpha}=-A^{\alpha}_{i_{1}p}A_{\alpha}\in\mathcal{V}_{1,0}. \]
Hence the statement is true for $k=1$. 
Assume that for $k-1$ the statement is true. 
Then, for $k$, the statement is also true since we have 
\[F^{\alpha}_{p}\nabla_{i_{1}}\dots\nabla_{i_{k}}A_{\alpha}=\nabla_{i_{1}}(F^{\alpha}_{p}\nabla_{i_{2}}\dots\nabla_{i_{k}}A_{\alpha})-A^{\alpha}_{i_{1}p}\nabla_{i_{2}}\dots\nabla_{i_{k}}A_{\alpha}. \]
We completed the proof. 
\end{proof}

\begin{lemma}\label{evofAij}
There exist tensors 
$\mathcal{B}=(\mathcal{B}^{\alpha}_{ij})\in\mathcal{V}_{\frac{3}{2},0}$ and $\mathcal{C}=(\mathcal{C}^{p}_{ij})\in\mathcal{V}_{\frac{3}{2},1}$ such that 
\[\nabla_{t}A^{\alpha}_{ij}=\Delta A^{\alpha}_{ij}+\mathcal{B}^{\alpha}_{ij}+\mathcal{C}^{p}_{ij}F^{\alpha}_{p}. \]
\end{lemma}
\begin{proof}
By identies $A^{\alpha}_{ij}=\nabla_{i}F^{\alpha}_{j}$ and $\nabla_{t}F^{\alpha}_{j}=\nabla_{j}H^{\alpha}$ and Lemma~\ref{changetj}, we have
\begin{align*}
\nabla_{t}A^{\alpha}_{ij}=&\nabla_{t}\nabla_{i}F^{\alpha}_{j}\\
=&\nabla_{i}\nabla_{t}F^{\alpha}_{j}+R^{\alpha}_{\ \gamma\delta\beta}H^{\delta}F^{\beta}_{i}F^{\gamma}_{j}
+\frac{\partial}{\partial t}\Gamma^{\alpha}_{\beta\gamma}F^{\beta}_{i}F^{\gamma}_{j}-\frac{\partial}{\partial t}\Gamma^{\ell}_{ij}F^{\alpha}_{\ell}\\
=&\nabla_{i}\nabla_{j}H^{\alpha}+R^{\alpha}_{\ \gamma\delta\beta}H^{\delta}F^{\beta}_{i}F^{\gamma}_{j}
+\frac{\partial}{\partial t}\Gamma^{\alpha}_{\beta\gamma}F^{\beta}_{i}F^{\gamma}_{j}-\frac{\partial}{\partial t}\Gamma^{\ell}_{ij}F^{\alpha}_{\ell}. 
\end{align*}
Furthermore, by using Simons' identity: 
\begin{align*}
\nabla_{k}\nabla_{\ell}H^{\alpha}=&\Delta A^{\alpha}_{k\ell}+(\nabla_{\epsilon}R^{\alpha}_{\ \beta\gamma\delta}+\nabla_{\gamma}R^{\alpha}_{\ \delta\beta\epsilon})
F^{\epsilon}_{i}F^{\beta}_{\ell}F^{\gamma}_{k}F^{\delta i}\\
&+R^{\alpha}_{\ \beta\gamma\delta}(2A^{\beta}_{ik}F^{\gamma}_{\ell}F^{\delta i}+2A^{\beta}_{i\ell}F^{\gamma}_{k}F^{\delta i}+H^{\delta}F^{\beta}_{\ell}F^{\gamma}_{k}+A^{\gamma}_{\ell k}F^{\beta}_{i}F^{\delta i})\\
&-(\nabla_{k}R^{p}_{\ \ell}+\nabla_{\ell}R^{p}_{\ k}-\nabla^{p}R_{k\ell})F^{\alpha}_{p}\\
&+2R^{\,\,\, i\,\,\, j}_{k\,\, \ell}A^{\alpha}_{ij}-R^{p}_{\ k}A^{\alpha}_{p\ell}-R^{p}_{\ \ell}A^{\alpha}_{pk}, 
\end{align*}
we have
\begin{align*}
\nabla_{t}A^{\alpha}_{ij}=&\Delta A^{\alpha}_{ij}\\
&+(\nabla_{\epsilon}R^{\alpha}_{\ \beta\gamma\delta}+\nabla_{\gamma}R^{\alpha}_{\ \delta\beta\epsilon})
F^{\epsilon}_{k}F^{\beta}_{j}F^{\gamma}_{i}F^{\delta k}\\
&+R^{\alpha}_{\ \beta\gamma\delta}(2A^{\beta}_{ki}F^{\gamma}_{j}F^{\delta k}+2A^{\beta}_{kj}F^{\gamma}_{i}F^{\delta k}+H^{\delta}F^{\beta}_{j}F^{\gamma}_{i}+A^{\gamma}_{j i}F^{\beta}_{k}F^{\delta k})\\
&-(\nabla_{i}R^{p}_{\ j}+\nabla_{j}R^{p}_{\ i}-\nabla^{p}R_{ij})F^{\alpha}_{p}\\
&+2R^{\,\,\, k\,\, \ell}_{i\,\,\, j}A^{\alpha}_{k\ell}-R^{p}_{\ i}A^{\alpha}_{pj}-R^{p}_{\ j}A^{\alpha}_{pi}\\
&+R^{\alpha}_{\ \gamma\delta\beta}H^{\delta}F^{\beta}_{i}F^{\gamma}_{j}
+\frac{\partial}{\partial t}\Gamma^{\alpha}_{\beta\gamma}F^{\beta}_{i}F^{\gamma}_{j}-\frac{\partial}{\partial t}\Gamma^{p}_{ij}F^{\alpha}_{p}. 
\end{align*}
By putting 
\begin{align*}
\mathcal{B}^{\alpha}_{ij}:=&(\nabla_{\epsilon}R^{\alpha}_{\ \beta\gamma\delta}+\nabla_{\gamma}R^{\alpha}_{\ \delta\beta\epsilon})
F^{\epsilon}_{k}F^{\beta}_{j}F^{\gamma}_{i}F^{\delta k}\\
&+R^{\alpha}_{\ \beta\gamma\delta}(2A^{\beta}_{ki}F^{\gamma}_{j}F^{\delta k}+2A^{\beta}_{kj}F^{\gamma}_{i}F^{\delta k}+H^{\delta}F^{\beta}_{j}F^{\gamma}_{i}+A^{\gamma}_{j i}F^{\beta}_{k}F^{\delta k})\\
&+2R^{\,\,\, k\,\, \ell}_{i\,\,\, j}A^{\alpha}_{k\ell}-R^{p}_{\ i}A^{\alpha}_{pj}-R^{p}_{\ j}A^{\alpha}_{pi}\\
&+R^{\alpha}_{\ \gamma\delta\beta}H^{\delta}F^{\beta}_{i}F^{\gamma}_{j}
+\frac{\partial}{\partial t}\Gamma^{\alpha}_{\beta\gamma}F^{\beta}_{i}F^{\gamma}_{j} \\
\mathcal{C}^{p}_{ij}:=&-(\nabla_{i}R^{p}_{\ j}+\nabla_{j}R^{p}_{\ i}-\nabla^{p}R_{ij})-\frac{\partial}{\partial t}\Gamma^{p}_{ij}, 
\end{align*}
We have
\[\nabla_{t}A^{\alpha}_{ij}=\Delta A^{\alpha}_{ij}+\mathcal{B}^{\alpha}_{ij}+\mathcal{C}^{p}_{ij}F^{\alpha}_{p}. \]
Furthermore, by using Lemma \ref{evofGamma}, \ref{evofindGamma} and Gauss equation (\ref{Gaussequation}), 
one can easily see that 
\[\mathcal{B}=(\mathcal{B}^{\alpha}_{ij})\in\mathcal{V}_{\frac{3}{2},0}\quad\mathrm{and}\quad \mathcal{C}=(\mathcal{C}^{p}_{ij})\in\mathcal{V}_{\frac{3}{2},1}. \]
Here we completed the proof. 
\end{proof}

\begin{proposition}\label{evofnormsqofA}
There exists a tensor $\mathcal{E}\in\mathcal{V}_{\frac{3}{2},0}$ such that 
\[\frac{\partial}{\partial t}|A|^2=\Delta |A|^2-2|\nabla A|^2+\mathcal{E}\ast A. \]
\end{proposition}
\begin{proof}
By Lemma~\ref{degord3}, there exists a tensor $\mathcal{D}=\mathcal{D}(A)\in\mathcal{V}_{\frac{3}{2},0}$ such that 
\[\frac{\partial}{\partial t}|A|^2=\nabla_{t}|A|^2=2\langle \nabla_{t}A,A\rangle+\mathcal{D}\ast A. \]
By Lemma~\ref{evofAij}, we have 
\begin{align*}
2\langle \nabla_{t}A,A\rangle=&2(\Delta A^{\alpha}_{ij}+\mathcal{B}^{\alpha}_{ij}+\mathcal{C}^{p}_{ij}F^{\alpha}_{p})A_{\alpha}^{ij}\\
=&\Delta |A|^2-2|\nabla A|^2+\mathcal{B}\ast A. 
\end{align*}
Here we used $F^{\alpha}_{p}A_{\alpha}^{ij}=0$. 
Hence, by putting $\mathcal{E}:=\mathcal{D}+\mathcal{B}$, we have
\[\frac{\partial}{\partial t}|A|^2=\Delta |A|^2-2|\nabla A|^2+\mathcal{E}\ast A, \]
and, we have that 
\[\mathcal{E}\in\mathcal{V}_{\frac{3}{2},0}. \]
Here we completed the proof. 
\end{proof}

\begin{proposition}
We have
\[\frac{\partial}{\partial t}|A|^2\leq \Delta |A|^2-2|\nabla A|^2+C_{1}|A|^4+C_{2}|\mathrm{Rm}||A|^2+C_{3}|\nabla \mathrm{Rm}||A|, \]
for positive constants $C_{1}$, $C_{2}$, $C_{3}$ which depend only on the dimension of $M$ and $N$, 
where $\mathrm{Rm}$ is the Riemannian curvature tensor of $(N,g_{t})$. 
\end{proposition}
\begin{proof}
By Proposition~\ref{evofnormsqofA}, we know that 
\[\frac{\partial}{\partial t}|A|^2=\Delta |A|^2-2|\nabla A|^2+\mathcal{E}\ast A. \]
Since $\mathcal{E}\in\mathcal{V}_{\frac{3}{2},0}$, 
the tensor $\mathcal{E}$ can be constructed by 
\[A\ast A\ast A\ast(\ast^{i} DF),\ \mathrm{Rm}\ast A\ast(\ast^{j} DF),\ \nabla \mathrm{Rm}\ast(\ast^{k} DF). \]
Note that we can not decide $i$, $j$, $k$ from the information that $\mathcal{E}\in\mathcal{V}_{\frac{3}{2},0}$. 
However this is not a matter when we consider the norm of tensors since the norm of $DF$ is a constant $\sqrt{m}$. 
Hence we see that there exist positive constants $C_{1}$, $C_{2}$, $C_{3}$ which depend only on the dimensions of $M$ and $N$ such that 
\[|\mathcal{E}\ast A|\leq C_{1}|A|^4+C_{2}|\mathrm{Rm}||A|^2+C_{3}|\nabla \mathrm{Rm}||A|. \]
Here we completed the proof. 
\end{proof}

\begin{proposition}\label{lapandt}
For all $k\geq 0$ there exist tensors 
$\mathcal{B}[k]\in\mathcal{V}_{\frac{3}{2}+\frac{1}{2}k,k}$ and $\mathcal{C}[k]\in\mathcal{V}_{\frac{3}{2}+\frac{1}{2}k,k+1}$ such that 
\[\nabla_{t}\nabla_{\ell_{1}}\dots\nabla_{\ell_{k}}A^{\alpha}_{ij}=\Delta \nabla_{\ell_{1}}\dots\nabla_{\ell_{k}}A^{\alpha}_{ij}+\mathcal{B}[k]+\mathcal{C}[k]^{p}F^{\alpha}_{p}. \]
\end{proposition}
\begin{proof}
We work by induction on $k\in\mathbb{N}$. 
For the case $k=0$, the statement is true by Lemma~\ref{evofAij}. 
Assume that for $k-1$ the statement is true. 
Since 
\[\nabla_{\ell_{2}}\dots\nabla_{\ell_{k}}A^{\alpha}_{ij}\in\mathcal{V}_{\frac{3}{2}+\frac{1}{2}(k-1),k-1}, \]
by defining $\mathcal{D}$ as 
\[\nabla_{t}\nabla_{\ell_{1}}\nabla_{\ell_{2}}\dots\nabla_{\ell_{k}}A^{\alpha}_{ij}=\nabla_{\ell_{1}}\nabla_{t}\nabla_{\ell_{2}}\dots\nabla_{\ell_{k}}A^{\alpha}_{ij}+\mathcal{D}, \]
we have, by Lemma~\ref{degord1}, that 
\begin{align*}
\mathcal{D}\in \mathcal{V}_{\frac{3}{2}+\frac{1}{2}k,k}. 
\end{align*}
By the assumption of the induction for $k-1$, we have
\begin{align*}
&\nabla_{\ell_{1}}\nabla_{t}\nabla_{\ell_{2}}\dots\nabla_{\ell_{k}}A^{\alpha}_{ij}\\
=&\nabla_{\ell_{1}}(\Delta \nabla_{\ell_{2}}\dots\nabla_{\ell_{k}}A^{\alpha}_{ij}+\mathcal{B}[k-1]+\mathcal{C}[k-1]^{p}F^{\alpha}_{p})\\
=&\nabla_{\ell_{1}}\Delta \nabla_{\ell_{2}}\dots\nabla_{\ell_{k}}A^{\alpha}_{ij}+\nabla\mathcal{B}[k-1]+\nabla\mathcal{C}[k-1]^{p}F^{\alpha}_{p}+\mathcal{C}[k-1]\ast A. 
\end{align*}
Furthermore by Lemma~\ref{degord2}, 
by defining $\mathcal{D'}$ as 
\[\nabla_{\ell_{1}}\Delta \nabla_{\ell_{2}}\dots\nabla_{\ell_{k}}A^{\alpha}_{ij}=\Delta \nabla_{\ell_{1}}\nabla_{\ell_{2}}\dots\nabla_{\ell_{k}}A^{\alpha}_{ij}+\mathcal{D'}, \]
we have that 
\begin{align*}
\mathcal{D'}\in\mathcal{V}_{\frac{3}{2}+\frac{1}{2}k,k}. 
\end{align*}
Hence, by putting 
\begin{align*}
\mathcal{B}[k]:=&\nabla\mathcal{B}[k-1]+\mathcal{C}[k-1]\ast A+\mathcal{D}+\mathcal{D'}. \\
\mathcal{C}[k]:=&\nabla\mathcal{C}[k-1], 
\end{align*}
we have
\[\nabla_{t}\nabla_{\ell_{1}}\dots\nabla_{\ell_{k}}A^{\alpha}_{ij}=\Delta \nabla_{\ell_{1}}\dots\nabla_{\ell_{k}}A^{\alpha}_{ij}+\mathcal{B}[k]+\mathcal{C}[k]^{p}F^{\alpha}_{p}. \]
and 
\[\mathcal{B}[k]\in\mathcal{V}_{\frac{3}{2}+\frac{1}{2}k,k}\quad\mathrm{and}\quad\mathcal{C}[k]\in\mathcal{V}_{\frac{3}{2}+\frac{1}{2}k,k+1}. \]
Here we completed the proof. 
\end{proof}

\begin{proposition}\label{evofhiA}
For all $k\geq 0$ 
there exist tensors $\mathcal{E}[k]\in\mathcal{V}_{\frac{3}{2}+\frac{1}{2}k,k}$, $\mathcal{C}[k]\in\mathcal{V}_{\frac{3}{2}+\frac{1}{2}k,k+1}$ and 
$\mathcal{G}[k]\in\mathcal{V}_{\frac{1}{2}+\frac{1}{2}k,k-1}$ such that 
\[\frac{\partial}{\partial t}|\nabla^{k}A|^2=\Delta |\nabla^{k}A|^2-2|\nabla^{k+1}A|^2+\mathcal{E}[k]\ast \nabla^{k}A+\mathcal{C}[k]\ast \mathcal{G}[k]. \]
\end{proposition}
\begin{proof}
Put $T^{\alpha}_{\ell_{1}\dots \ell_{k}ij}:=\nabla_{\ell_{1}}\dots\nabla_{\ell_{k}}A^{\alpha}_{ij}$. 
Since $T\in\mathcal{V}_{\frac{1}{2}+\frac{1}{2}k,k}$, 
by Lemma~\ref{degord3} there exists a tensor $\mathcal{D}[k]=\mathcal{D}(T)\in\mathcal{V}_{\frac{3}{2}+\frac{1}{2}k,k}$ such that 
\[\frac{\partial}{\partial t}|T|^2=\nabla_{t}|T|^2=2\langle \nabla_{t}T,T\rangle+\mathcal{D}[k]\ast T. \]
By Proposition~\ref{lapandt}, 
there exist tensors $\mathcal{B}[k]\in\mathcal{V}_{\frac{3}{2}+\frac{1}{2}k,k}$ and $\mathcal{C}[k]\in\mathcal{V}_{\frac{3}{2}+\frac{1}{2}k,k+1}$ such that 
\[\nabla_{t}T=\Delta T+\mathcal{B}[k]+\mathcal{C}[k]^{p}F^{\alpha}_{p}. \]
Hence  we have 
\begin{align*}
2\langle \nabla_{t}T,T\rangle=&2\langle \Delta T,T\rangle+\mathcal{B}[k]\ast T+\mathcal{C}[k]^{p} F^{\alpha}_{p}T_{\alpha}\\
=&\Delta |T|^2-2|\nabla T|^2+\mathcal{B}[k]\ast T+\mathcal{C}[k]^{p} F^{\alpha}_{p}T_{\alpha}. 
\end{align*}
By Lemma~\ref{deglow}, we have 
\[\mathcal{G}[k]:=F^{\alpha}_{p}T_{\alpha}\in\mathcal{V}_{\frac{1}{2}+\frac{1}{2}k,k-1}. \]
Hence, by putting $\mathcal{E}[k]:=\mathcal{D}[k]+\mathcal{B}[k]\in\mathcal{V}_{\frac{3}{2}+\frac{1}{2}k,k}$, we have
\[\frac{\partial}{\partial t}|T|^2=\Delta |T|^2-2|\nabla T|^2+\mathcal{E}[k]\ast T+\mathcal{C}[k]\ast \mathcal{G}[k]. \]
Here we completed the proof. 
\end{proof}
\section{An estimate in the proof of Lemma~\ref{2ndderiofmono}}\label{EstiforLem}
In this appendix, we give a proof for the following estimate which is used in the proof of Lemma~\ref{2ndderiofmono}. 
It is just a straightforward long computation. 
\begin{lemma}\label{longcompu}
In the situation of Lemma~\ref{2ndderiofmono}, there exists a constant $C''''>0$ such that 
\begin{align}\label{derideri5'}
(T-t)^2 \Biggl| \biggl(\frac{\partial}{\partial t}-\Delta_{F_{t}^{*}g_{t}} \biggr)\Bigl|H(F_{t})+{\nabla f_{t}}^{\bot_{F_{t}}}\Bigr|_{g_{t}}^2 \Biggr|\leq C''''(1+\tilde{f}\circ\tilde{F}_{s}). 
\end{align}
\end{lemma}
\begin{proof}
First of all, we define $W_{c,d}$ and $\mathcal{W}_{c,d}$ as analogs of $V_{a,b}$ and $\mathcal{V}_{a,b}$. 
For $c,d\in\mathbb{N}$, we define a set  $W_{c,d}$ as the set of all (time-dependent) tensors $T$ on $M$ which can be expressed as 
\[T=\frac{1}{(T-t)^{q}}(\mathop{\ast}^{r}\nabla f)\ast(\nabla^{k_{1}}\mathrm{Rm}\ast\dots\ast\nabla^{k_{I}}\mathrm{Rm})\ast(\nabla^{\ell_{1}}A\ast\dots\ast\nabla^{\ell_{J}}A)\ast (\mathop{\ast}^{p}DF)\]
with $q,r,I,J,p,k_{1},\dots,k_{I},\ell_{1},\dots,\ell_{J}\in\mathbb{N}$ satisfying 
\begin{align*}
\begin{aligned}
q+\frac{1}{2}r+\sum_{i=1}^{I}\biggl(1+\frac{1}{2}k_{i}\biggr)+\sum_{j=1}^{J}\biggl(\frac{1}{2}+\frac{1}{2}\ell_{j}\biggr)=c\quad\mathrm{and}\quad r\leq d, 
\end{aligned}
\end{align*}
and we define a vector space $\mathcal{W}_{c,d}$ as the set of all tensors $T$ on $M$ which can be expressed as 
\[T=a_{1}T_{1}+\dots+a_{v}T_{v}\]
for some $v\in\mathbb{N}$, $a_{1}\dots a_{v}\in\mathbb{R}$ and $T_{1},\dots,T_{v}\in W_{c,d}$. 
By the definition, it is clear that $\mathcal{V}_{a,b}\subset \mathcal{W}_{a,0}$, and if $T_{1}\in \mathcal{W}_{c_{1},d_{1}}$ and $T_{2}\in \mathcal{W}_{c_{2},d_{2}}$ 
then $T_{1}\ast T_{2}\in \mathcal{W}_{c_{1}+c_{2},d_{1}+d_{2}}$. 

Note that we consider $\nabla^{\alpha} f$ as a tensor field over $M$ by pulling it back by $F_{t}$. However we sometimes omit the symbol $\circ F_{t}$. 
Then we have 
\begin{align}\label{nabnabf}
\nabla_{i}\nabla^{\alpha}f=F^{\beta}_{i}\nabla_{\beta}\nabla^{\alpha}f=-\mathrm{Ric}^{\ \alpha}_{\beta}F^{\beta}_{i}+\frac{1}{2(T-t)}F^{\alpha}_{i}\in\mathcal{W}_{1,0}, 
\end{align}
where we used $\mathrm{Ric}_{\alpha \beta}+\nabla_{\alpha}\nabla_{\beta}f=\frac{1}{2(T-t)}g_{\alpha\beta}$. 
Hence we can see that if $T\in \mathcal{W}_{c,d}$ then $\nabla T\in \mathcal{W}_{c+\frac{1}{2},d}$. 

To prove this lemma, we use the identity 
\[\Bigl|H(F_{t})+\nabla f_{t}^{\bot_{F_{t}}}\Bigr|_{g_{t}}^2=|H(F_{t})|_{g_{t}}^2+2g_{t}(H(F_{t}),\nabla f_{t})+|\nabla f_{t}|_{g_{t}}^2-|{\nabla f_{t}}^{\top_{F_{t}}}|_{g_{t}}^2. \]

By Lemma \ref{evofindgRM} and \ref{evofAij}, we have 
\begin{align*}
\nabla_{t}H^{\alpha}=&\nabla_{t}(g^{ij}A^{\alpha}_{ij})\\
=&2(\mathrm{Ric}_{\beta\gamma}F^{\beta i}F^{\gamma j}+H^{\beta}A_{\beta}^{ij})A^{\alpha}_{ij}+g^{ij}(\Delta A^{\alpha}_{ij}+\mathcal{B}^{\alpha}_{ij}+\mathcal{C}^{p}_{ij}F^{\alpha}_{p})\\
=&\Delta H^{\alpha}+\bar{\mathcal{B}}^{\alpha}+g^{ij}\mathcal{C}^{p}_{ij}F^{\alpha}_{p}, 
\end{align*}
where we put $\bar{\mathcal{B}}^{\alpha}:=2(\mathrm{Ric}_{\beta\gamma}F^{\beta i}F^{\gamma j}+H^{\beta}A_{\beta}^{ij})A^{\alpha}_{ij}+g^{ij}\mathcal{B}^{\alpha}_{ij}\in \mathcal{V}_{\frac{3}{2},0}$. 
Since $H\in \mathcal{V}_{\frac{1}{2},0}$, by Lemma~\ref{degord3} there exists a tensor $\mathcal{D}=\mathcal{D}(H)\in\mathcal{V}_{\frac{3}{2},0}$ such that 
\begin{align*}
\frac{\partial}{\partial t}|H|^2=&2\langle \nabla_{t}H,H\rangle +\mathcal{D}^{\alpha}H_{\alpha}\\
=&2\langle \Delta H, H \rangle + 2\bar{\mathcal{B}}^{\alpha}H_{\alpha}+\mathcal{D}^{\alpha}H_{\alpha}\\
=&\Delta |H|^2 -2 |\nabla H|^2 + 2\bar{\mathcal{B}}^{\alpha}H_{\alpha}+\mathcal{D}^{\alpha}H_{\alpha}, 
\end{align*}
where we used $F^{\alpha}_{p}H_{\alpha}=0$. 
Thus we have 
\[\biggl(\frac{\partial}{\partial t}-\Delta_{F^{*}g} \biggr)|H|^{2}=-2 |\nabla H|^2 +  2\bar{\mathcal{B}}^{\alpha}H_{\alpha}+\mathcal{D}^{\alpha}H_{\alpha}, \]
and it is clear that 
\begin{align}\label{AW20}
\biggl(\frac{\partial}{\partial t}-\Delta_{F^{*}g} \biggr)|H|^{2}\in \mathcal{W}_{2,0}. 
\end{align}

Next, we consider $\nabla_{t}\nabla^{\alpha}f$. 
Then, by the definition of $\nabla_{t}$, we have
\begin{align*}
\nabla_{t}\nabla^{\alpha}f
=&\frac{\partial}{\partial t}\nabla^{\alpha}f+H^{\beta}\nabla_{\beta}\nabla^{\alpha}f\\
=&\frac{\partial}{\partial t}\nabla^{\alpha}f-\mathrm{Ric}_{\ \beta}^{\alpha}H^{\beta}+\frac{1}{2(T-t)}H^{\alpha}, \\
\end{align*}
where we used $\mathrm{Ric}_{\alpha \beta}+\nabla_{\alpha}\nabla_{\beta}f=\frac{1}{2(T-t)}g_{\alpha\beta}$. 
Furthermore, by using $\frac{\partial}{\partial t}g^{\alpha\beta}=2\mathrm{Ric}^{\alpha\beta}$ and $\frac{\partial}{\partial t}f=|\nabla f|^2$, 
one can easily see that 
\[\frac{\partial}{\partial t}\nabla^{\alpha}f=\frac{1}{T-t}\nabla^{\alpha}f. \]
Thus  we have 
\[\nabla_{t}\nabla^{\alpha}f=\frac{1}{T-t}\nabla^{\alpha}f-\mathrm{Ric}_{\ \beta}^{\alpha}H^{\beta}+\frac{1}{2(T-t)}H^{\alpha} \in \mathcal{W}_{\frac{3}{2},1}. \]
Hence we can see that 
\[\nabla_{t}(g_{\alpha\beta}H^{\alpha}\nabla^{\beta}f)
=\nabla_{t}g_{\alpha\beta}H^{\alpha}\nabla^{\beta}f+g_{\alpha\beta}\nabla_{t}H^{\alpha}\nabla^{\beta}f+g_{\alpha\beta}H^{\alpha}\nabla_{t}\nabla^{\beta}f\in\mathcal{W}_{2,1}\]
and 
\[\Delta(g_{\alpha\beta}H^{\alpha}\nabla^{\beta}f) \in \mathcal{W}_{2,1}\]
Thus we have 
\begin{align}\label{AfW21}
\biggl(\frac{\partial}{\partial t}-\Delta_{F^{*}g} \biggr)g(H(F),\nabla f)\in\mathcal{W}_{2,1}. 
\end{align}

Next, one can easily see that 
\[\nabla_{t}|\nabla f|^2=\nabla_{t}g_{\alpha\beta}\nabla^{\alpha}f\nabla^{\beta}f+2\nabla_{t}\nabla^{\alpha}f\nabla_{\alpha}f\in\mathcal{W}_{2,2}\]
and 
\[\Delta|\nabla f|^2\in\mathcal{W}_{2,2}. \]
Hence we have 
\begin{align}\label{ffW22}
\biggl(\frac{\partial}{\partial t}-\Delta_{F^{*}g} \biggr)|\nabla f|^2 \in\mathcal{W}_{2,2}.
\end{align}

Finally, one can easily see that 
\[\nabla_{t}|{\nabla f_{t}}^{\top_{F_{t}}}|_{g_{t}}^2=\nabla_{t}((F^{*}g)^{k\ell}g_{\alpha\beta}g_{\gamma\delta}F^{\beta}_{k}F^{\delta}_{\ell}\nabla^{\alpha}f\nabla^{\gamma}f)\in\mathcal{W}_{2,2}\]
and 
\[\Delta |{\nabla f_{t}}^{\top_{F_{t}}}|_{g_{t}}^2 \in \mathcal{W}_{2,2}. \]
Hence we have
\begin{align}\label{ftftW22}
\biggl(\frac{\partial}{\partial t}-\Delta_{F^{*}g} \biggr) |{\nabla f_{t}}^{\top_{F_{t}}}|_{g_{t}}^2\in \mathcal{W}_{2,2}. 
\end{align}

Hence, by (\ref{AW20})--(\ref{ftftW22}),  we have 
\[\biggl(\frac{\partial}{\partial t}-\Delta_{F^{*}g} \biggr)\Bigl|H(F)+\nabla f^{\bot_{F}}\Bigr|_{g}^2\in \mathcal{W}_{2,2}. \]
By the definition of $\mathcal{W}_{2,2}$, 
there exist $v\in\mathbb{N}$, $a_{1}\dots a_{v}\in\mathbb{R}$ and $T_{1},\dots,T_{v}\in W_{2,2}$ such that 
\[\biggl(\frac{\partial}{\partial t}-\Delta_{F^{*}g} \biggr)\Bigl|H(F)+\nabla f^{\bot_{F}}\Bigr|_{g}^2=a_{1}T_{1}+\dots+a_{v}T_{v}. \]
Hence we have 
\[\left|\left(\frac{\partial}{\partial t}-\Delta_{F^{*}g} \right)\left|H(F)+\nabla f^{\bot_{F}}\right|_{g}^2\right|\leq |a_{1}||T_{1}|+\dots+|a_{v}||T_{v}|. \]
By the definition of $W_{2,2}$, each $T_{\bullet}$ can be expressed as 
\[T_{\bullet}=\frac{1}{(T-t)^{q}}(\mathop{\ast}^{r}\nabla f)\ast(\nabla^{k_{1}}\mathrm{Rm}\ast\dots\ast\nabla^{k_{I}}\mathrm{Rm})\ast(\nabla^{\ell_{1}}A\ast\dots\ast\nabla^{\ell_{J}}A)\ast (\mathop{\ast}^{p}DF)\]
with some $I,J,q,r,p,k_{1},\dots,k_{I},\ell_{1},\dots,\ell_{J}\in\mathbb{N}$ satisfying 
\begin{align*}
\begin{aligned}
q+\frac{1}{2}r+\sum_{i=1}^{I}\biggl(1+\frac{1}{2}k_{i}\biggr)+\sum_{j=1}^{J}\biggl(\frac{1}{2}+\frac{1}{2}\ell_{j}\biggr)=2\quad\mathrm{and}\quad r\leq 2, 
\end{aligned}
\end{align*}
Here note that by Proposition~\ref{bdofallderiA} and the equation (\ref{orderchange1}) it follows that $(T-t)^{\frac{1}{2}+\frac{1}{2}\ell}|\nabla^{\ell}A|$ is bounded for all $\ell\geq 0$. 
Furthermore by the equation (\ref{orderchange2}) it is clear that $(T-t)^{1+\frac{1}{2}k}|\nabla^{k}\mathrm{Rm}|$ is bounded for all $k\geq 0$. 
Hence, for $T_{\bullet}$ above, we have 
\[(T-t)^2|T_{\bullet}|\leq C(T-t)^{\frac{1}{2}r}|\nabla f|^r\]
for some $C>0$. 
Furthermore, we have
\[|\nabla f|_{g}=\frac{1}{\sqrt{T-t}}|\nabla \tilde{f}|_{\tilde{g}}\leq \frac{1}{\sqrt{T-t}}\sqrt{\tilde{f}}. \]
Thus we have 
\[(T-t)^2|T_{\bullet}|\leq C{\tilde{f}}^{\frac{1}{2}r}. \]
Now each $T_{\bullet}$ is in $\mathcal{W}_{2,2}$, so $r\leq 2$. For $r=0,1,2$, it is clear that ${\tilde{f}}^{\frac{1}{2}r}\leq 1+\tilde{f}$. 
Thus we have proved that there exists a constant $C''''>0$ such that 
\[(T-t)^2\left|\left(\frac{\partial}{\partial t}-\Delta_{F^{*}g} \right)\left|H(F)+\nabla f^{\bot_{F}}\right|_{g}^2\right| \leq C''''(1+\tilde{f}). \]
\end{proof}
\section{convergence of submanifolds}\label{convofsubsec}
In this appendix, we give a definition of the convergence of immersion maps into a Riemannian manifolds and prove some propositions. 

Let $(N,g)$ be an $n$-dimensional Riemannian manifold and $E$ be a real vector bundle over $N$ with a metric $h$. 
Take a compatible connection $\nabla$ over $E$, that is, for all smooth sections $e,f\in\Gamma(N,E)$ and a vector field $X\in\mathfrak{X}(N)$ we have
\[X(h(e,f))=h(\nabla_{X}e,f)+h(e,\nabla_{X}f). \]
\begin{definition}
Let $p\in\mathbb{N}$. 
Let $K\subset N$ be a compact set and $\Omega\subset N$ be an open set satisfying $K\subset\Omega$. 
Let $\{\xi_{k}\}_{k=1}^{\infty}$ be a sequence of sections of $E$ defined on $\Omega$ and $\xi_{\infty}$ be a section of $E$ defined on $\Omega$. 
We say that $\xi_{k}$ converges in $C^{p}$ to $\xi_{\infty}$ uniformly on $K$ if for every $\epsilon>0$ there exists $k_{0}=k_{0}(\epsilon)$ such that 
for $k\geq k_{0}$, 
\[\sup_{0\leq\alpha\leq p}\sup_{x\in K}|\nabla^{\alpha}(\xi_{k}-\xi_{\infty})|_{g^{\alpha}\otimes h}<\epsilon. \]
Furthermore, we say $\xi_{k}$ converges in $C^{\infty}$ to $\xi_{\infty}$ uniformly on $K$ if $\xi_{k}$ converges in $C^{p}$ to $\xi_{\infty}$ uniformly on $K$ for every $p\in\mathbb{N}$. 
\end{definition}

Let $\{U_{k}\}_{k=1}^{\infty}$ be a sequence of open sets in $N$. 
We call $\{U_{k}\}_{k=1}^{\infty}$ an exhaustion of $N$ if $\overline{U}_{k}$ is compact and $\overline{U}_{k}\subset U_{k+1}$ for all $k$, and $\cup_{k=1}^{\infty}U_{k}=N$. 

\begin{definition}
Let $\{U_{k}\}_{k=1}^{\infty}$ be an exhaustion of $N$. 
Let $\{\xi_{k}\}_{k=1}^{\infty}$ be a sequence of locally defined sections of $E$ such that each $\xi_{k}$ is defined on $U_{k}$. 
Let $\xi_{\infty}$ be a section of $E$ defined on $N$. 
We say that $\xi_{k}$ converges in $C^{\infty}$ to $\xi_{\infty}$ uniformly on compact sets in $N$ if for any compact set $K\subset N$ 
there exists $k_{0}=k_{0}(K)$ such that $K\subset U_{k}$ for all $k\geq k_{0}$ and the sequence $\{\xi_{k}|_{U_{k_{0}}}\}_{k=k_{0}}^{\infty}$ 
converges in $C^{\infty}$ to $\xi_{\infty}|_{U_{k_{0}}}$ uniformly on $K$. 
\end{definition}

\begin{definition}\label{CheeGro}
A sequence $\{(N_{k},g_{k},x_{k})\}_{k=1}^{\infty}$ of complete pointed Riemannian manifolds converges to a complete pointed Riemannian manifold $(N_{\infty},g_{\infty},x_{\infty})$ 
if there exists 
\begin{enumerate}
\item[(1)] an exhaustion $\{U_{k}\}_{k=1}^{\infty}$ of $N_{\infty}$ with $x_{\infty}\in U_{k}$ and 
\item[(2)] a sequence of diffeomorphisms $\Psi_{k}:U_{k}\to V_{k}\subset N_{k}$ with $\Psi_{k}(x_{\infty})=x_{k}$ 
\end{enumerate}
such that $\Psi_{k}^{*}g_{k}$ converges in $C^{\infty}$ to $g_{\infty}$ uniformly on compact sets in $N_{\infty}$. 
\end{definition}

This notion of convergence is often referred as (smooth) Cheeger--Gromov convergence, $C^{\infty}$-convergence or geometric convergence. 
A basic fact of Cheeger--Gromov convergence is the following. 
For the proof, see \cite{MorganTian}. 

\begin{theorem}\label{CG}
Let $\{(N_{k},g_{k},x_{k})\}_{k=1}^{\infty}$ be a sequence of $n$-dimensional complete pointed Riemannian manifolds. 
Suppose that 
\begin{enumerate}
\item[(1)] for each integer $p\geq 0$, there exists a constant $0<C_{p}<\infty$ such that 
\[|\nabla^{p}\mathrm{Rm}(g_{k})|_{g_{k}}\leq C_{p}\quad\mathrm{for~all~}k\geq1\] 
\item[(2)] there exists a constant $0<\eta<\infty$ such that 
\[\mathrm{inj}(x_{k},g_{k})\geq\eta\quad\mathrm{for~all~}k\geq1\]
\end{enumerate}
where $\mathrm{Rm}(g_{k})$ is the Riemannian curvature tensor of $(N_{k},g_{k})$ and $\mathrm{inj}(x_{k},g_{k})$ is the injectivity radius at $x_{k}$ with respect to $g_{k}$. 
Then, there exist a complete pointed Riemannian manifold $(N_{\infty},g_{\infty},x_{\infty})$ and a subsequence $\{k_{\ell}\}_{\ell=1}^{\infty}$ such that 
the subsequence $\{(N_{k_{\ell}},g_{k_{\ell}},x_{k_{\ell}})\}_{\ell=1}^{\infty}$ converges to $(N_{\infty},g_{\infty},x_{\infty})$. 
\end{theorem}

To prove the convergence of submanifolds in a Riemannian manifold, we need the following estimate for the injectivity radius of a submanifold. 
This estimate is proved by combining Klingenberg's lemma and Hessian comparison theorem of the square of the distance function (cf. Theorem 2.1 in \cite{ChenYin}).
 \begin{theorem}\label{injofsub2}
 Let $(N,g)$ be an $n$-dimensional complete Riemannian manifold with 
 \[|\mathrm{Rm}(g)|\leq C\quad\mathrm{and}\quad\mathrm{inj}(N,g)\geq\eta\]
for some constants $C, \eta>0$.  Let $M$ be a compact manifold and $F:M\to N$ be an immersion map with 
\[|A(F)|\leq D\]
with some constant $D>0$. 
Then there exists a constant $\delta=\delta(C,\eta,D,n)>0$ such that 
\[\mathrm{inj}(M,F^{*}g)\geq\delta. \]
 \end{theorem}

The following remark partially overlaps with Remark~\ref{immRiem2}. 
\begin{remark}\label{immRiem}
In the remainder of this appendix, for a complete Riemannian manifold $(N,g)$, we assume that 
we have an isometrically embedding $\Theta:N\to (\mathbb{R}^{L},g_{\mathrm{st}})$ into some higher dimensional Euclidean space with 
\[|\nabla^{p}A(\Theta)|\leq \tilde{D}_{p}<\infty \]
for all $p\geq 0$. 
Under this assumption, one can see that $(N,g)$ must have the bounded geometry by Theorem~\ref{injofsub2} and Gauss equation (\ref{Gaussequation}) (and its iterated derivatives), 
and note that all compact Riemannian manifolds always satisfy this condition. 
For a map $F:U\to N$ from an open set $U$ in some Riemannian manifold $(M,h)$, by composing $\Theta$, 
we have a map $\Theta\circ F:U\to \mathbb{R}^{L}$, and furthermore we consider $\Theta\circ F$ as a section of the trivial $\mathbb{R}^{L}$-bundle over $U$ with a fiber metric $g_{\mathrm{st}}$. 
We write the standard flat connection of the trivial $\mathbb{R}^{L}$ bundle by $\bar{\nabla}$. 
Then $\bar{\nabla}(\Theta\circ F)$ is a section of $T^{*}M\otimes\mathbb{R}^{L}$ over $U$. 
The Levi--Civita connection on $TM$ and the connection $\bar{\nabla}$ on $\mathbb{R}^{L}$ induce the connection on $T^{*}M\otimes\mathbb{R}^{L}$, 
and we use the same symbol $\bar{\nabla}$ to denote this connection. 
\end{remark}

The following is the definition of the convergence of (pointed) immersions. 
It is the immersion map version of the Cheeger--Gromov convergence. 
\begin{definition}\label{convofimmmaps}
Let $(N,g)$ be a complete $n$-dimensional Riemannian manifold satisfying the assumption in Remark~\ref{immRiem} ($=$ Remark~\ref{immRiem2}). 
Assume that for each $k\geq 1$ we have an $m$-dimensional pointed manifold $(M_{k},x_{k})$ and an immersion map $F_{k}:M_{k}\to N$. 
Then we say that a sequence of immersion maps $\{F_{k}:M_{k}\to N\}_{k=1}^{\infty}$ converges to an immersion map $F_{\infty}:M_{\infty}\to N$ 
from an $m$-dimensional pointed manifold $(M_{\infty},x_{\infty})$ if there exist 
\begin{enumerate}
\item[(1)] an exhaustion $\{U_{k}\}_{k=1}^{\infty}$ of $M_{\infty}$ with $x_{\infty}\in U_{k}$ and 
\item[(2)] a sequence of diffeomorphisms $\Psi_{k}:U_{k}\to V_{k}\subset M_{k}$ with $\Psi_{k}(x_{\infty})=x_{k}$ such that 
the sequence of maps $F_{k}\circ \Psi_{k}:U_{k}\to N$ converges in $C^{\infty}$ to $F_{\infty}:M_{\infty}\to N$ uniformly on compact sets in $M_{\infty}$. 
\end{enumerate}
\end{definition}

It is clear that if $\{F_{k}:(M_{k},x_{k})\to N\}_{k=1}^{\infty}$ converges to $F_{\infty}:(M_{\infty},x_{\infty})\to N$ then 
$\{(M_{k},F_{k}^{*}g,x_{k})\}_{k=1}^{\infty}$ converges to $(M_{\infty},F_{\infty}^{*}g,x_{\infty})$ in the sense of Cheeger--Gromov convergence. 
To prove Theorem~\ref{convofimm}, we need the following Lemma. 
For the proof, see Corollary 4.6 in \cite{Chow}. 
\begin{lemma}\label{diffoftwoconn}
Let $M$ be a manifold, $K$ be a compact set in $M$ and $U$ be an open set in $M$ with $K\subset U$. 
Assume that we have Riemannian metrics $g$ and $\hat{g}$ on $U$, and these two satisfy 
\[|\nabla^{\ell}(g-\hat{g})|_{g}\leq \epsilon_{\ell}\quad\mathrm{on}\quad K\]
for some constants $\epsilon_{\ell}$ for all $\ell\geq 0$, where $\nabla$ is the Levi--Civita connection with respect to $g$. 
Let $E\to U$ be a vector bundle over $U$ with a fiber metric $h$ and a compatible connection $\bar\nabla$, and $T$ be a section of $E$ over $U$ 
which satisfies 
\[|\hat{\nabla}^{\ell}T|_{\hat{g}\otimes h}\leq \hat{C}_{\ell}\quad\mathrm{on}\quad K\]
for some constants $\hat{C}_{\ell}$ for all $\ell\geq 0$, 
where $\hat{\nabla}$ is the connection induced by the Levi--Civita connection with respect to $\hat{g}$ and the connection $\bar\nabla$. 
Then for each $\ell\geq 0$ there exists a constant $C_{\ell}$ which depends only on $\{\epsilon_{p}\}_{p=0}^{\ell}$ and $\{\hat{C}_{p}\}_{p=0}^{\ell}$ such that 
\[|\nabla^{\ell}T|_{g \otimes h}\leq C_{\ell}\quad\mathrm{on}\quad K, \]
where $\nabla$ is the connection induced by the Levi--Civita connection with respect to $g$ and the connection $\bar\nabla$. 
\end{lemma}

\begin{theorem}\label{convofimm}
Let $(N,g)$ be a complete $n$-dimensional Riemannian manifold satisfying the assumption in Remark~{\rm\ref{immRiem}} {\rm(}$=$ Remark~{\rm\ref{immRiem2}}{\rm)}. 
Let $\{(M_{k},x_{k})\}_{k=1}^{\infty}$ be a sequence of compact pointed $m$-dimensional manifolds and 
$\{F_{k}:M_{k}\to N\}_{k=1}^{\infty}$ be a sequence of immersions with 
\[|\nabla^{p}A(F_{k})|\leq D_{p}<\infty\]
for all $p\geq 0$. 
In the case that $(N,g)$ is non-compact, we further assume that $\{F_{k}(x_{k})\}_{k=1}^{\infty}$ is a bounded sequence in $N$. 
Then, there exist a pointed manifold $(M_{\infty},x_{\infty})$, an immersion $F_{\infty}:M_{\infty}\to N$ and a subsequence $\{k_{\ell}\}_{\ell=1}^{\infty}$ such that 
$\{F_{k_{\ell}}:M_{k_{\ell}}\to N\}_{\ell=1}^{\infty}$ converges to $F_{\infty}:M_{\infty}\to N$ and $(M_{\infty},F_{\infty}^{*}g)$ is a complete Riemannian manifold. 
\end{theorem}
\begin{proof}
First of all, we prove that the sequence $\{(M_{k},F_{k}^{*}g,x_{k})\}_{k=1}^{\infty}$ sub-converges to some complete pointed Riemannian manifold $(M_{\infty},h_{\infty},x_{\infty})$. 
By Remark~\ref{immRiem} ($=$ Remark~\ref{immRiem2}), $(N,g)$ has bounded geometry, that is, 
\[|\nabla^{p}\mathrm{Rm}(g)|\leq C_{p}<\infty\quad\mathrm{and}\quad \mathrm{inj}(N,g)\geq\eta>0\]
for some positive constants $C_{p}$ and $\eta$. 
Then, by Theorem~\ref{injofsub2}, there exists a constant $\delta=\delta(C_{0},\eta,D_{0},n)>0$ such that 
\[\mathrm{inj}(M_{k},F_{k}^{*}g)\geq\delta>0. \]
We denote the Riemannian curvature tensor of $(M_{k}, F^{*}_{k}g)$ by $\mathrm{Rm}(F^{*}_{k}g)$. 
Then, by Gauss equation (\ref{Gaussequation}) and its iterated derivatives, we can see that there exist constants $\tilde{C}_{p}>0$ such that 
\[|\nabla^{p}\mathrm{Rm}(F^{*}_{k}g)|\leq \tilde{C}_{p}<\infty,\]
for all $p\geq 0$, where each $\tilde{C}_{p}$ does not depend on $k$. 
Then, by Theorem~\ref{CG}, $\{(M_{k},F_{k}^{*}g,x_{k})\}_{k=1}^{\infty}$ sub-converges to 
some complete pointed Riemannian manifold $(M_{\infty},h_{\infty},x_{\infty})$. 
Note that, in the following in this proof, we continue to use the letter $k$ for indices of subsequences. 
Since $(M_{k},F_{k}^{*}g,x_{k})$ converge to $(M_{\infty},h_{\infty},x_{\infty})$, 
there exist an exhaustion $U_{k}$ of $M_{\infty}$ with $x_{\infty}\in U_{k}$ and 
a sequence of diffeomorphisms $\Psi_{k}:U_{k}\to \Psi_{k}(U_{k})\subset M_{k}$ with $\Psi(x_{\infty})=x_{k}$. 

Next, we prove that the sequence of smooth maps $F_{k}\circ \Psi_{k}:U_{k}\to N$ sub-converge to some smooth map $F_{\infty}:M_{\infty}\to N$ uniformly on compact sets in $M_{\infty}$. 
We denote $\Theta\circ F_{k}\circ \Psi_{k}:U_{k}\to\mathbb{R}^{L}$ by $\bar{F}_{k}$ for short. 
We will use the standard diagonal argument to construct a map $F_{\infty}:M_{\infty}\to N$. 
Take a sequence of radii $R_{1}<R_{2}<\cdots\to\infty$, and consider balls $B_{i}:=B_{h_{\infty}}(x_{\infty},R_{i})\subset M_{\infty}$. 

First of all, we work on $B_{1}$. Since $U_{k}$ is an exhaustion, there exists $k_{1}$ such that $\overline{B_{1}}\subset U_{k}$ for all $k\geq k_{1}$. 
Hence we have a sequence of $C^{\infty}$-maps $\bar{F}_{k}=\Theta\circ F_{k}\circ \Psi_{k}:(U_{k}\supset)B_{1}\to\mathbb{R}^{L}$ restricted on $B_{1}$ for all $k\geq k_{1}$. 

\vspace{3mm}
\noindent 
\underline{(0): $C^{0}$-estimate. } First, we derive a $C^{0}$-bound for $\bar{F}_{k}$. 
If $N$ is compact, then the image $\Theta(N)$ is a compact set in $\mathbb{R}^{L}$ and contained in some ball 
\[B_{g_{\mathrm{st}}}(0,\hat{C}_{0})=\{\,y\in\mathbb{R}^{L}\mid |y|_{g_{\mathrm{st}}}<\hat{C}_{0} \,\}\] 
with radius $\hat{C}_{0}$. 
Since each image $\bar{F}_{k}(B_{1})$ is contained in $\Theta(N)$, we have 
\[|\bar{F}_{k}|_{g_{\mathrm{st}}}\leq \hat{C}_{0}\quad\mathrm{on}\quad\overline{B_{1}}. \]
It is clear that the constant $\hat{C}_{0}$ does not depend on $k$. 
If $N$ is non-compact, we need some additional argument to get a $C^{0}$-bound. 
Since $|\bar{F}_{k}^{*}g_{\mathrm{st}}-h_{\infty}|_{h_{\infty}}=|\Psi_{k}^{*}(F_{k}^{*}g)-h_{\infty}|_{h_{\infty}}\to 0$ uniformly on $\overline{B_{1}}$, 
for a given $\epsilon>0$ there exists $k'_{1}(\geq k_{1})$ such that on $\overline{B_{1}}$ 
\[|\bar{F}_{k}^{*}g_{\mathrm{st}}-h_{\infty}|_{h_{\infty}}<\epsilon \quad\mathrm{for}\quad k\geq k'_{1}, \] 
and this implies that 
\[|\bar{F}_{k}(x_{\infty})-\bar{F}_{k}(x)|_{g_{\mathrm{st}}}\leq \sqrt{1+\epsilon}\, d_{h_{\infty}}(x_{\infty},x)\leq\sqrt{1+\epsilon}R_{1}\]
for all $x\in B_{1}$ and $k\geq k'_{1}$. 
Furthermore, by the assumption for the non-compact case, $\{F_{k}(x_{k})\}_{k=1}^{\infty}$ is a bounded sequence in $N$. 
Hence $\bar{F}_{k}(x_{\infty})=(\Theta\circ F_{k})(x_{k})$ is also a bounded sequence in $\mathbb{R}^{L}$, that is, there exists a constant $\hat{C}'_{0}$ such that 
$|\bar{F}_{k}(x_{\infty})|_{g_{\mathrm{st}}}\leq \hat{C}_{0}'$. 
Hence we have
\[|\bar{F}_{k}|_{g_{\mathrm{st}}}\leq \hat{C}'_{0}+\sqrt{1+\epsilon}R_{1}=:\hat{C}_{0}\]
for $k\geq k'_{1}$. It is clear that $\hat{C}_{0}$ does not depend on $k$. 
Hence we get a $C^{0}$-bound. 

\vspace{3mm}
\noindent 
\underline{(1): $C^{1}$-estimate. }
Next, we consider a $C^{1}$-bound for $\bar{F}_{k}$. 
One can easily see that $\nabla_{g_{\mathrm{st}}}\bar{F}_{k}=D\bar{F}_{k}$. 
Since $\bar{F}_{k}:(B_{1},\bar{F}_{k}^{*}g_{\mathrm{st}})\to (\mathbb{R}^{L},g_{\mathrm{st}})$ is an isometric immersion, we have a $C^{1}$-bound 
\[|\nabla_{g_{\mathrm{st}}}\bar{F}_{k}|_{\bar{F}_{k}^{*}g_{\mathrm{st}}\otimes g_{\mathrm{st}}}=|D\bar{F}_{k}|_{\bar{F}_{k}^{*}g_{\mathrm{st}}\otimes g_{\mathrm{st}}}=\sqrt{m}=:\hat{C}_{1}. \]

\vspace{3mm}
\noindent 
\underline{(2): $C^{2}$-estimate. }
Next, we derive a $C^{2}$-bound for $\bar{F}_{k}$. 
Let $\hat{\nabla}$ be the connection on $(\otimes^{p}T^{*}M)\otimes \mathbb{R}^{L}$ ($p\geq 0$) over $B_{1}$ induced by the metric $\bar{F}_{k}^{*}g_{\mathrm{st}}$ and $g_{\mathrm{st}}$. 
Note that $\hat{\nabla}=\nabla_{g_{\mathrm{st}}}$ for $p=0$. 
Since $\hat{\nabla}\bar{F}_{k}=D\bar{F}_{k}$, we have 
\[\hat{\nabla}^2\bar{F}_{k}=A(\bar{F}_{k}), \]
the second fundamental form of the isometric immersion $\bar{F}_{k}=\Theta\circ F_{k}\circ \Psi_{k}:(B_{1},\bar{F}_{k}^{*}g_{\mathrm{st}})\to(\mathbb{R}^{L},g_{\mathrm{st}})$. 
Hence, by using the composition rule for the second fundamental forms of immersions,  we have 
\begin{align*}
\hat{\nabla}^2\bar{F}_{k}(X,Y)=&A(\bar{F}_{k})(X,Y)\\
=&A(\Theta)((F_{k}\circ \Psi_{k})_{*}X,(F_{k}\circ \Psi_{k})_{*}Y)+\Theta_{*}(A(F_{k})(\Psi_{k*}X,\Psi_{k*}Y))
\end{align*}
for any tangent vectors $X$ and $Y$ on $M$. 
By using the notion of $\ast$-product, this identity is written as 
\begin{align}\label{AbarF}
\hat{\nabla}^2\bar{F}_{k}=A(\Theta)\ast(\mathop{\ast}^{2}D(F_{k}\circ \Psi_{k}))+A(F_{k})\ast D\Theta \ast (\mathop{\ast}^{2}D\Psi_{k}). 
\end{align}
Since $|D(F_{k}\circ \Psi_{k})|_{\bar{F}_{k}^{*}g_{\mathrm{st}}\otimes g}=|D\Psi_{k}|_{\bar{F}_{k}^{*}g_{\mathrm{st}}\otimes F_{k}^{*}g}=\sqrt{m}$ and $|D\Theta|_{g\otimes g_{\mathrm{st}}}=\sqrt{n}$, 
we have 
\[|\hat{\nabla}^2\bar{F}_{k}|_{\bar{F}_{k}^{*}g_{\mathrm{st}}\otimes g_{\mathrm{st}}}\leq \hat{C}'_{2}|A(\Theta)|_{g\otimes g_{\mathrm{st}}}+\hat{C}''_{2}|A(F_{k})|_{F_{k}^{*}g\otimes g}\]
for some constants $\hat{C}'_{2}$ and $\hat{C}''_{2}$ which do not depend on $k$. 
Furthermore, by the assumptions, we have $|A(\Theta)|_{g\otimes g_{\mathrm{st}}}\leq \tilde{D}_{0}$ and $|A(F_{k})|_{F_{k}^{*}g\otimes g}\leq D_{0}$. 
Hence we have a $C^{2}$-bound 
\[|\hat{\nabla}^2\bar{F}_{k}|_{\bar{F}_{k}^{*}g_{\mathrm{st}}\otimes g_{\mathrm{st}}}\leq \hat{C}'_{2}\tilde{D}_{0}+\hat{C}''_{2}D_{0}=:\hat{C}_{2}. \]
It is clear that $\hat{C}_{2}$ does not depend on $k$. 

\vspace{3mm}
\noindent 
\underline{($p$): $C^{p}$-estimate. }
By differentiating (\ref{AbarF}), we can get a $C^{p}$-bound. 
We only observe a $C^{3}$-bound. 
Note that for any tangent vectors $X$ and $Y$ on $M$ we have
\[(\nabla_{\bar{F}_{k}^{*}g_{\mathrm{st}}\otimes g}D(F_{k}\circ \Psi_{k}))(X,Y)=A(F_{k}\circ \Psi_{k})(X,Y)=A(F_{k})(\Psi_{k*}X,\Psi_{k*}Y). \]
By using the notion of $\ast$-product, this identity is written as 
\begin{align*}
\nabla_{\bar{F}_{k}^{*}g_{\mathrm{st}}\otimes g}D(F_{k}\circ \Psi_{k})=A(F_{k})\ast (\mathop{\ast}^{2}D\Psi_{k}). 
\end{align*}
Furthermore, note that $\nabla_{g\otimes g_{\mathrm{st}}}D\Theta=A(\Theta)$ and $\nabla_{\bar{F}_{k}^{*}g_{\mathrm{st}}\otimes F_{k}^{*}g}D\Psi_{k}=0$. 
Hence we have 
\begin{align*}
\hat{\nabla}^3\bar{F}_{k}=&\nabla_{g\otimes g_{\mathrm{st}}}A(\Theta)\ast(\mathop{\ast}^{2}D(F_{k}\circ \Psi_{k}))\\
&+2A(\Theta)\ast D(F_{k}\circ \Psi_{k})\ast A(F_{k})\ast (\mathop{\ast}^{2}D\Psi_{k})  \\
&+\nabla_{F_{k}^{*}g\otimes g}A(F_{k})\ast D\Theta \ast (\mathop{\ast}^{2}D\Psi_{k})\\
&+A(F_{k})\ast A(\Theta) \ast (\mathop{\ast}^{2}D\Psi_{k}). 
\end{align*}
By the assumptions, 
norms of all tensors appeared in the above inequality is bounded. 
Hence we have a $C^{3}$-bound 
\[|\hat{\nabla}^3\bar{F}_{k}|_{\bar{F}_{k}^{*}g_{\mathrm{st}}\otimes g_{\mathrm{st}}}\leq \hat{C}_{3}\]
for some constant $\hat{C}_{3}$ which does not depend on $k$. 
For higher derivatives, one can prove that there exists a constant $\hat{C}_{p}>0$ which does not depend on $k$ such that 
\[|\hat{\nabla}^p\bar{F}_{k}|_{\bar{F}_{k}^{*}g_{\mathrm{st}}\otimes g_{\mathrm{st}}}\leq \hat{C}_{p}, \]
by induction. 

On the above argument, we have proved that there exist constants $\hat{C}_{p}$ ($p\geq 0$) which do not depend on $k$ such that 
$|\hat{\nabla}^p\bar{F}_{k}|_{\bar{F}_{k}^{*}g_{\mathrm{st}}\otimes g_{\mathrm{st}}}\leq \hat{C}_{p}. $
Hence by Lemma~\ref{diffoftwoconn} we can prove that there exist constants $C_{p}$ ($p\geq 0$) which do not depend on $k$ such that 
\[|\nabla^p\bar{F}_{k}|_{h_{\infty}\otimes g_{\mathrm{st}}}\leq C_{p}. \]
Hence, by The Arzel\`a--Ascoli Theorem, there exists a smooth map $\bar{F}_{1,\infty}:\overline{B}_{1}\to\mathbb{R}^{L}$ 
and $\bar{F}_{k}$ sub-converges to $\bar{F}_{1,\infty}$ in $C^{\infty}$ on $\overline{B}_{1}$. 
Since all images $\bar{F}_{k}(\overline{B}_{1})$ are contained in $\Theta(N)$, the image $\bar{F}_{1,\infty}(\overline{B}_{1})$ is also contained in $\Theta(N)$. 
Furthermore $\bar{F}_{1,\infty}:\overline{B}_{1}\to\mathbb{R}^{L}$ has the property that 
\[\bar{F}_{1,\infty}^{*}g_{\mathrm{st}}=h_{\infty}, \]
since $|\bar{F}_{1,\infty}^{*}g_{\mathrm{st}}-h_{\infty}|_{h_{\infty}}\leq |\bar{F}_{1,\infty}^{*}g_{\mathrm{st}}-\bar{F}_{k}^{*}g_{\mathrm{st}}|_{h_{\infty}}
+|\bar{F}_{k}^{*}g_{\mathrm{st}}-h_{\infty}|_{h_{\infty}}$ and the right hand side converges to $0$ as $k\to \infty$ on $\overline{B}_{1}$. 
Thus, especially, $\bar{F}_{1,\infty}:\overline{B}_{1}\to\mathbb{R}^{L}$ is an immersion map. 

Next, for the subsequence of $\bar{F}_{k}$ which converges to $\bar{F}_{1,\infty}$, we work on $B_{2}$. 
Then all the above argument also work on $B_{2}$ and we can show that 
there exists a smooth immersion map $\bar{F}_{2,\infty}:\overline{B}_{2}\to\Theta(N)\subset\mathbb{R}^{L}$ with 
$\bar{F}_{2,\infty}^{*}g_{\mathrm{st}}=h_{\infty}$ and $\bar{F}_{2,\infty}=\bar{F}_{1,\infty}$ on $\overline{B}_{1}$ 
and $\bar{F}_{k}$ sub-converges to $\bar{F}_{2,\infty}$ in $C^{\infty}$ on $\overline{B}_{2}$. 
By iterating this construction and the diagonal argument, finally we get a smooth immersion map $\bar{F}_{\infty}:M_{\infty}\to\Theta(N)\subset\mathbb{R}^{L}$ 
with $\bar{F}_{\infty}^{*}g_{\mathrm{st}}=h_{\infty}$ and $\bar{F}_{k}$ sub-converges to $\bar{F}_{\infty}$ uniformly on compact sets in $M_{\infty}$ in $C^{\infty}$, 
and the map defined by $F_{\infty}:=\Theta^{-1}\circ \bar{F}_{\infty}:M_{\infty}\to N$ is the requiring one satisfying the properties in the statement. 
\end{proof}

\end{document}